\documentclass[bj,numbers]{imsart}

\usepackage[utf8]{inputenc}
\usepackage{amsmath}
\usepackage{amsfonts}
\usepackage{amssymb}
\usepackage{amsthm}
\usepackage{microtype}
\usepackage{dsfont}
\usepackage{color}
\usepackage{fontawesome5}
\usepackage[capitalize]{cleveref}
\usepackage{csquotes}

\startlocaldefs
\MakeOuterQuote{"}

\usepackage{tikz-cd} 

\newtheorem{theorem}{Theorem}[section]
\newtheorem{proposition}[theorem]{Proposition}
\newtheorem{lemma}[theorem]{Lemma}
\newtheorem{corollary}[theorem]{Corollary}

\newtheorem{definition}[theorem]{Definition}

\theoremstyle{definition}
\newtheorem{remark}[theorem]{Remark}
\newtheorem{example}[theorem]{Example}

\definecolor{auburn}{rgb}{0.43, 0.21, 0.1}
\definecolor{britishracinggreen}{rgb}{0.0, 0.26, 0.15}
\definecolor{burntumber}{rgb}{0.54, 0.2, 0.14}
\definecolor{carmine}{rgb}{0.59, 0.0, 0.09}
\definecolor{aurometalsaurus}{rgb}{0.43, 0.5, 0.5}
\definecolor{gray}{rgb}{0.4, 0.4, 0.4 }
\definecolor{darkteal}{rgb}{0, 0.35, 0.35}
\definecolor{darkpurple}{rgb}{0.4, 0, 0.23}
\definecolor{orange}{rgb}{1,0.5,0}

\newcommand{\am}[1]{\textcolor{darkpurple}{\small\sffamily [Alexandre: {#1}]}}
\newcommand{\as}[1]{\textcolor{carmine}{\small\sffamily [Anne: {#1}]}}
\newcommand{\js}[1]{\textcolor{darkteal}{\small\sffamily\upshape [Johan: {#1}]}}

\newcommand{\mdf}[1]{\textcolor{red}{#1}}

\newcommand{\mbb}{\mathbb}
\newcommand{\mrm}{\mathrm}
\newcommand{\mcal}{\mathcal}

\renewcommand{\le}{\leqslant}
\renewcommand{\leq}{\leqslant}
\renewcommand{\ge}{\geqslant}
\renewcommand{\geq}{\geqslant}

\newcommand{\abr}[1]{\left|#1\right|}
\newcommand{\cbr}[1]{\left\{#1\right\}}
\newcommand{\nbr}[1]{\left\|#1\right\|}
\newcommand{\brbr}[1]{\bigl(#1\bigr)}
\newcommand{\rbr}[1]{\left(#1\right)}
\newcommand{\sbr}[1]{\left[#1\right]}

\newcommand{\sauf}{\setminus}
\renewcommand{\emptyset}{\varnothing}

\newcommand{\graph}{\operatorname{gph}}
\newcommand{\gph}{\graph}

\newcommand{\restr}{\llcorner}

\newcommand{\supp}{\operatorname{spt}}
\newcommand{\spt}{\supp}
\newcommand{\epi}{\operatorname{epi}}
\newcommand{\cl}{\operatorname{cl}}
\newcommand{\ri}{\operatorname{ri}}
\newcommand{\intr}{\operatorname{int}}
\newcommand{\bnd}{\operatorname{bnd}}
\newcommand{\dom}{\operatorname{dom}}
\newcommand{\rge}{\operatorname{rge}}
\newcommand{\proj}{\operatorname{proj}}

\newcommand{\conv}{\operatorname{conv}}
\newcommand{\ball}{\mathds{B}}

\newcommand{\Prob}{\mathcal{P}}
\newcommand{\Fell}{\mathcal{F}}
\newcommand{\borel}{\mathcal{B}}

\newcommand{\NN}{\mathbb{N}}
\newcommand{\Id}{\operatorname{Id}}

\newcommand{\RV}{\operatorname{RV}}
\newcommand{\SL}{\operatorname{SL}}

\newcommand{\Mz}{\mathrm{M}_{0}}
\newcommand{\Mzp}{\Mz}
\newcommand{\Mzpp}{\mathrm{M}_{0}^{+}}
\newcommand{\tMzp}{\tilde{\mathrm{M}}_{0}}
\newcommand{\Mzpcm}{\mathrm{M}_{0,cm}}

\newcommand{\pos}{\operatorname{pos}}

\newcommand{\res}{\operatorname{res}}

\newcommand{\nat}{\mathbb{N}}
\newcommand{\reals}{\mathbb{R}}
\newcommand{\R}{\reals}
\newcommand{\Rd}{\R^d}

\newcommand{\Rddmz}{\R^{2d}\setminus\cbr{0}}
\newcommand{\RdRd}{\Rd\times\Rd}
\newcommand{\Rdmz}{\Rd_{\smallsetminus 0}}
\newcommand{\Rmz}{\R_{\smallsetminus 0}}
\newcommand{\RdRdmz}{(\Rd\times\Rd)\setminus\{0\}}

\newcommand{\coup}{\Pi}
\newcommand{\cmcoup}{\Pi_{\mathrm{cm}}}

\newcommand{\zcoup}{\Gamma_{0}}
\newcommand{\zmcoup}{\Gamma_{0,\mathrm{m}}}
\newcommand{\zpmcoup}{\Gamma_{0,\mathrm{m}}^{\mathrm{p}}}
\newcommand{\zcmcoup}{\Gamma_{0,\mathrm{cm}}}
\newcommand{\zpcoup}{\Gamma_{0,\mathrm{cm}}^{\mathrm{p}}}
\newcommand{\zprcoup}{\Gamma_{0}^{\mathrm{p}}}

\newcommand{\Cbzp}{\mathrm{C}_{b,0}^{+}}
\newcommand{\Cbz}{\mathcal{C}_0}

\newcommand{\Fellmm}{\Fell_{\mathrm{mm}}}
\newcommand{\Fellcm}{\Fell_{\mathrm{cm}}}
\newcommand{\Fellmcm}{\Fell_{\mathrm{mcm}}}

\newcommand{\Picm}{\Pi_{cm}}

\newcommand{\eps}{\varepsilon}
\newcommand{\diff}{\mathrm{d}}
\newcommand{\point}{\,\cdot\,}
\newcommand{\diag}{\operatorname{diag}}

\newcommand{\vto}{\raisebox{-0.5pt}{\,\scriptsize$\stackrel{\raisebox{-0.5pt}{\mbox{\tiny $0$}}}{\longrightarrow}$}\,}
\newcommand{\Mto}{\vto}
\newcommand{\Fto}{\raisebox{-0.5pt}{\,\scriptsize$\stackrel{\raisebox{-0.5pt}{\mbox{\tiny $\mathcal{F}$}}}{\longrightarrow}$}\,}

\newcommand{\inpr}[1]{\langle{#1}\rangle}
\newcommand{\norm}[1]{\left\lvert{#1}\right\rvert}

\newcommand{\bpsi}{\bar{\psi}}

\newcommand{\pr}{\operatorname{\mathsf{P}}}

\newcommand{\Szbor}{\mcal{S}_0}

\newcommand{\nset}{\mathbb{N}}
\newcommand{\rset}{\mathbb{R}}

\newcommand{\Mp}{\mathrm{M}^{+}}

\newcommand{\E}{\mathbb{E}}

\newcommand{\dist}{\operatorname{\mrm{d}}}
\endlocaldefs

\begin{document}

\begin{frontmatter}

\title{Zero-couplings of infinite measures with cyclically monotone support and multivariate regular variation}
\runtitle{Zero-couplings of infinite measures}

\begin{aug}
    \author[ar]{\inits{A.}\fnms{Alexandre}~\snm{Reber}\ead[label=e1]{alexandre.reber@math.au.dk}}
    \author[as]{\inits{A.}\fnms{Anne}~\snm{Sabourin}\ead[label=e2]{anne.sabourin@u-paris.fr}}
    \author[js]{\inits{J.}\fnms{Johan}~\snm{Segers}\ead[label=e3]{jjjsegers@kuleuven.be}}
    \author[cv]{\inits{C.}\fnms{Cees}~\snm{de Valk}\ead[label=e4]{cees.de.valk@knmi.nl}}

    \address[ar]{Aarhus University, Department of Mathematics, Ny Munkegade 118, 8000 Aarhus C, Denmark\printead[presep={,\ }]{e1}}

    \address[as]{Université Paris Cité, Université Paris Saclay, ENS Paris Saclay, CNRS, SSA, INSERM, Centre Borelli, F-75006, Paris, France\printead[presep={,\ }]{e2}}

    \address[js]{KU Leuven, 
        Department of Mathematics, 
        Celestijnenlaan 200B, 
        3001 Heverlee, 
        Belgium, and UCLouvain, LIDAM/ISBA\printead[presep={,\ }]{e3}}

    \address[cv]{Koninklijk Nederlands Meteorologisch Instituut, 
        PO Box 201, 
        NL-3730 AE De Bilt, 
        Netherlands\printead[presep={,\ }]{e4}}

        
\end{aug}

\begin{abstract}
We study cyclically monotone transport plans between measures in $\mathrm{M}_0(\mathbb{R}^d)$, the class of Borel measures on $\mathbb{R}^d \setminus \{0\}$ that are finite on sets bounded away from the origin but may have infinite total mass. We avoid moment assumptions and allow the transport cost to be infinite. This framework naturally arises for exponent measures in multivariate regular variation and includes other examples such as Lévy measures.

We introduce the notion of a zero-coupling and establish existence of cyclically monotone zero-couplings for arbitrary pairs of measures in $\mathrm{M}_0(\mathbb{R}^d)$. Under a Hausdorff-dimension condition on the first measure and when at least one of the two measures has infinite mass, we prove uniqueness of the cyclically monotone zero-coupling, yielding an analogue of the Brenier--McCann theorem in this infinite-measure setting. We further derive a representation of such couplings through gradients of closed convex functions and identify conditions under which the zero-coupling is proper in the sense that the second measure is equal to the restriction to the punctured space of the push-forward of the first measure by a cyclically monotone transport map.

Finally, we apply these results to regularly varying probability measures. We show that a cyclically monotone coupling between two such distributions admits a tail limit that coincides with the unique proper cyclically monotone zero-coupling between the corresponding exponent measures.
\end{abstract}

\begin{keyword}
	\kwd{Cyclical monotonicity}
	\kwd{Fell topology}
    \kwd{Multivariate extremes}
	\kwd{Optimal transport}
	\kwd{Regular variation}
\end{keyword}

\end{frontmatter}

\section{Introduction}
\label{sec:introduction}



\subsection{General motivation}

We study measure transportation by cyclically monotone mappings between nonnegative Borel measures on Euclidean space with possibly infinite total mass that are finite on complements of neighborhoods of the origin. Such measures arise in the definition of multivariate regular variation; Lévy measures provide another example. Especially in the first setting, which appears naturally in extreme value theory, the usual moment assumptions that are key to classical results in optimal transport theory cannot be taken for granted.


Our goal is to strengthen the theoretical foundation for the use of cyclically monotone mappings in statistics, where such maps are of intrinsic interest---even when they cannot be interpreted as optimal transport plans due to infinite transport costs. A prominent example is the theory of center-outward multivariate quantiles \citep{Beirlant2020,Chernozhukov2017,Hallin2022,hallin2021}, which provides a notion of rank statistics for multivariate observations.  Our results are intended to lay the groundwork for extensions to multivariate extreme quantile contours. A difference to the existing literature is that a natural reference distribution in an extreme value context is not spherical uniform,  
but would typically factorize into a product of a reference angular distribution and a heavy-tailed radial component. In another direction, monotone maps and optimal transport play a central role in generative modeling \citep{lipman2022flow,albergo2022building}. We anticipate that our results may help extend these methods to the generation of synthetic extreme samples.


Our focus is on transport maps and an appropriate notion of coupling between measures in the set $\Mzp(\Rd)$, the set of 
Borel measures 
on $\Rdmz = \Rd \sauf \cbr{0}$ that are finite on sets bounded away from the origin but may have infinite total mass \citep{hult2006regular,lindskog2014regular}. The space $\Mzp(\Rd)$ is equipped with a notion of convergence in duality with the set $\Cbz(\Rd)$ of real-valued bounded continuous functions whose support is bounded away from the origin. A typical assumption in multivariate extreme value theory 
is that observations are drawn from a probability distribution $Q$ on $\Rd$ that is regularly varying: 
there exists a function $b$ on the positive reals and a limit measure $\nu \in \Mzp(\Rd)$ such that
\begin{equation}
    \label{eq:def_regular_variation}
    t \, Q(b(t)\point) \to \nu, \qquad t \to \infty,
\end{equation}
in $\Mzp(\Rd)$. In practice, $\nu$ can be used to approximate the distribution $Q$ in regions of $\Rd$ far from the origin, where few or no observations are available.


\subsection{Related work, zero-couplings, proper couplings}

Our interest lies in transport maps between measures $\mu, \nu \in \Mz(\Rd)$ with arbitrarily heavy tails. In such cases, neither the Monge problem for the $L^2$ cost,
\begin{equation}
    \label{eq:Monge}
    \min_{T: T_\#\mu = \nu} \int \|x-T(x)\|^2 \, \diff\mu(x),
\end{equation}
nor the Kantorovich problem,
\begin{equation}
    \label{eq:Kantorovich}
    \min_{\gamma \in \Pi(\mu,\nu)} \int \|x-y\|^2 \, \diff\gamma(x,y),
\end{equation}
may admit solutions with finite objective, since $\mu$ and $\nu$ need not have finite second moments themselves. Here, $T_\#\mu$ is the push-forward measure of $\mu$ induced by $T$, see Section~\ref{sec:notation} below, while $\Pi(\mu,\nu)$ denotes the set of coupling measures between $\mu$ and $\nu$, i.e., measures on $\RdRd$ with margins $\mu$ and $\nu$. As we will explain below, we will actually use a slightly weakened notion of couplings, called zero-couplings, to deal with the possible singularity at the origin.

The special case where $\mu, \nu$ are probability measures with finite second moments falls within the classical framework of optimal transport, for which there are many excellent and extensive textbook introductions \citep[e.g.][]{ambrosio2021lectures, santambrogio2015optimal,villani2009optimal}. If $\mu$ places no mass on sets of Hausdorff dimension less than $d-1$, Brenier's Theorem \cite{brenier1987decomposition,brenier1991polar} (attributed also to \cite{ruschendorf1990characterization}; see \cite{villani2009optimal}, Chapter~9, and the bibliographic comments on page~212 on many other seminal papers) guarantees that the Kantorovich problem admits a unique solution $\pi$, which can be written as $\pi = (\Id \times T)_\#\mu$, where the map $T$ is defined $\mu$-almost everywhere as $T = \nabla \psi$ for some convex function $\psi$ and which, up to a $\mu$-null set, also solves the Monge problem \eqref{eq:Monge}. A key property of $T$---or, equivalently, of the support $\spt \pi$ of $\pi$---is cyclical monotonicity, the definition of which is recalled in \cref{rel:cyclical_mononicity0} below. Conversely, if a coupling measure $\pi$ has cyclically monotone support, then it is necessarily optimal in the sense that it solves the Kantorovich problem~\eqref{eq:Kantorovich}, see \citep{Beiglböck2015,de202360,pratelli2008sufficiency,schachermayer2009characterization}. This is why we focus on cyclically monotone transport plans (maps and couplings) even when the expected cost may be infinite. For simplicity, we focus on the notion of cyclical monotonicity with respect to the quadratic cost function $c(x,y) = \|x-y\|^2$, deferring the study of general cost functions $c(x,y)$ to future research.


Brenier's result has been extended in two distinct directions:

(i) Positive and potentially infinite Borel measures are considered in several articles, though typically under moment assumptions that allow for the study of the Kantorovich problem. For instance, \cite{GuillenMou2019} introduces a transportation metric between Lévy measures in the context of viscosity equations, a metric that \cite{CatalanoLavenant2021} later uses to construct an index of dependence between random measures. In \cite{figalli2010new}, solutions of the Kantorovich problem are linked to gradient flows that solve the heat equation with Dirichlet boundary conditions. A common feature of these references is the presence of a specific subset of the domain---the origin in \cite{GuillenMou2019,CatalanoLavenant2021} or the boundary in \cite{figalli2010new}---which acts as a potential "sink" or "reservoir" of mass in the transport problem. These works also introduce extended notions of coupling, called "admissible couplings" in \cite{figalli2010new,GuillenMou2019} and "extended couplings" in \cite{CatalanoLavenant2021}, in which the first and second marginals are required to coincide with $\mu$ and $\nu$ only on the complement of the reservoir or sink. In our paper, we adopt the terminology "zero-couplings" to describe such couplings, reflecting the special role of the origin in $\Mz(\Rd)$.

(ii)
In \cite{McCann1995}, the condition that second moments exist is relaxed, and existence and uniqueness of a cyclically monotone coupling and the associated transport map are established under assumptions similar to those in \cite{brenier1987decomposition}. Let $\mathcal{P}(\Rd)$ denote the set of probability measures on $\Rd$ and $\cmcoup(P, Q)$ the set of coupling measures of $P, Q \in \mathcal{P}(\Rd)$ with cyclically monotone support. The subdifferential of a convex function $\psi : \Rd \to \reals \cup \{\infty\}$ is denoted by $\partial \psi$ (see Section~\ref{section:CVX} for preliminaries).

\begin{theorem}[McCann, \cite{McCann1995}]
   \label{thm:McCann}  
    Let $P, Q \in \mathcal{P}(\Rd)$, with $P$ vanishing on sets of Hausdorff dimension at most $d-1$. The set $\cmcoup(P,Q)$ contains a single element $\pi$. For any closed convex function $\psi$ such that $\spt\pi \subset \partial\psi$, the gradient $\nabla \psi$ is a transport map, $(\nabla\psi)_{\#}P = Q$. This transport map is unique $P$-a.e.
\end{theorem}

In this context, the optimality of a transport map is not applicable if the Kantorovich problem has no solutions for which the objective is finite. Still, the focus on cyclical monotonicity is justified by Brenier's theorem and the earlier cited results on the sufficiency of cyclical monotonicity for optimality in case second moments do exist.


\subsection{Problem statement and main results}

To our knowledge, the following questions have not been treated in the literature for infinite measures without finite-moment assumptions:
\begin{enumerate}[label=(Q\arabic*), ref=(Q\arabic*)]
\item \label{Q:existence} {\bf Existence, Uniqueness:} Given two potentially infinite measures $\mu,\nu\in\Mz(\Rd)$, is there a (unique) cyclically monotone zero-coupling (see Definition~\ref{def:zero-coupling}) of $\mu,\nu$?
\item \label{Q:transport} {\bf Transport map:} Given such a zero-coupling $\gamma$, and a convex function $\psi$ such that $\spt \gamma\subset \partial \psi$, is it necessarily the case that the map $\nabla\psi$ pushes $\mu$ forward to $\nu$?
\item \label{Q:stability} {\bf Stability:} In case $\mu,\nu$ are limit measures of $P,Q$, respectively, in the definition in \cref{eq:def_regular_variation} of regular variation, and if $\pi$ is a cyclically monotone coupling between $P$ and $Q$, 
what are the connections between $(P,Q, \pi, \nabla\psi)$ and $(\mu,\nu, \gamma, \nabla\bpsi)$, where $\nabla\psi$ and $\nabla\bpsi$ are transport maps from $P$ to $Q$ and from $\mu$ to $\nu$, respectively? 
\end{enumerate}

The answer to \ref{Q:existence}, and in particular the uniqueness part, involves substantially more technicality than in \cite{McCann1995}, due to the potentially infinite masses of sets. 
A cyclically monotone zero-coupling always exists (Theorem~\ref{thm:existence}) and a sufficient condition for it to be unique is that one of the two measures vanishes on sets of Hausdorff dimension at most $d-1$ and, moreover, one of the two measures has infinite mass (Theorem~\ref{thm:unique}). We give a counter-example to the uniqueness if both measures are absolutely continuous and have finite mass (Remark~\ref{rk:c-e_uniqueness}).

Regarding \ref{Q:transport}, if $P,Q$ in \cref{thm:McCann} are replaced by infinite measures $\mu,\nu\in\Mzp(\Rd)$, the conclusion is no longer valid: even if the cyclically monotone zero-coupling between $\mu$ and $\nu$ exists and is unique, it may not be possible to write $\nu$ as $(\nabla \psi)_\# \mu$ for a closed convex function $\psi$, as the following counterexample shows. 

\begin{example}
\label{c-ex:existence_push_forward}
    Let $\mu,\nu\in\Mzp(\rset)$ and suppose that 
    $\mu(\rset_{>0}) = \nu(\rset_{<0}) = \infty$ while $\mu(\rset_{<0}) = \nu(\rset_{>0}) = 0$. The measure $\gamma = \mu \otimes \delta_0 + \delta_0 \otimes \nu$ is a zero-coupling of $\mu$ and $\nu$ [see Definition~\ref{def:zero-coupling}: we have $\gamma(A \times \reals) = \mu(A)$ and $\gamma(\reals \times A) = \nu(A)$ for every Borel set $A \subset \reals \sauf \{0\}$], and its support is contained in $\rbr{\{0\} \times \reals_{\le 0}} \cup \rbr{\reals_{\ge 0} \times \{0\}}$ and therefore (cyclically) monotone. Still, we claim that there cannot exist a non-decreasing function $T : \reals \to \reals$ such that $\nu = T_\# \mu$. Indeed, the existence of such a map $T$ would lead to a contradiction, as we show next.
    Let $x>0$. If $T(x)<0$, then, since $\mu(\rset_{>0}) = \infty$ but $\mu((x, \infty)) < \infty$ (as $\mu \in \Mzp(\rset)$), we would have
    \[ 
        \nu((-\infty,T(x)])
        = \mu(\{z\in\rset\sauf\{0\} \mid T(z) \le T(x)\}) 
        \ge \mu((0,x])
        = \infty, 
    \]
    in contradiction to the assumption that $\nu\in\Mzp(\rset)$. Hence, we have $T(x)\ge0$ for all $x>0$. But then we obtain another contradiction, since
    \[ 
        \infty 
        = \nu(\rset_{<0})
        = \mu(\{z\in\rset \sauf \{0\} \mid T(z)<0\})
        \le \mu(\rset_{<0})
        = 0.
    \]
    Therefore, a map $T$ with the stated properties cannot exist, proving the claim. \hfill \qed
\end{example}

In an earlier version of this paper, \citet{devalk2019tailsoptimaltransportplans}, ordinary couplings rather than zero-couplings between infinite measures were considered instead. This eventually led to an error: Example~\ref{c-ex:existence_push_forward} above is in contradiction to Theorem~4.4 in \citep{devalk2019tailsoptimaltransportplans}. The crucial point is that the measure $\gamma$ in the example is not a coupling between $\mu$ and $\nu$: we have $\gamma(\{0\} \times \reals) = \infty$ while $\mu(\{0\})$ is either undefined or zero, and similarly $\gamma(\reals \times \{0\}) = \infty$ is not equal to $\nu(\{0\})$. The mistake can be traced back to Lemma~A.2 in \citep{devalk2019tailsoptimaltransportplans}: the $\Mz$-limit $\gamma$ of a sequence of coupling measures $\gamma_n$ of measures $\mu_n$ and $\nu_n$ (see Section~\ref{sec:stability} for the definition of $\Mz$-convergence) is not necessarily a coupling itself between the $\Mz$-limits $\mu$ and $\nu$, but only a zero-coupling.


As will appear clearly at later stages of the paper, the answer to \ref{Q:transport} in Theorem~\ref{thm:expr_zcmcoup} depends on whether the origin plays the role of a reservoir.  A zero-coupling provides a cyclically monotone transport map $\nabla\psi$ if and only if the coupling is what we call \emph{proper}, that is, $\gamma(\{0\}\times\Rdmz)=0$: the origin $0 \in \Rd$ does not play the role of reservoir of mass. Sufficient conditions for cyclically monotone couplings to be proper are given in Section~\ref{sec:proper-zcoup}.

Our main results in relation to \ref{Q:stability} are summarized in \cref{fig:cd}. A sufficient condition for them to hold is that (a) both $P$ and $\mu$ vanish on sets of Hausdorff dimension at most $d-1$, referred to as ``small sets'', and (b) $\spt\nu\subset\spt\mu$. Weaker assumptions for (b) are asserted in \cref{main_thm_2}.
In the commutative diagram in \cref{fig:cd}, the notation $\RV$ denotes the operator that associates a probability measure to the limit measure in the definition in \cref{eq:def_regular_variation} of regular variation (or an extension thereof in the case of the coupling measure).
There are two ways to travel from probability measures $P,Q$ in the top-left of the diagram to the unique cyclically monotone zero-coupling $\gamma$ between $\mu,\nu$ in the bottom-right:
\begin{itemize}
    \item first take the limit measures $\mu$ and $\nu$ in the definition of regular variation of $P$ and $Q$ respectively, and then consider the unique cyclically monotone zero-coupling $\gamma$ between them;
    \item first consider the unique cyclically monotone coupling $\pi$ between $P$ and $Q$, and then take its limit measure $\gamma$ in an extension of the definition of regular variation.
\end{itemize}
Moreover, the transport map $\nabla \bar{\psi}$ associated with the zero-coupling $\gamma$ of $\mu$ and $\nu$ is the tail limit of the transport map $\nabla \psi$ associated with the coupling $\pi$ of $P$ and $Q$.

\begin{figure}
\begin{center}
    \begin{tikzcd}[nodes in empty cells=true, row sep=1.5cm, column sep=3cm]
    P,Q
    \arrow{r}{  } 
    \arrow[swap]{d}{\RV}
    &
    \left\{ \pi = (\Id\times\nabla\psi)_\#P\right\} = \cmcoup(P,Q)
    \arrow{d}{\RV}
    \\
    \mu,\nu
    \arrow[swap]{r}{ }
    & 
    \left\{ \gamma = (\Id\times\nabla\bpsi)_\#\mu \right\} = \zcmcoup(\mu,\nu)
    \end{tikzcd}
\end{center}
\caption{\label{fig:cd} 
Probability measures $P$ and $Q$ on $\Rd$, with $P$ vanishing on small sets; the coupling $\pi$ of $P$ and $Q$ with cyclically monotone support contained in the subdifferential of a convex function $\psi$; the limit measures $\mu$ and $\nu$ on $\Rdmz$ of $P$ and $Q$, respectively, in the definition of regular variation, with $\mu$ vanishing on small sets; the zero-coupling $\gamma$ of $\mu$ and $\nu$ with cyclically monotone support contained in the subdifferential of a convex function $\bpsi$, this subdifferential being the limit in the Fell topology of the one of the subdifferential of $\psi$. Under an additional condition on the support of $\mu$, the zero-coupling $\gamma$ is proper, so that $\gamma$ and $\nu$ can be written as push-forward measures of $\mu$ by $\Id \times \nabla \bpsi$ and $\nabla \bpsi$, respectively.}
\end{figure}
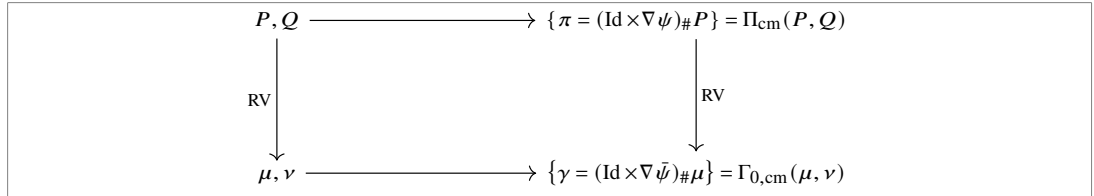

\subsection{Outline of the paper}

In Section~\ref{section:zero_couplings}, zero-marginals and zero-couplings are introduced, and
sequences of zero-couplings and sets containing their supports are studied. The existence of zero-couplings whose supports are cyclically monotone is proved and a representation of such zero-couplings using gradients of closed convex functions is derived.

Proper zero-couplings are the focus of Section~\ref{sec:proper-zcoup}. We provide conditions on the supports of the two measures to be coupled for the zero-coupling between them with (cyclically) monotone support to be proper. The latter property is crucial for the representation in Section~\ref{section:zero_couplings} to yield a push-forward formula for one measure in terms of the other, precluding situations as in the counterexample in Example~\ref{c-ex:existence_push_forward}.

In Section~\ref{sec:brenier}, an extension of Brenier's Theorem to measures in $\Mzp(\Rd)$ with potentially infinite total mass is provided. When at least one of the zero-marginals has infinite mass, conditions are given for the cyclically monotone zero-coupling to be unique and for a cyclically monotone transport map to exist.

Section~\ref{sec:EVTOT} investigates the connections linking couplings between multivariate regularly varying probability distributions on the one hand and zero-couplings between their limit measures in the definition of regular variation on the other hand. A criterion is given for a cyclically monotone transport map to exist; it only concerns the support of the first zero-marginal.
An extended form of regular variation of the cyclically monotone coupling is proven, together with the other assertions in the commutative diagram in \cref{fig:cd}. A short discussion in Section~\ref{sec:discussion} concludes the paper.

The proofs of the main results are given at the end of each section. The appendix contains background material on $\Mz$-convergence and convex analysis (Appendix~\ref{sec:background}), together with auxiliary results on zero-couplings (Appendix~\ref{section:seq_zcoup}).

\subsection{Notation}
\label{sec:notation}


On $\reals^k$, we always work with the Euclidean inner product $\inpr{\point,\point}$ and the Euclidean norm $\nbr{\point}$. The interior, closure and boundary of a set $A$ of a topological space are denoted by $\intr A$, $\cl A$ and $\bnd A$, 
respectively. The open ball in $\rset^k$ with center $x \in \rset^k$ and radius $r > 0$ is denoted by $\ball_{x,r} = \ball(x, r) = \ball_k(x,r) = \cbr{y \in \rset^k \mid \nbr{y-x} < r}$. The $\sigma$-field of Borel subsets on a topological space $X$ is denoted by $\borel(X)$. We frequently use the projections $\proj_1, \proj_2 : \RdRd \to \Rd$ defined by $\proj_1(x,y) = x$ and $\proj_2(x,y) = y$.

The effective domain, or just domain in short, of a convex function $\psi : \Rd \to \rset \cup \cbr{+\infty}$ is the set $\dom \psi = \{x \in \Rd \mid \psi(x) < \infty\}$, while its subgradient is the set
\begin{equation}
    \label{eq:subdiff}
    \partial \psi = \cbr{(x,v) \in \dom \psi \times \Rd \mid \forall y \in \Rd : \psi(y) \ge \psi(x) + \inpr{y-x, v}}.
\end{equation}
For $x \in \Rd$, we write $\partial \psi(x) = \cbr{v \in \Rd \mid (x, v) \in \partial \psi}$; note that $\partial \psi(x) = \emptyset$ if $x \not\in \dom \psi$. We refer to \cref{section:CVX} for more details and useful results on convex analysis. 

For two Borel measures $\mu$ and $\nu$ on $\Rd$, let $\coup(\mu,\nu)$ denote the set of coupling measures, i.e., the Borel measures $\gamma$ on $\RdRd$ such that $\gamma(B \times \Rd) = \mu(B)$ and $\gamma(\Rd \times B) = \nu(B)$ for all $B \in \borel(\Rd)$. Furthermore, $\cmcoup(\mu,\nu)$ is the subset of such couplings $\gamma$ whose support is cyclically monotone in the sense of \cref{rel:cyclical_mononicity0} below. The set of Borel probability measures on $\Rd$ is denoted by $\Prob(\Rd)$. 


Let the map $T : \dom T \subset \reals^k \to \reals^l$ be Borel measurable and let $\mu$ be a Borel measure on a $S \subset \reals^k$, where both $\dom T$ and $S$ are Borel sets. The push-forward measure $T_\# \mu$ of $\mu$ by $T$ is the Borel measure on $\reals^l$ defined by
\[
    \forall B \in \borel(\reals^l): \qquad
    T_{\#}\mu(B) 
    = \mu\bigl( T^{-1}(B) \bigr)
    = \mu \rbr{\cbr{x \in S \cap \dom T \mid T(x) \in B}}
    .
\]
\section{Zero-couplings between infinite measures}
\label{section:zero_couplings}

First we introduce the notion of zero-couplings in Section~\ref{sec:Mzp}, and then discuss their stability under $\Mz$-convergence of measures in Section~\ref{sec:stability}. Finally, we discuss their existence and their representation in terms of push-forward measures plus a potential remainder term in Section~\ref{sec:representation}.

\subsection{Zero-couplings}
\label{sec:Mzp}

A subset $B$ of $\Rd$ is bounded away from $0 \in \Rd$ if $0 \not\in \cl(B)$. Let $\borel_0(\Rd)$ denote the collection of Borel subsets of $\Rdmz = \Rd \sauf \cbr{0}$ that are bounded away from $0$. Let $\Mz(\Rd)$ denote the set of nonnegative Borel measures $\mu$ on $\Rdmz$ such that $\mu(B) < \infty$ for all $B \in \borel_0(\Rd)$. For $\mu \in \Mz(\Rd)$, let $[\mu]$ denote the set of nonnegative Borel measures $\bar\mu$ on $\Rd$ whose restriction $\res\bar\mu$ to $\Rdmz$ is equal to $\mu$, i.e., $\bar\mu(B) = \mu(B)$ for any $B \in \borel(\Rdmz)$. Further, let $\tilde{\mu} \in [\mu]$ be the unique member of $[\mu]$ such that $\tilde{\mu}(\cbr{0}) = 0$, i.e., $\tilde{\mu}(B) = \mu(B \sauf \cbr{0})$ for $B \in \borel(\Rd)$.
For Borel measures $\bar{\mu}$ on $\Rd$, we use the notation $[\bar{\mu}]$ in the same way.

It will be convenient to define the support of $\mu \in \Mz(\Rd)$ as the set of points $x \in \Rd$ such that $\tilde{\mu}(U) > 0$ for every open set $U \subset \Rd$ such that $x \in U$; that is, the support of $\mu$ is defined as the support (in the traditional sense) of $\tilde{\mu}$. This means that $0$ can belong to the support of $\mu$, even though $\mu$ is a measure on $\Rdmz$ only. In particular, $0 \in \spt \mu$ if $\mu(U \sauf \cbr{0}) > 0$ for every open set $U \in \Rd$ containing $0$. Note that $\spt \mu$ cannot be equal to $\cbr{0}$.

\begin{definition}[Zero-marginals]
    \label{def:zero-marginals}
    The left and right zero-marginals of a nonnegative Borel measure $\gamma$ on $(\RdRd) \sauf \cbr{0}$ are the nonnegative Borel measures $\varphi_1 \gamma, \varphi_2 \gamma$ on $\Rdmz$, respectively, with
    \begin{align*}
       \varphi_1 \gamma(A) = \gamma(A \times \Rd)
       \quad \text{and} \quad
       \varphi_2 \gamma(A) = \gamma(\Rd \times A) 
       \quad \text{for $A \in \borel(\Rdmz)$.}
    \end{align*}
\end{definition}

If $A \subset \Rdmz$ is bounded away from $0 \in \Rd$, then $A \times \Rd$ and $\Rd \times A$ are bounded away from $0 \in \RdRd$. It follows that $\gamma \in \Mzp(\RdRd)$ implies $\varphi_1 \gamma, \varphi_2 \gamma \in \Mzp(\Rd)$.

A subset $T$ of $\RdRd$ is \emph{monotone} if
\begin{equation}
    \label{eq:monotone}
    \forall (x_1,y_1),(x_2,y_2) \in T: \qquad
    \inpr{x_1-x_2,y_1-y_2} \ge 0.
\end{equation}
In the following, we need a stronger notion: a set $T \subset \RdRd$ is \emph{cyclically monotone} if for every integer $n \ge 2$ and all $(x_1,y_1),\ldots,(x_n,y_n) \in T$ with $y_{n+1} := y_1$, we have 
\begin{equation}
    \label{rel:cyclical_mononicity0}
    \sum_{i=1}^n \inpr{x_i,y_i} \ge \sum_{i=1}^n \inpr{x_i,y_{i+1}}.
\end{equation}
The case $n = 2$ in \cref{rel:cyclical_mononicity0} leads back to \cref{eq:monotone}, so that a cyclically monotone set is also monotone.



\begin{definition}[Zero-couplings]
    \label{def:zero-coupling}
    Let $\mu,\nu\in\Mzp(\Rd)$. The set $\zcoup(\mu,\nu)$ of zero-couplings between $\mu$ and $\nu$ consists of all nonnegative Borel measures $\gamma$ on $(\RdRd)\sauf\{0\}$ such that 
    \[ 
        \varphi_1\gamma = \mu 
        \quad \text{and} \quad 
        \varphi_2\gamma = \nu. 
    \]
    Let $\zmcoup(\mu,\nu)$ and $\zcmcoup(\mu,\nu)$ denote the sets of zero-couplings $\gamma$ between $\mu$ and $\nu$ such that $\spt\gamma$ is monotone and cyclically monotone, respectively.
\end{definition}

In Definition~\ref{def:zero-coupling}, we implicitly use the property that $\zcoup(\mu,\nu) \subset \Mz(\RdRd)$, which is justified by Lemma~\ref{lm:zcoup_in_Mzp}. Clearly, $\zcmcoup(\mu,\nu) \subset \zmcoup(\mu,\nu) \subset \zcoup(\mu,\nu)$ for $\mu,\nu \in \Mzp(\Rd)$.

\begin{remark}
    \label{rkdef:couplings}
    \cref{def:zero-marginals,def:zero-coupling} are similar to the ones of the usual marginals and couplings of measures.
    Indeed, for a measure $\pi$ in the set $\Mp(\RdRd)$ of nonnegative Borel measures on $\Rd$, the marginals $\phi_1 \pi, \phi_2 \pi$ are defined like the zero-marginals with the difference that the set $A$ is allowed to contain the origin, while the collection of (cyclically monotone) couplings $\Pi_{(\mathrm{cm})}(p,q)\subset\Mp(\RdRd)$ between $p,q\in\Mp(\Rd)$ is defined like the one of zero-couplings by substituting $\varphi_1,\varphi_2$ with $\phi_1,\phi_2$.
\end{remark}

The following lemma ensures $\zcoup(\mu,\nu)\subset \Mzp(\Rd\times \Rd)$, so that for any zero-coupling $\gamma$, the expressions $\spt \gamma$, $\tilde{\gamma}$ and $[\gamma]$ are well-defined.

\begin{lemma}
\label{lm:zcoup_in_Mzp}
For all $\mu,\nu \in \Mzp(\Rd)$, we have $\zcoup(\mu,\nu) \subset \Mzp(\RdRd)$.
\end{lemma}

\begin{proof}
    To avoid ambiguities, we indicate the dimension of a ball by a subscript. Let $r > 0$. For $0 < r' \le r/\sqrt{2}$, we have $\ball_d(0, r') \times \ball_d(0, r') \subset \ball_{2d}(0,r)$ and thus
    \[
        (\RdRd) \sauf \ball_{2d}(0,r)
        \subset
        \rbr{\rbr{\Rd\sauf\ball_{d}(0,r')}\times\Rd}
        \cup
        \rbr{\Rd\times\rbr{\Rd\sauf\ball_{d}(0,r')}}.
    \]
    For any $\gamma \in \zcoup(\mu,\nu)$, we have therefore
    \[
        \gamma\rbr{(\RdRd) \sauf \ball_{2d}(0,r)}
        \le \mu(\Rd\sauf\ball_d(0,r')) + \nu(\Rd\sauf\ball_d(0,r')) 
        < \infty. 
        \qedhere
    \]
\end{proof}

We now define a special set of zero-couplings that will be of major interest in the remainder of this paper.

\begin{definition}[Proper zero-couplings]
    \label{def:proper}
    A zero-coupling $\gamma \in \zcoup(\mu,\nu)$ is called \emph{proper} if 
    \[
    \gamma\rbr{\cbr{0} \times \Rdmz}
    = 
    0
    .
    \]
    Let $\zprcoup(\mu,\nu)$ denote the set of proper zero-couplings and let $\zpmcoup(\mu,\nu) = \zprcoup(\mu,\nu) \cap \zmcoup(\mu,\nu)$ and $\zpcoup(\mu,\nu) = \zprcoup(\mu,\nu) \cap \zcmcoup(\mu,\nu)$ denote the sets of proper zero-couplings with monotone and cyclically monotone support, respectively.
\end{definition}

\cref{def:proper} is crucial because for applications involving transport maps between measures, only proper zero-couplings are relevant. Indeed, if $\gamma(\{0\}\times \Rdmz)$ is non-zero, the origin $\{0\}$ serves as an additional reservoir of mass when transporting $\mu$ onto $\nu$. As shall become clear in Theorem~\ref{thm:expr_zcmcoup}, this precludes the construction of a map $T$ that pushes $\mu$ forward to $\nu$.

Note that `being proper' is a property which concerns equivalence classes of zero-couplings: If the zero-couplings $\gamma_1$ and $\gamma_2$ belong to the same equivalence class $[\gamma]$, then $\gamma_1$ is proper if and only if $\gamma_2$ is proper. Additional facts about zero-couplings and their supports are collected in Appendix~\ref{section:seq_zcoup}.


\subsection{Stability}
\label{sec:stability}


To deal with transport problems between infinite measures, a common approach is to consider transport problems between finite approximations for which there are already results in the literature. We refer to \citet{huesmann2013optimal,huesmann2016transportrandom} and \citet{CatalanoLavenant2021} where the authors are interested in vague convergence and weak* convergence respectively.
This approach is all the more natural in our setting since we aim at studying transport problems between infinite measures in $\Mzp(\Rd)$ appearing as limits of rescaled probability measures in the definition of regular variation (Section~\ref{sec:EVTOT}).

Let $\Cbz(\Rd)$ denote the set of bounded and real-valued functions $f$ on $\Rdmz$ that vanish on a neighbourhood of the origin, i.e., for which there exists $r > 0$ such that $f(x) = 0$ whenever $\nbr{x} \le r$. The topology on $\Mzp(\Rd)$ is the weakest one for which the evaluation functionals $\mu \mapsto \mu(f) = \int f \, \diff \mu$ are continuous for all $f \in \mathcal{C}_0(\Rd)$. The space $\Mzp(\Rd)$ equipped with this topology is Polish \citep[Theorem~2.3]{hult2006regular}. For a sequence $(\mu_n)_n$ in $\Mzp(\Rd)$, the convergence $\mu_n \Mto \mu$ in $\Mzp(\Rd)$ as $n\to\infty$ is equivalent to the convergence $\mu_n(f) \to \mu(f)$ as $n\to\infty$ for all $f \in \mathcal{C}_0(\Rd)$. For convergence of measures in $\Mzp(\Rd)$, a version of the Portmanteau theorem and of Prohorov's theorem are available in \cite{hult2006regular} and are quoted in \cref{section:Mzp}.

Consider two sequences $\mu_n,\nu_n\in\Mzp(\Rd)$, $n\in\nat$, converging to $\mu,\nu\in\Mzp(\Rd)$, respectively, and a sequence of zero-couplings $\gamma_n\in\zcoup(\mu_n,\nu_n)$ and closed sets $T_n \subset \RdRd$ satisfying $\spt\gamma_n\subset T_n$.
The asymptotic behaviour of $\gamma_n$ and $T_n$ is described below in $\Mzp(\RdRd)$ and in the Fell topology, respectively.
\cref{cor:stability} is inspired by results from \citet[Theorem~1.1]{segers2022}, treating weakly converging sequences of probability measures and their couplings.

We follow \cite{Molchanov2017} for the description and properties of the Fell topology. For an open subset $\mbb{E}$ of $\Rd$, let $\Fell(\mbb{E})$ denote the set of all subsets of $\mbb{E}$ which are closed in the trace topology on $\mbb{E}$. 
For a subset $A$ of $\mbb{E}$, we consider the collections $\Fell_A=\{F\in\Fell(\mbb{E})\mid F\cap A \neq \emptyset\}$ and $\Fell^A=\{F\in\Fell(\mbb{E})\mid F\cap A = \emptyset\}$ of sets that respectively hit and miss $A$.
The Fell hit-and-miss topology is defined by the sub-base consisting of $\Fell_G$ for all open subsets $G$ of $\mbb{E}$ and $\Fell^K$ for all compact subsets $K$ of $\mbb{E}$. 
Let $F,F_n\in\Fell(\mbb{E})$, $n \in \nat$. The sequence $F_n$ $\Fell$-converges to $F$ as $n\to\infty$ in the Fell topology---we write $F_n\Fto F$---if for every open $G\subset\mbb{E}$ such that $F\in\Fell_G$ we have $F_n\in\Fell_G$ for all large $n$, and for every compact $K\subset\mbb{E}$ satisfying $F\in\Fell^K$ we have $F_n\in\Fell^K$ for all large $n$.

We write $\Fellcm$ for the collection of closed subsets of $\RdRd$ satisfying the cyclical monotonicity property in \cref{rel:cyclical_mononicity0} 
and $\Fellmcm$ for the collection of maximal cyclically monotone sets, which are the cyclically monotone sets that are not strictly included in any other cyclically monotone set---obviously any maximal cyclically monotone set is closed. Likewise, we write $\Fellmm$ for the collection of maximal monotone sets, that is, monotone sets as in \cref{eq:monotone} that are not strictly included in any other monotone set.


\begin{theorem}[Stability]
\label{cor:stability}
    Let $\mu_n,\nu_n\in\Mzp(\Rd)$ and 
    $\gamma_n\in\zcoup(\mu_n,\nu_n)$,  
    $n \in \NN$, be such that $\mu_n \vto \mu$ and $\nu_n \vto \nu$. 
    Let $T_n\in\Fell(\RdRd)$ be a sequence of closed sets in $\RdRd$ such that $\spt\gamma_n\subset T_n$ for each $n \in \NN$. 
    Then there exists an infinite subset $\mathrm{N}$ of $\NN$ such that
    \begin{enumerate}[label=(\alph*)]
    \item $\gamma_n\Mto\gamma$ as $n\to\infty$ in $\mathrm{N}$ for some $\gamma\in\zcoup(\mu,\nu)$,
    \item $T_n\Fto T$ as $n\to\infty$ in $\mathrm{N}$ for some $T\in\Fell({\RdRd})$,
    \item $\spt\gamma\subset T$.
    \end{enumerate}
    If we further assume
    that $\gamma_n\in\zcmcoup(\mu_n,\nu_n)$ and $T_n\in\Fellmcm$ for all $n \in \NN$, then all limit points of $\gamma_n$ belong to $\zcmcoup(\mu,\nu)$ while all limit points of $T_n$ lie in $\Fellmcm$.
\end{theorem}

Anticipating the uniqueness of the cyclically monotone zero-coupling asserted in \cref{thm:unique}, we prove a stronger result in the cyclically monotone case.
It will be of major interest in \cref{sec:EVTOT} to derive the convergences announced in the commutative diagram in \cref{fig:cd}. For $V \subset \Rd$ and $T \subset \RdRd$, we write $T \llcorner V = T \cap (V \times \Rd)$, to be thought of as the restriction of $T$ (seen as a multi-valued map) to $V$.

\begin{corollary}
\label{cv:cm_coupling}
    Besides the assumptions of \cref{cor:stability}, assume that $\gamma_n\in\zcmcoup(\mu_n,\nu_n)$ and $T_n\in\Fellmcm$ for every $n \in \NN$ and that there exist $\gamma\in\Mzp(\RdRd)$ and $T\in\Fellmcm$ such that $\zcmcoup(\mu,\nu)=\{\gamma\}$ and $\spt\gamma\subset T$. Then, writing $V = \intr(\spt \mu)$, we have
    \[
        \gamma_n\Mto\gamma 
        \qquad \text{and} \qquad 
        T_n \llcorner V \Fto T \llcorner V \: \text{ in } \Fell(V\times\Rd)
        \qquad
        \text{as } n\to\infty.
    \]
\end{corollary}


\subsection*{Proofs}

\begin{proof}[Proof of \cref{cor:stability}]
We deal with the assertions (a) and (b,c) separately.
\smallskip

\noindent\emph{Proof of (a)} ---
    If $\mu_n, \nu_n$ $\Mzp$-converge to $\mu,\nu\in\Mzp$ respectively, then the sets $K=\{\mu_n:n\ge1\}\cup\{\mu\}$ and $L=\{\nu_n:n\ge1\}\cup\{\nu\}$ are both compact. By \cref{segers2022:lemma4.1}(c), the set $Z = \left\{ (a,b,\zeta )\mid a\in K, b\in L, \zeta \in \zcoup(a,b) \right\}$ is then compact too. 
    Since the triplets $(\mu_n,\nu_n,\gamma_n)$ belong to $Z$, there is an infinite subset  $\mathrm{N} \subset \nset$ and a limit point $(\mu',\nu',\gamma)\in Z$ such that $(\mu_n,\nu_n,\gamma_n)\to (\mu',\nu',\gamma)$ along $\mathrm{N}$, in $\Mz(\Rd)\times\Mz(\Rd)\times\Mz(\Rd\times\Rd)$. In particular, $\gamma_n\vto \gamma$ and, by definition of $Z$, also $\gamma\in\zcoup(\mu',\nu')$. Since, by assumption, $\mu_n$ and $\nu_n$ converge to $\mu$ and $\nu$, respectively, we obtain $\mu'=\mu$ and $\nu'=\nu$. The assertion follows.

    We turn to  the additional assumption that $\gamma_n \in \zcmcoup(\mu_n,\nu_n)$. 
    The set 
    \[ 
        \tilde Z = \cbr{ (a,b,\zeta )\mid a\in K, b\in L, \zeta \in \zcmcoup(a,b) }
    \]
    can be written as $\tilde Z = Z \cap \{ (a,b,\zeta )\mid a\in \Mz(\Rd), b\in \Mz(\Rd), \zeta \in \zcmcoup(a,b) \}$. The first set in the intersection is compact while, from \cref{segers2022:lemma4.1}(b), the second set is closed. Thus $\tilde Z$ is closed. Since $(\mu_n,\nu_n,\gamma_n)\in \tilde Z$ for all $n$, so does the limit point $(\mu,\nu,\gamma)$, whence $\gamma\in\zcmcoup(\mu,\nu)$.

\smallskip{}

\noindent\emph{Proof of (b) and (c)} ---
    Since $\Fell(\RdRd)$ is compact, there exist an infinite subset $\mrm{N}^\prime$ of $\mrm{N}$ in the proof of~(a) together with some $T$ in $\Fell(\RdRd)$ such that $T_n$ $\Fell$-converges to $T$ as $n\to\infty$ in $\mrm{N}^\prime$. Thanks to Lemma~3.2 in \cite{segers2022}, $\Fellmcm(\RdRd)\cup\{\emptyset\}$ is a closed subset of $\Fell(\RdRd)$. As a consequence, $T_n\in\Fellmcm(\RdRd),n\ge1$ implies $T\in\Fellmcm(\RdRd)\cup\{\emptyset\}$.
        Since $\mu$ and $\nu$ are nonzero measures and since $\gamma\in\zcmcoup(\mu,\nu)$, we have $\spt\gamma\neq\emptyset$. By \cref{lem:sptlsc}, we have $\spt\gamma\subset T$, so that $T$ is non empty, and $T$ must lie in  $\Fellmcm(\RdRd)$.
\end{proof}

\begin{proof}[Proof of Corollary~\ref{cv:cm_coupling}]
    The assertion $\gamma_n\Mto\gamma$ is a direct consequence of \cref{cor:stability}: if every subsequence of $(\gamma_n)_n$ contains a convergent subsequence and if all limits points of such subsequences are equal to $\gamma$, then $\gamma_n$ must converge to $\gamma$ in $\Mzp(\RdRd)$ along the whole sequence.

    Since $\gamma_n\in\zcmcoup(\mu_n,\nu_n)$ for all $n \in \nat$, \cref{cor:stability} implies that any limit point of $(T_n)_n$ belongs to $\Fellmcm(\RdRd)$ and thus to $\Fellmm(\RdRd)$, because any maximal cyclically monotone set is also maximal monotone \citep[Theorems~12.17 and~12.25]{rockafellarwets98}. Let $T_\infty \in \Fellmm$ be a limit point of $(T_n)_n$ as $n \to \infty$ for $n \in \mathrm{N}$, with $\mathrm{N}$ an infinite subset of $\nat$. By point~(c) of \cref{cor:stability}, we have $\spt\gamma\subset T_\infty$. For $T \in \Fellmcm \subset \Fellmm$ in the statement of the corollory, \cref{lem:STV} implies $T \llcorner V = T_\infty \llcorner V$. Since $V\subset\Rd$ is open, the restriction map $\point \llcorner V : \Fell(\RdRd) \to \Fell(V \times \Rd)$ is continuous  \citep[Proposition~D.12]{segers2022}. It follows that along the subsequence $\mathrm{N}$, the restrictions $T_n \llcorner V$ converge in $\Fell(V \times \Rd)$ to $T_\infty \llcorner V=T \llcorner V$. Since this is true for an arbitrary converging subsequence and since any subsequence contains a converging sequence (\cref{cor:stability}), we conclude that $T_n \llcorner V \to T \llcorner V$ in $\Fell(V \times \Rd)$ along the whole sequence, as required.
\end{proof}

\subsection{Existence and representation}
\label{sec:representation}
When a moment assumption is made on both zero-marginals $\mu$ and $\nu$, the existence of a cyclically monotone zero-coupling between $\mu$ and $\nu$ is a direct application of Theorems~A.5 and Theorem~A.13 in \citet{GuillenMou2019}.
Relying on Theorem~\ref{cor:stability}, which provides stability results and the existence of limit points for the sequence of couplings, we derive the same conclusion without making any moment assumption.

\begin{theorem}[Existence]
\label{thm:existence}
For all $\mu,\nu\in\Mzp(\Rd)$ we have $\zcmcoup(\mu,\nu) \neq \emptyset$.
\end{theorem}



Recall the notion of the subdifferential of a convex function in \cref{eq:subdiff}.
Relying on Rockafellar's Theorem~\ref{thm:rockafellar}, which establishes a connection between cyclical monotonicity and subdifferentials of closed convex functions, we derive a representation result for cyclically monotone zero-couplings similar to Proposition~10 in \citet{McCann1995}. Recall \cref{def:proper} of proper zero-couplings and recall that for $\mu \in \Mz(\Rd)$, the measure $\tilde{\mu}$ is the extension of $\mu$ to $\Rd$ determined by $\tilde{\mu}(\cbr{0}) = 0$, while $\res m$ denotes the restriction to $\Rdmz$ of a measure $m$ defined on $\Rd$. 
Let $\Id$ be the identity map on $\Rd$, and for a map $f$ defined on some subset of $\Rd$, put $(\Id \times f)(x) = (x, f(x))$ for $x \in \dom f$. Here and throughout, $\dom f$ denotes the subset of $\Rd$ where $f$ is defined (except for convex functions $\psi$ from $\Rd$ into $\reals \cup \cbr{\infty}$, where $\dom \psi = \cbr{x \in \Rd \mid \psi(x) < \infty}$). 
For a measure $m$ on $\Rd$ and a measurable map $f$ from some subset of $\Rd$ to $\Rd$, the measure $(\Id \times f)_{\#}m$ on $\RdRd$ is determined for $A, B \in \borel(\Rd)$ by
\[
    (\Id \times f)_{\#}m (A \times B)
    = m \rbr{A \cap f^{-1}(B)}
    = m \rbr{\cbr{x \in A \cap \dom f: f(x) \in B}}.
\]
If $m$ is a measure defined on some subset $M$ of $\Rd$ and if $f$ is a measurable function on some subset $\dom f$ of $\Rd$, we use $f_\# m$ to mean the measure on $\Rd$ defined as 
\[ 
    f_\# m(A) = m(M \cap f^{-1}(A)) = m(\cbr{x \in M \cap \dom f : f(x) \in A}),
\] 
for Borel sets $A$.
Also, we extend common terminology and call the measure $f_\# M$ the push-forward measure of $m$ by the mapping $f$, and if $\lambda= f_\# m$ we shall say that $f$ pushes $m$ forward to $\lambda$. For instance, the measure $(\nabla \psi)_{\#} \mu$ in \cref{eq:proper} below is defined for $A \in \borel(\Rd)$ by
\[
    (\nabla \psi)_{\#} \mu(A) = \mu \rbr{(\nabla \psi)^{-1}(A) \sauf \cbr{0}}
    \text{ with }
    (\nabla \psi)^{-1}(A) = \cbr{x \in \dom \nabla \psi \mid \nabla\psi(x) \in A}.
\]





By Rockafellar's Theorem \ref{thm:rockafellar}, any cyclically monotone set such as $\spt \gamma$ in the statement below is contained in the subdifferential $\partial \psi$ of some closed convex function $\psi$. 

\begin{theorem}[Representation]
    \label{thm:expr_zcmcoup}
    Let $\mu,\nu\in\Mzp(\Rd)$ and $\gamma\in\zcmcoup(\mu,\nu)$, and assume that $\mu$ vanishes on sets of Hausdorff dimension at most $d-1$. Then for any closed convex function $\psi: \Rd \to \reals \cup \cbr{\infty}$ satisfying $\spt\gamma\subset \partial\psi$, we have $\mu(\Rdmz \sauf \dom \nabla\psi) = 0$. The measure $\nu$ can be written as the push-forward of $\mu$ by $\nabla\psi$, up to a residual term: for every Borel set $B\subset\Rdmz$,
    \[
        \nu(B) = (\nabla\psi)_{\#} \mu(B) + \gamma(\cbr{0} \times B). 
    \]
    In addition, $\gamma$ is uniquely determined by $\nabla \psi, \mu, \nu$ through the identities
    \begin{align*}
        \gamma(A \times B) 
        &= (\Id \times \nabla\psi)_{\#}\mu(A\times B)
        = \mu \rbr{A \cap (\nabla\psi)^{-1}(B)}, &&
        A \in \borel(\Rdmz), \, B \in \borel(\Rd), \\
        \gamma(\cbr{0} \times B)
        &= \nu(B) - (\nabla\psi)_{\#} \mu(B), &&
        B \in \borel(\Rd), \, 0 \not\in \cl(B).
    \end{align*}
    %
    Finally, if $\gamma$ is proper, i.e., $\gamma\in\zpcoup(\mu,\nu)$, then we have the \emph{proper representation} 
    \begin{equation}
        \label{eq:proper}
        \gamma = \res \rbr{ (\Id \times \nabla\psi)_{\#}\mu }
        \quad \text{and} \quad
        \nu = \res \rbr{(\nabla \psi)_{\#}\mu}.
    \end{equation}
\end{theorem}

Rephrasing the last part of the theorem, if a proper cyclically monotone zero-coupling $\gamma$  of $\mu,\nu \in \Mzp(\Rd)$ exists, then there is a closed convex function $\psi$ whose gradient $\nabla\psi$ pushes $\mu$ forward to a measure $\bar \nu$ on $\Rd$ of which the restriction to $\Rdmz$ is equal to $\nu$.

Contrary to McCann's result for cyclically monotone couplings between finite measures with equal mass in Theorem~\ref{thm:McCann}, there is in \cref{thm:expr_zcmcoup} an extra term $\gamma(\cbr{0} \times B)$. It can be interpreted as assigning mass to $\{0\}\times\Rdmz$ when we need to ``extract'' mass from the origin for the cyclically monotone zero-coupling to exist. Example~\ref{c-ex:existence_push_forward} presents a case where this term is non-zero and the zero-coupling is not proper.

\subsection*{Proofs}


\begin{proof}[Proof of Theorem~\ref{thm:existence}]
    If $\mu$ and $\nu$ are both the zero measure on $\Rdmz$, then the zero measure on $\RdRdmz$ is a zero-coupling of $\mu$ and $\nu$. So assume at least one of $\mu$ and $\nu$ is not the zero measure.

    We construct two sequences $\mu_n,\nu_n\in\Mzp(\Rd)$, $n\in\nat$, of finite and equal mass converging to $\mu,\nu$ in $\Mzp(\Rd)$, respectively, together with a sequence $\tilde\gamma_n\in\Pi_{cm}(\tilde\mu_n,\tilde\nu_n)$,
    $n\in\nat$, of cyclically monotone couplings between $\tilde\mu_n$ and $\tilde\nu_n$ whose restrictions $\gamma_n=\res\tilde\gamma_n$ to $\RdRdmz$ are zero-couplings between $\mu_n=\res\tilde\mu_n$ and $\nu_n=\res\tilde\nu_n$; recall that $\tilde{\mu}_n$ is the extension of $\mu_n$ from $\Rdmz$ to $\Rd$ by $\tilde{\mu}_n(\cbr{0})=0$, and similarly for the other measures. The result then follows by \cref{cor:stability} upon taking the limit $\gamma$ of a converging subsequence of $\gamma_n$.
    
    For $n \in \mbb{N}$, let $B_n$ be the closed ball in $\Rd$ with center $0$ and radius $1/n$. Define $\tilde\mu_n = \tilde\mu\rbr{\point \cap B_n^c}$ and $\tilde{\nu}_n = \tilde\nu\rbr{\point \cap B_n^c}$. These are measures on $\Rd$ that coincide with $\mu$ and $\nu$, respectively, on the complement of $B_n$. We have $\mu_n=\res\tilde\mu_n \Mto \mu$ as $n \to \infty$,
    since for every $f \in \Cbz(\Rd)$, we have $\tilde\mu_n(f) = \mu(f)$ as soon as $n \in \nat$ is sufficiently large such that $f$ vanishes on $B_n$. Likewise, we have $\nu_n=\res\tilde\nu_n \Mto \nu$.

    It is clear that both $\tilde\mu_n$ and $\tilde\nu_n$ have finite total mass since they originate from measures lying in $\Mzp(\Rd)$, restricted to a set bounded away from the origin. If $\tilde\mu_n(\Rd)$ and $\tilde\nu_n(\Rd)$ differ, we modify the one with the smallest total mass. We write the reasoning for the case $\tilde\mu_n(\Rd) < \tilde\nu_n(\Rd)$; the same argument can be used in the other case.
    Let $\delta_n = \tilde\nu_n(\Rd) - \tilde\mu_n(\Rd) \in (0,\infty)$ and write $\tilde\kappa_n$ for the Lebesgue-uniform probability distribution on $B_n$.
    Replace $\tilde\mu_n$ by $\tilde\mu_n + \delta_n\cdot\tilde\kappa_n$. The same argument as above then still gives $\res\tilde\mu_n \Mto \mu$ as $n \to \infty$, but this time we have $a_n=\tilde\mu_n(\Rd)=\tilde\nu_n(\Rd)$. Since $\nu$ is not the zero measure, we have $a_n > 0$ for all sufficiently large $n$.
    
    We can now apply Theorem~6 in \cite{McCann1995} to the probability measures $\hat{\mu}_n = a_n^{-1} \tilde\mu_n$ and $\hat{\nu}_n = a_n^{-1} \tilde\nu_n$. It yields the existence of some $\hat{\pi}_n \in \Pi(\hat{\mu}_n,\hat{\nu}_n)$ with cyclically monotone support. The measure $\tilde\gamma_n = a_n \hat{\pi}_n$ then belongs to $\Pi(\tilde\mu_n,\tilde\nu_n)$ and has the same cyclically monotone support as $\hat{\pi}_n$. By Rockafellar's Theorem~\ref{thm:rockafellar}, there exists a closed convex function $\psi_n$ on $\Rd$ such that the set $\spt \gamma_n=\spt\tilde\gamma_n$ is contained in $T_n = \partial \psi_n\in\Fellmcm(\RdRd)$. 
    One may notice that $\tilde\gamma_n$ is a true coupling between $\tilde\mu_n$ and $\tilde\nu_n$ seen as Borel measures on $\Rd$ putting no mass on $\{0\}$; indeed, by construction,  $\tilde\gamma_n$ puts no mass on $(\{0\}\times\Rd) \cup (\Rd\times\{0\})$. Its restriction $\gamma_n = \res \tilde\gamma_n$ to $\RdRdmz$ lies in $\zcmcoup(\mu_n,\nu_n)$ and we have $\spt \gamma_n=\spt\tilde\gamma_n$.
    
    By Corollary~\ref{cor:stability}, there exist an infinite subset $\mrm{N}$ of $\mbb{N}$ and a measure $\gamma$ in $\zcmcoup(\mu,\nu)$ such that $\gamma_n$ $\Mzp$-converges to $\gamma$ as $n\to\infty$ in $\mrm{N}$.
\end{proof}


\begin{proof}[Proof of \cref{thm:expr_zcmcoup}]
    The proof is largely inspired by the one of Proposition~10 in \cite{McCann1995}.
    We split $\RdRdmz$ into $\Rdmz\times\Rd$ and $\{0\}\times\Rdmz$. 
    
    \smallskip{}
    
    \noindent\emph{Restriction of $\gamma$ to $\Rdmz\times\Rd$.} ---
    We first consider $\gamma$ on subsets of $\Rdmz\times\Rd$.
    Let 
    \begin{align*}
        S
        &=\cbr{(x,\nabla\psi(x)):x\in\dom\nabla\psi\cap\Rdmz} \\
        &=\rbr{(\dom\nabla\psi\cap \Rdmz)\times\Rd}\cap\partial\psi.
        \end{align*}
    Following the argument from the proof of Proposition~10 in \cite{McCann1995}, the measurability of $\nabla\psi$ is manifest since it coincides with the pointwise limit of a sequence of continuous approximants
    \[
    \inpr{\nabla\psi(x),z}
    =
    \lim_{n\to\infty} n \brbr{ \psi(x+z/n) - \psi(x) },
    \]
    so that $\dom\nabla\psi$ belongs to $\borel(\Rdmz)$.
    Since $\partial\psi$ is closed and contains $\spt\gamma$, the set $S $ lies in $\borel(\RdRdmz)$ and contains the set 
    \[ 
    \tilde S = \rbr{\rbr{\dom\nabla\psi\cap\Rdmz} \times\Rd} \cap \spt\gamma.
    \]
    The latter set satisfies 
    \begin{equation}
    \label{eq:gammatSc0}
        \gamma \brbr{\tilde S^c \cap \rbr{\Rdmz\times\Rd}} 
        = \tilde\mu \rbr{(\dom\nabla\psi)^c} 
        = 0. 
    \end{equation}
    Indeed, by \cref{lem:supp-zero-proj} and \cref{eq:domains}, $\spt\mu\subset\cl(\dom\partial\psi) = \cl(\dom\psi)$. The latter closure is convex and its boundary has therefore Hausdorff dimension at most $d-1$.
    Moreover, a convex function is differentiable in the interior of its domain except perhaps at a set of Hausdorff dimension at most $d-1$, see \citet[(F) and (3.2)]{AndersonKlee1952}. It follows that $\tilde\mu((\dom\nabla\psi)^c) = \tilde\mu(\cl(\dom\psi)\sauf\dom\nabla\psi) = 0$, confirming \cref{eq:gammatSc0}.
    
    Thus, for every $A\in\borel(\Rdmz)$ and $B\in\borel(\Rd)$ we have
    \begin{align*}
        \gamma(A\times B) 
        = \gamma((A\times B)\cap S)
        &= \gamma\rbr{\rbr{A\cap(\nabla\psi)^{-1}(B)}\times\Rd} \\
        &= \mu\rbr{A\cap(\nabla\psi)^{-1}(B)}
        = (\Id\times\nabla\psi)_{\#}\mu(A\times B)
            .
    \end{align*}
    By the monotone class theorem, this identity uniquely determines the restriction of $\gamma$ to $\Rdmz \times \Rd$ as being equal to the restriction of $(\Id\times\nabla\psi)_{\#}\mu$ to $\Rdmz \times \Rd$.
    
    \medskip{}
    
    \noindent\emph{Restriction of $\gamma$ to $\{0\}\times\Rdmz$.} --- 
    We now deal with subsets of $\{0\}\times\Rdmz$.
    Let $B\in\borel(\Rdmz)$. 
    Since $\nu$ is the right zero-margin of $\gamma$, we can write successively
    \begin{align*}
    \nu(B) 
    &= \gamma(\Rd\times B) \\
    &= \gamma(\Rdmz\times B) + \gamma(\{0\}\times B) \\
    &= (\Id\times\nabla\psi)_\#\mu(\Rdmz\times B) + \gamma(\{0\}\times B),
    \end{align*}
    where we used the representation of $\gamma$ on $\Rdmz \times \Rd$ in the last equality.
    This proves the first assertion of the statement.
    
    Assume now that $0\notin\cl B$. Since $B$ and $\Rdmz\times B$ are bounded away from the origin in $\Rd$ and $\RdRd$, respectively, both $\nu(B)$ and $\gamma(\Rdmz\times B)$ are finite. Hence we have 
    \begin{equation} 
    \label{eq:gam0B}
    \gamma(\{0\}\times B) 
    = \nu(B) - (\Id\times\nabla\psi)_\#\mu(\Rdmz\times B). 
    \end{equation}
    This uniquely determines $\gamma$ on $\cbr{0} \times \Rdmz$, since for any $A \subset \borel(\Rdmz)$ we can write $\gamma(\cbr{0} \times A) = \lim_{n\to\infty} \gamma(\cbr{0} \times (A \setminus \ball_{0,1/n}))$.

    \medskip{}

    \noindent\emph{Definition of $\gamma$ on $\RdRdmz$.} ---
    Any Borel set $E$ of $\RdRdmz$ can be partitioned into two pieces, $E \cap (\Rdmz \times \Rd)$ and $E \cap (\cbr{0} \times \Rdmz)$. The mass that $\gamma$ assigns to each of these pieces is uniquely determined by the representations in the two previous paragraphs. 

    \medskip{}

    \noindent\emph{Proper representation.} ---
    Assume $\gamma\in\zpcoup(\mu,\nu)$, so $\gamma(\cbr{0} \times \Rdmz) = 0$, i.e., the restriction of $\gamma$ to $\{0\}\times\Rdmz$ is the zero measure. The same is true for $(\Id \times \nabla\psi)_{\#}\mu$, since $\mu$ is only defined on $\Rdmz$ and thus, by definition,
    \[
        (\Id \times \nabla \psi)_{\#}\mu\rbr{\cbr{0} \times \Rd}
        = \mu \rbr{\Rdmz \cap \rbr{\cbr{0} \times \dom \psi}}
        = \mu (\emptyset) = 0.
    \]
    Moreover, we already know that the measures $\gamma$ and $(\Id \times \nabla\psi)_{\#}\mu$ coincide on $\Rdmz \times \Rd$. Hence, both measures coincide on the union of $\cbr{0} \times \Rdmz$ and $\Rdmz \times \Rd$, which is $\RdRdmz$, and therefore $\gamma = \res \rbr{(\Id \times \nabla\psi)_{\#}\mu}$.

    We use the latter identity to find a representation of $\nu$, the right zero-margin of $\gamma$. For any $B \in \borel(\Rdmz)$, we have
    \[
        \nu(B) 
        = \gamma(\Rd \times B)
        = (\Id \times \nabla\psi)_{\#}\mu(B)
        = \mu \rbr{\Rdmz \cap (\nabla\psi)^{-1}(B)}
        = (\nabla\psi)_{\#}\mu(B).
    \]
    It finally follows that $\nu = \res ((\nabla \psi)_{\#}\mu)$, as announced.
\end{proof}

\section{Proper zero-couplings and representations}
\label{sec:proper-zcoup}





The interest in the proper zero-coupling and the proper representation in \cref{eq:proper} is that we can write the ``target measure'' $\nu \in \Mz(\Rd)$ as a push-forward of the ``reference measure'' $\mu \in \Mz(\Rd)$, i.e., $\nu = \res((\nabla\psi)_{\#}\mu)$, or, more explicitly,
\[
    \forall B \in \borel(\Rdmz): \qquad
    \nu(B) 
    = (\nabla\psi)_{\#}\mu(B) 
    = \mu\rbr{(\nabla\psi)^{-1}(B) \sauf \cbr{0}},
\]
see also the comment after \cref{thm:expr_zcmcoup}. In such a case, $\nabla\psi$ is indeed a transport map pushing forward the reference $\mu$ to the target $\nu$. This opens a way towards optimal-transport based statistical methodology without moment conditions as reviewed in \cite{Hallin2022} but then for Lévy measures and for certain infinite measures arising in multivariate regular variation, as developed in \cref{sec:EVTOT} below.


For finite measures $\mu$ and $\nu$, a sufficient criterion for the existence of such a push-forward representation is that $\mu$ is sufficiently smooth and has at least as much mass as $\nu$. This follows from the Brenier--McCann Theorem \cite{brenier1987decomposition,brenier1991polar,McCann1995}. 

\begin{proposition}
    \label{prop:existpropcoupfinite}
    For $\mu,\nu \in \Mz(\Rd)$ with $\mu$ vanishing on sets of Hausdorff dimension at most $d-1$, assume that $\nu(\Rdmz) \le \mu(\Rdmz) < \infty$. Then there exists a closed convex function $\psi:\Rd\to\rset\cup\cbr{0}$ satisfying $\nu = \res((\nabla \psi)_{\#}\mu)$. Moreover, $\mu(\Rdmz \sauf \dom \nabla\psi) = 0$ and the measure $\gamma = \res((\Id \times \nabla\psi)_{\#}\mu)$ belongs to $\zpcoup(\mu,\nu)$. 
\end{proposition}

\begin{remark}
    The condition $\nu(\Rdmz) \le \mu(\Rdmz)$ in \cref{prop:existpropcoupfinite} is actually necessary for the desired $\psi$ to exist, even when no assumption is made on $\mu$. Indeed, if $\nu = \res((\nabla\psi)_{\#}\mu)$, then $\nu(\Rdmz) = (\nabla\psi)_{\#}\mu(\Rdmz) = \mu((\nabla\psi)^{-1}(\Rdmz) \sauf \cbr{0}) \le \mu(\Rdmz)$. 
\end{remark}

In case $\mu$ and $\nu$ have infinite mass, 
the situation is more complicated. The problem is that the origin acts as a kind of barrier across which a (cyclically) monotone transport plan cannot move an infinite amount of mass, 
if the measures are to assign finite mass to sets bounded away from the origin. Example~\ref{c-ex:existence_push_forward} illustrates this situation in one dimension, when the measures $\mu$ and $\nu$ put all their mass on opposite sides of the origin.

Recall the definition of a (cyclically) monotone subset of $\RdRd$ in \cref{eq:monotone,rel:cyclical_mononicity0} and the collections of proper zero-couplings $\zpmcoup(\mu,\nu)$ and $\zpcoup(\mu,\nu)$ in \cref{def:proper} with (cyclically) monotone support. By the following lemma, the existence of a proper zero-coupling with monotone support between non-finite measures $\mu$ and $\nu$ implies that (almost) any point in the support of $\nu$ can be ``seen'' from a point in the support of $\mu$.

\begin{lemma}[Necessary condition for proper, monotone zero-couplings]
    \label{lem:necessary-proper}
    Let $\mu, \nu \in \Mz(\Rd)$ and suppose at least one of them has infinite mass. If there exists $\gamma \in \zpmcoup(\mu,\nu)$, 
    then the following two properties hold:
    \begin{itemize}
        \item 
        for $\nu$-almost all $y\in \spt \nu$, there exists $x \in \spt \mu \sauf \cbr{0}$ such that $\inpr{x,y} \ge 0$;
        \item 
        for all $y \in \spt \nu$, we have $\sup_{x\in\spt\mu\sauf\cbr{0}} \inpr{x/\nbr{x},y} \ge 0$.
    \end{itemize}
\end{lemma}


In the rest of this section we will make the stronger assumption
\begin{equation}
    \label{necessary_condition}
    \forall y \in \spt\nu \sauf \cbr{0}:
    \exists x \in \spt\mu: \qquad
     \inpr{x,y}>0.
\end{equation} 
We will seek additional conditions under which a zero-coupling with monotone support is proper. A first such condition is that the support of the zero-coupling is in a certain sense homogeneous. This situation will naturally arise for zero-couplings between limit measures in multivariate regular variation, as in \cref{main_thm_2} below.

\begin{proposition}
    \label{cor:cdtn_push_forward}
    Let $\mu, \nu \in \Mz(\Rd)$ and assume \cref{necessary_condition}. Then $\gamma \in \zmcoup(\mu,\nu)$ is proper as soon as its support is homogeneous in the sense that
    \begin{equation}
        \label{homogenous_support}
        \exists \alpha,\beta > 0: 
        \forall \lambda>0: \: 
        (x,y)\in \spt\gamma 
            \implies (\lambda^\alpha x, \lambda^\beta y)\in \spt\gamma
        .
    \end{equation} 
\end{proposition}

The homogeneity assumption \eqref{homogenous_support} made on the support of $\gamma$ also implies an assumption on the supports of $\mu$ and $\nu$. Moreover, this criterion does not allow one to prove directly that any (cyclically) monotone zero-coupling is proper, because of the additional assumption on its support. From now on we will only make assumptions on the zero-marginals $\mu,\nu$ to derive the desired result.


According to the next result, a zero-coupling with monotone support is proper as soon as $\mu$ has sufficient mass in the directions that matter for $\nu$. For $b \in \Rdmz$ and $\eps > 0$, consider the cone
\begin{equation}
    \label{eq:Hpl}
    H_+(b,\eps)
    = \bigl\{x\in\Rdmz \mid \inpr{x/\nbr{x}, b} > \eps\bigr\}.
\end{equation}

\begin{proposition}
    \label{prop:cdtn-proper-zcoup}
    Let $\mu, \nu \in \Mz(\Rd)$ and assume
    \begin{equation}
        \label{eq:cond_cone}
    \forall y \in \spt\nu\sauf\cbr{0}, 
    \exists \eps >0: \qquad
    \mu\rbr{H_+(y,\eps)}=\infty
    .
    \end{equation}
    Then every $\gamma \in \zmcoup(\mu,\nu)$ is proper, i.e., $\zmcoup(\mu,\nu) = \zpmcoup(\mu,\nu)$, and therefore also $\zcmcoup(\mu,\nu) = \zpcoup(\mu,\nu)$. 
\end{proposition}

Note that the assumption in \cref{eq:cond_cone} is stronger than the one in \cref{necessary_condition}. Condition~\eqref{eq:cond_cone} is trivially satisfied if $\mu$ has infinite mass and is spherically symmetric.

We define the \emph{positive hull} of a set $C\subset\Rd$ as the smallest cone containing it, i.e.
\[
    \pos(C)
    = 
    \cbr{
        \lambda x \mid x\in C, \lambda \ge 0
    }
    .
\]

\begin{corollary}
    \label{cor:cdtn-proper-zcoup}
    Let $\mu, \nu \in \Mz(\Rd)$ and assume \cref{necessary_condition}. Let $\gamma \in \zmcoup(\mu,\nu)$. 
    Then $\gamma$ is proper as soon as at least one of the two following properties holds: 
    \begin{enumerate}[label=(\roman*)]
        \item $\forall x \in\spt\mu, \forall \delta>0 : \mu\rbr{\pos(\ball(x,\delta))\sauf\cbr{0}}=\infty$;
        \item $\exists \alpha > 0, \forall \lambda > 0 : \mu(\lambda \point) = \lambda^{-\alpha} \mu(\point)$.
    \end{enumerate}
\end{corollary}

We conclude this section by a simple sufficient criterion in the one-dimensional case.

\begin{proposition}
    \label{prop:existpropcoup1D}
    Let $\mu, \nu \in \Mz(\rset)$, and assume
    \[
    \mu(\rset_{>0}) \ge \nu(\rset_{>0})
    \quad \text{ and } \quad
    \mu(\rset_{<0}) \ge \nu(\rset_{<0}).
    \]
    If $\gamma \in \zcoup(\mu,\nu)$ has monotone support, then $\gamma$ is proper.
   Conversely, if at least one of $\mu$ or $\nu$ has infinite mass, the existence of $\gamma \in \zpmcoup(\mu,\nu)$ implies that the above inequalities are both satisfied. 
\end{proposition}

It is unclear if \cref{prop:existpropcoup1D} admits an extension to the $d$-dimensional case in terms of masses assigned to half-spaces of the form $\cbr{x \in \Rd \mid \inpr{x,b} > 0}$.


\subsection*{Proofs}

\begin{proof}[Proof of \cref{prop:existpropcoupfinite}]
    If $\mu(\Rdmz) = 0$, both $\mu$ and $\nu$ are the null measure and there is nothing to prove, so assume $0 < \mu(\Rdmz) < \infty$.
    Consider the extensions $\tilde\mu\in[\mu], \bar\nu\in[\nu]$ of $\mu,\nu$ to $\Rd$ given by $\tilde\mu(\cbr{0})=0$, $\bar\nu(\cbr{0}) = \mu(\Rdmz)-\nu(\Rdmz)$, respectively. Since $\nu(\Rdmz) \le \mu(\Rdmz)$ and since both are finite, $\bar\nu(\cbr{0})$ is nonnegative and finite, and $\tilde\mu(\Rd) = \bar\nu(\Rd)$. By assumption, $\tilde\mu$ vanishes on sets of Hausdorff dimension at most $d-1$, allowing one to apply \citet[Proof of existence, page 317]{McCann1995} to the probability measures $m^{-1}\tilde\mu$ and $m^{-1}\bar\nu$, where $m = \tilde\mu(\Rd) = \bar\nu(\Rd)$. This yields the existence of a closed convex function $\psi:\Rd\to\rset\cup\cbr{\infty}$ 
    satisfying $\bar\nu = (\nabla \psi)_{\#}\tilde{\mu} = (\nabla \psi)_{\#}\mu$. Apply the operator $\res$ (restricting to $\Rdmz$) on both sides of the equation to find that $\nu = \res \rbr{(\nabla \psi)_{\#}\mu}$. 

    The measure $\mu$ puts no mass on the complement of $\dom \nabla \psi$ since
    \begin{align*}
        \mu(\Rdmz) 
        = \tilde\mu(\Rd) 
        &= \bar\nu(\Rd) \\
        &= (\nabla \psi)_{\#}\mu(\Rd) 
        = \mu\rbr{\Rdmz \cap (\nabla\psi)^{-1}(\Rd)}
        = \mu\rbr{\Rdmz \cap \dom \nabla\psi},
    \end{align*}
    and both sides of the equation are finite. 
    
    Finally, we show that the measure $\gamma = \res \brbr{\rbr{\Id \times \nabla \psi}_{\#}\mu}$ is a proper cyclically monotone zero-coupling of $\mu$ and $\nu$. First, $\gamma$ is proper since, by definition,
    \begin{align*}
        \gamma\rbr{\cbr{0} \times \Rdmz}
        &= \rbr{\Id \times \nabla \psi}_{\#}\mu \rbr{\cbr{0} \times \Rdmz} \\
        &= \mu\rbr{\Rdmz \cap \rbr{\cbr{0} \cap (\nabla\psi)^{-1}(\Rdmz)}}
        = \mu(\emptyset)
        = 0.
    \end{align*}
    Second, for $A \in \borel(\Rdmz)$ and $B \in \borel(\Rd)$, we have
    \[
        \gamma(A \times B) 
        = (\Id \times \nabla\psi)_{\#}\mu(A \times B)
        = \mu \rbr{A \cap (\nabla\psi)^{-1}(B)}.
    \]
    On the one hand, putting $B = \Rd$, we find 
    \[ 
        \gamma(A \times \Rd) 
        = \mu \rbr{A \cap \dom \nabla\psi} 
        = \mu(A), 
    \] 
    so $\varphi_1 \gamma = \mu$. On the other hand, putting $A = \Rdmz$, we find, if $B \subset \Rdmz$,
    \[
        \gamma(\Rd \times B)
        = \gamma(\Rdmz \times B)
        = \mu\rbr{\Rdmz \cap (\nabla\psi)^{-1}(B)}
        = \nu(B).
    \]
    and thus $\varphi_2 \gamma = \nu$.
\end{proof}

\begin{proof}[Proof of \cref{lem:necessary-proper}]
    \emph{Proof of (i)}. ---
    Define $B$ by removing the ``vertical axis'' $\cbr{0} \times \Rd$ from $\spt \gamma$ and projecting the difference set onto the second coordinate and removing the origin, i.e.,
    \[ 
        B = \proj_2\rbr{\spt \gamma \sauf (\cbr{0} \times \Rd)} \sauf \cbr{0},
    \]
    where $\proj_2(x,y) = y$ for $(x,y) \in \RdRd$.
    Clearly, $B \subset \Rdmz$. Further, $B$ is a Borel set because $\proj_2$ is continuous and $\spt \gamma \sauf (\cbr{0} \times \Rd)$ is a countable union of compact sets, 
    specifically $\{(x,y) \in \spt\gamma \mid 1/n \le \nbr{x} \le n, \nbr{y} \le n\}$ for $n \in \nat$, since $\spt \gamma$ is closed.

    The measure $\nu$ is concentrated on $B$ because (recall $\gamma$ is a measure on $\reals^{2d}_{\smallsetminus0}$)
    \begin{align*}
        \nu(\Rdmz \sauf B)
        &= \nu\rbr{\Rdmz \sauf \proj_2\rbr{\spt \gamma \sauf (\cbr{0} \times \Rd)}} \\
        &= \gamma \rbr{\Rd \times \sbr{\Rdmz \sauf \proj_2\rbr{\spt \gamma \sauf (\cbr{0} \times \Rd)}}} \\
        &\le \gamma \rbr{\reals^{2d}_{\smallsetminus0} \sauf \rbr{\spt \gamma \sauf (\cbr{0} \times \Rd)}} \\
        &\le \gamma\rbr{\reals^{2d}_{\smallsetminus0} \sauf \spt \gamma}
        + \gamma\rbr{\cbr{0} \times \Rdmz}
        = 0,
    \end{align*}
    where, in the last step, we used the property that $\gamma$ is proper.

    Let $y \in B$. By definition of $B$, there exists $x \in \Rd$ such that $(x, y) \in \spt \gamma \sauf (\cbr{0} \times \Rd)$. This implies both $x \ne 0$ and $x \in \spt \mu$: for any open set $U \subset \Rdmz$ that contains $x$, we have $\mu(U) = \gamma(U \times \Rd) > 0$, as $U \times \Rd$ is an open set that contains $(x,y) \in \spt \gamma$.

    Since at least one of the two measures $\mu$ and $\nu$ has infinite mass, so has $\gamma$. As $\gamma \in \Mz(\RdRd)$, necessarily $\gamma(\ball_{2d}(0,\eps) \sauf (0,0)) = \infty > 0$ for all $\eps > 0$. Therefore, $(0, 0)$ belongs to $\spt \gamma$ too.

    Both $(x,y)$ and $(0,0)$ thus belong to $\spt \gamma$. The monotonicity of the latter set implies $\inpr{x,y} = \inpr{x-0,y-0} \ge 0$, as required.

    \smallskip

    \noindent \emph{Proof of (ii).} ---
    For $y = 0$, there is nothing to show. Assume there exists $y\in\spt\nu\sauf\cbr{0}$ such that 
    \[ 
        \sup_{x\in\spt\mu\sauf\cbr{0}} \inpr{x/\nbr{x},y}<0. 
    \] 
    We will prove that this leads to a contradiction.

    Let $N_y(\epsilon)=\cbr{z\in\Rdmz \mid \nbr{ z/\nbr{z} - y/\nbr{y}} < \epsilon}$ for $\epsilon > 0$. For $x\in\spt\mu\sauf\cbr{0}$ and $z \in N_y(\epsilon)$, the Cauchy--Schwarz inequality yields
    \begin{align*}
        \inpr{x/\nbr{x}, z/\nbr{z}} 
        &= \inpr{x/\nbr{x}, z/\nbr{z}-y/\nbr{y}} + \inpr{x/\nbr{x}, y/\nbr{y}}\\
        &\le \epsilon + \sup_{x\in\spt\mu\sauf\cbr{0}}
     \inpr{x/\nbr{x},y/\nbr{y}}
        .
    \end{align*}
    Taking $\epsilon < \abr{ \sup \cbr{ \inpr{x/\nbr{x},y/\nbr{y}} \mid x\in\spt\mu\sauf\cbr{0}} }$ small enough, we get $\inpr{x,z} < 0$ for all $x\in\spt\mu\sauf\cbr{0}$ and $z \in N_y(\epsilon)$, 
     where $N_y(\epsilon)$ is an open neighbourhood of $y$.

    Since $y\in\spt\nu\sauf\cbr{0}$ and since $\gamma$ is a proper zero-coupling we have
    \[
        0<\nu(N_y(\epsilon))
        =
        \gamma\rbr{\Rd\times N_y(\epsilon)}
        = \gamma\rbr{\spt\gamma \cap (\Rdmz\times N_y(\epsilon))}.
    \]
    The set $\spt\gamma \cap (\Rdmz\times N_y(\epsilon))$ has positive measure and is therefore not empty. Let $(a, b)$ be a point in this set; then $a \ne 0$ and $b \in N_{y}(\epsilon)$. Necessarily $a \in \spt \mu$: for any open neighbourhood $U \subset \Rdmz$ of $a$, we have $\mu(U) = \gamma(U \times \Rd) > 0$, since $U \times \Rd$ is an open neighbourhood of $(a, b) \in \spt \gamma$. By the previous paragraph, we find $\inpr{a,b} < 0$.
    
    However, since at least one of the two measures $\mu$ and $\nu$ has infinite mass, we have $(0, 0) \in \spt \gamma$ by the argument in the beginning of the proof of \cref{lem:Psi-finite}.
    But $\spt \gamma$ is monotone, and thus $\inpr{a,b} = \inpr{a-0,b-0} \ge 0$, in contradiction to the conclusion of the preceding paragraph. 
    
    We conclude that the assumption at the beginning of the proof of (ii) is wrong, which leads us to the desired conclusion.
\end{proof}

\begin{lemma}\label{lm:homo_monot}
    Let $A \subset \RdRd$ satisfy the following two properties:
    \begin{enumerate}[label=(\alph*)]
    \item (monotonicity) 
        $\forall (x,y),(a,b)\in A: \inpr{x-a,y-b} \ge 0$;
    \item (homogeneity) 
        $\exists \alpha,\beta > 0, 
        \forall \lambda>0,
        \forall (x,y)\in A: 
            (\lambda^\alpha x, \lambda^\beta y)\in A$.
    \end{enumerate}
    Then $x \in \cl(\proj_1 A)$ and $(0, b) \in A$ imply $\inpr{x,b} \le 0$.
\end{lemma}

\begin{proof}
    Let $x \in \cl(\proj_1 A)$ and $(0, b) \in A$. We need to show $\inpr{x,b} \le 0$. Let $(x_n,y_n) \in A$ for $n \ge 1$ be such that $x_n \to x$ as $n\to\infty$. If $y_n = 0$ for all $n$ in an infinite subset $N$ of $\nat$, then $\inpr{x_n,-b} = \inpr{x_n-0,0-b} \ge 0$ for $n \in N$ by monotonicity and thus $\inpr{x,b} = -\lim_{n\to\infty,n\in N} \inpr{x_n,-b} \le 0$. Otherwise, let $\lambda > 0$ and put $\lambda_n = \lambda^{1/\beta} \nbr{y_n}^{-1/\beta}$ and $z_n = y_n/\nbr{y_n}$. By homogeneity, $(\lambda_n^\alpha x_n,\lambda_n^\beta y_n)\in A$ and thus, by monotonicity,
    \begin{align*}
        \inpr{\lambda_n^\alpha x_n - 0,\lambda_n^\beta y_n - b} 
        &\ge0,
        \\
        \text{from which} \qquad
        \lambda \inpr{x_n, z_n} = \lambda_n^\beta \inpr{x_n,y_n} &\ge \inpr{ x_n,b}.
    \end{align*}
    If $z$ is an accumulation point of $(z_n)_n$, we get $\lambda \inpr{x, z} \ge \inpr{x, b}$. Since this is true for any $\lambda > 0$, we can take the limit as $\lambda \to 0$ to find $0 \ge \inpr{x, b}$, as required. 
\end{proof}

\begin{proof}[Proof of \cref{cor:cdtn_push_forward}]
    We apply \cref{lm:homo_monot} to $A=\spt\gamma$, a set which is both cyclically monotone and thus monotone as well as homogeneous in the sense of \cref{homogenous_support}. Writing $H_+(x) = \cbr{y \in \Rd \mid \inpr{x,y} > 0}$, we get 
    \[
        \forall x \in \cl(\proj_1 \spt\gamma): \qquad 
        \spt\gamma \cap\rbr{ \{0\} \times H_+(x) } = \emptyset.
    \]
    Since $\spt\mu\subset\cl(\proj_1\spt\gamma)$ by \cref{lem:supp-zero-proj} and thanks to the assumption on the supports of $\mu$ and $\nu$, it follows that 
    \[ 
        \spt\gamma \cap \rbr{\cbr{0}\times\rbr{\spt\nu \sauf \cbr{0}}}
        \subset \spt\gamma \cap \rbr{\cbr{0} \times \bigcup_{x \in \spt\mu} H_+(x)}
        = \emptyset.
    \] 
    As a consequence, $\gamma(\{0\}\times(\spt\nu\sauf\{0\}))=0$. Moreover, since $\gamma$ is a zero-coupling of $\mu$ and $\nu$,
    \begin{align*}
        \gamma\rbr{\cbr{0} \times \rbr{\Rdmz \sauf \spt\nu}}
        &\le \gamma\rbr{\Rd \times \rbr{\Rdmz \sauf \spt\nu}} \\
        &= \nu \rbr{\Rdmz \sauf \spt\nu}
        = 0.
    \end{align*}
    We conclude that $\gamma(\cbr{0} \times \Rdmz) = 0$, i.e., $\gamma$ is a proper zero-coupling. 
\end{proof}


\begin{proof}[Proof of \cref{prop:cdtn-proper-zcoup}]
The goal is to show that under the conditions of the statement, $\gamma(\cbr{0}\times\Rdmz)=0$. A sufficient condition is that $\spt\gamma \cap \rbr{\cbr{0}\times\Rdmz} = \varnothing$.  We proceed by contradiction and assume the existence of $b\in\Rdmz$ such that $(0,b)\in\spt\gamma$. Then also $b\in\spt\nu$, since for any $0<r<\nbr{b}$, we have, by definition of $\spt\gamma$,
\[
    \nu(\ball(b,r))
    = \gamma(\Rd\times\ball(b,r))
    > 0,
\]
since $\Rd\times\ball(b,r)$ is an open set containing $(0, b)$.

From the assumption in the statement, there exists $\eps>0$ such that $\mu(H_+(b,\eps)) = \infty$, with $H_+(b,\eps)$ defined in \cref{eq:Hpl}. For the latter $\eps$, consider the set  
\begin{align*}
    A 
    &= \cbr{ y \in \Rd \mid \exists x_y \in H_+(b,\epsilon): (x_y,y) \in \spt\gamma } \\
    &= \proj_2 \brbr{ \spt\gamma \cap \rbr{H_+(b,\epsilon) \times \Rd}}.
\end{align*}
(Note that if $\spt\gamma$ corresponds to the graph of a mapping $\Rd \to \Rd$, then $A$ is the image of $H_+(b,\eps)$ by the latter mapping.) 
The set $A$ is a Borel set because the projection $\proj_2$ is continuous and $\spt\gamma \cap \rbr{H_+(b,\epsilon) \times \Rd}$ is equal to the countable union of compact sets
\[
    \cbr{
        (x,y) \in \spt \gamma 
        \mid 
        1/n \le \nbr{x} \le n, \,
        \nbr{y} \le n, 
        \inpr{x/\nbr{x}, b} \ge \eps + 1/n
    },
    \qquad n \in \nat.
\]

\smallskip

\noindent\emph{Step 1.} --- We first show that $A \subset \ball(0,\eps)^c$. Let $y \in A$ and let $x_y\in H_+(b,\epsilon)$ be such that $(x_y,y)\in\spt\gamma$. Then by the Cauchy--Schwarz inequality, we have 
\[
    \nbr{y} \ge \inpr{  x_y/\nbr{x_y}, y }.
\]
Now because $(0,b)$ and $(x_y,y)$ both belong to the monotone set $\spt \gamma$, we have  
$\inpr{x_y, y-b} = \inpr{x_y-0,y-b} \ge 0$,  
whence $\inpr{x_y,y} \ge \inpr{x_y,b}$ and thus, from the definition of $H_+(b,\eps)$,  
\[
    \inpr{ x_y /\nbr{x_y}, y } \ge 
    \inpr{ x_y /\nbr{x_y}, b } \ge \eps. 
\]
The latter two displays prove that $\nbr{y} \ge \eps$,  as claimed. 

\smallskip{}

\noindent\emph{Step 2.} ---
Next we show that $\nu(A)= +\infty$, from which the contradiction will follow, since by definition of $\Mzp(\Rd)$ and by Step~1, we should have $\nu(A)\le \nu(\ball(0,\eps)^c) <\infty$. From the definition of $A$, and since $0\notin A$ by Step~1, we have
\begin{align*}
    \nu(A)
    &= \gamma(\Rd \times A) \\
    &\ge \gamma\rbr{\{(x,y) \in \spt\gamma \mid x \in H_+(b,\epsilon), y \in A \}} \\
    &= \gamma\rbr{\{(x,y) \in \spt\gamma \mid x \in H_+(b,\epsilon) \}} \\
    &= \gamma\rbr{ H_+(b,\epsilon) \times \Rd } 
    = \mu\rbr{H_+(b,\epsilon)} = +\infty,
\end{align*}
leading to the stated contradiction.
\end{proof}

\begin{proof}[Proof of \cref{cor:cdtn-proper-zcoup}]
    It suffices to check that the condition in \cref{prop:cdtn-proper-zcoup} is satisfied.
    Let $y \in \spt\nu\sauf\cbr{0}$. By assumption there exists $x \in \spt\mu$ such that $\inpr{x,y}>0$. We can find $\epsilon > 0$ such that $x\in H_+(y,\epsilon)$ and then, by continuity, also $\ball(x,\delta)\subset H_+(y,\epsilon)$ for some $\delta>0$.
    As $x \in \spt\mu$, we have $\mu( \ball(x,\delta) ) > 0$.
    
    \begin{enumerate}[label=(\roman*)]
        \item By assumption, $\mu\rbr{\pos\ball(x,\delta)}=\infty$; note that $0 \not\in \ball(x,\delta)$ by the choice of $\delta$. Since $H_+(y,\epsilon)$ is a cone and $\ball(x,\delta)\subset H_+(y,\epsilon)$, the inclusion $\pos\ball(x,\delta) \subset H_+(y,\epsilon)$ holds too. Therefore, $\mu(H_+(y,\epsilon))=\infty$ and \cref{prop:cdtn-proper-zcoup} can be applied.
        \item This is a particular case of the last point. Let $x \in \spt \mu$ and $\delta> 0$. For any $\lambda>0$, we have $\lambda \ball(x,\delta)\subset\pos\ball(x,\delta)$ and thus $\mu\rbr{\pos\ball(x,\delta)} \ge \mu(\lambda \ball(x, \delta)) = \lambda^{-\alpha}\mu(\ball(x,\delta))$. Since $\mu(\ball(x,\delta))>0$, we find assumption~(i) by taking the limit as $\lambda\to0$.
        \qedhere
    \end{enumerate}
\end{proof}

\begin{proof}[Proof of \cref{prop:existpropcoup1D}]
If $\spt \gamma \cap (\cbr{0} \times \Rmz) = \emptyset$, then $\gamma\rbr{\cbr{0} \times \Rmz} = 0$ and thus $\gamma$ is proper. So assume $\spt \gamma$ intersects $\cbr{0} \times \Rmz$; we will show that still $\gamma\rbr{\cbr{0} \times \rset_{>0}} = \gamma\rbr{\cbr{0} \times \rset_{<0}} = 0$.

    Assume there exists $b>0$ such that $(0,b)\in\spt\gamma$; if such $b$ does not exist, then $\spt \gamma$ does not interset $\cbr{0} \times \R_{>0}$ and thus $\gamma(\cbr{0} \times \rset_{>0}) = 0$, and we can move to the next paragraph. Consider $(x,y)\in\spt\gamma$ with $x>0$;
    by monotonicity of $\spt\gamma$, we have $(x-0)(y-b) \ge 0$, i.e., $y \ge b$. It follows that
    \[
        \spt \gamma \cap \rbr{\rset_{>0} \times \rset}
        \subset
        \rset_{>0} \times [b, \infty).
    \]
    But then $\gamma(\rset_{>0} \times \rset) = \gamma(\rset_{>0} \times [b, \infty))$ and thus
    \begin{align*}
        \nu(\rset_{>0})
        \ge 
        \nu\rbr{ [b,\infty)} 
        &=\gamma\rbr{\rset\times[b,\infty)}
        \\
        &\ge \gamma\rbr{\rset_{>0}\times[b,\infty)} 
        \\
        & =\gamma\rbr{\rset_{>0}\times\rset} 
        = \mu(\rset_{>0}) \ge \nu(\rset_{>0})
        .
    \end{align*}
    As a consequence, the inequalities above become equalities, the common quantity being finite, and thus 
    \begin{align*} 
        \gamma\rbr{\cbr{0}\times[b,\infty)} 
        &\le \gamma\rbr{\rset_{\le 0} \times [b,\infty)} \le 
        \gamma\rbr{\rset \times [b,\infty)} - \gamma\rbr{\rset_{>0} \times [b,\infty)}
        = 0 \qquad \text{and} \\
        \gamma\rbr{\cbr{0}\times(0,b)}
        &\le \gamma\rbr{\reals \times (0,b)}
        = \nu((0,b))
        = \nu(\rset_{>0})-\nu([b,\infty))
        = 0,
    \end{align*}
    yielding $\gamma\rbr{\cbr{0}\times\rset_{>0}}=0$. 

    Similarly, if there does not exist $b < 0$ such that $(0, b) \in \spt \gamma$, then trivially $\gamma(\cbr{0} \times \rset_{<0}) = 0$, while if such a negative number $b$ exists, we still have $\gamma(\cbr{0} \times \rset_{<0}) = 0$ by the same argument as in the previous paragraph.

    Reciprocally, assume there exists $\gamma\in\zpmcoup(\mu,\nu)$ and assume $\mu\rbr{\rset\sauf\cbr{0}}=\infty$ or $\nu\rbr{\rset\sauf\cbr{0}}=\infty$. Then also $\gamma(\rset^2 \sauf \cbr{(0,0)}) = \infty$ and thus $\gamma([-r,r]^2) = \infty > 0$ for every $r > 0$, which implies $(0,0)\in\spt\gamma$. Moreover, as $\spt\gamma$ is monotone, we find that any $(\tilde x,\tilde y)\in\spt\gamma$ satisfies $\tilde x \tilde y = (\tilde x - 0)(\tilde y - 0) \ge 0$, i.e., $\spt \gamma \subset [0, \infty)^2 \cup (-\infty, 0]^2$; intuitively, $\mu$-mass from $\rset_{>0}$ (resp. $\rset_{<0}$) cannot be sent to $\rset_{<0}$ (resp. $\rset_{>0}$) by the zero-coupling $\gamma$. As a consequence, 
    \begin{align*}
        \mu(\rset_{>0}) 
        &= \gamma(\rset_{>0} \times \rset) \\
        &= \gamma\rbr{ \rset_{>0} \times \rset_{\ge0} } \\
        &= \gamma\rbr{ \rset_{\ge0} \times \rset_{\ge0} } \\
        &\ge \gamma\rbr{ \rset_{\ge0} \times \rset_{>0} } = \nu(\rset_{>0}),
    \end{align*}
    where we used that $\gamma$ is proper, i.e., $\gamma(\cbr{0}\times\rset)=0$. Bu the same reasoning, we can also prove $\mu(\rset_{<0}) \ge \nu(\rset_{<0})$.
\end{proof}

\section{Brenier--McCann theorem for zero-couplings}
\label{sec:brenier}


In \citet{McCann1995} it is proven that, for a probability measure $P$ on $\Rd$ vanishing on sets of Hausdorff dimension at most $d-1$ and for closed convex functions $\phi$ and $\psi$, the identity $\rbr{\Id \times \nabla\psi}_{\#}P = \rbr{\Id \times \nabla\phi}_{\#}P$ implies $\nabla\phi=\nabla\psi$ $P$-a.e.
In our setting, we prove that the same conclusion holds under appropriate assumptions when an improper term is added as in the representation from \cref{thm:expr_zcmcoup}.
The proof is inspired by the one of the uniqueness assertion of Main Theorem in \citet{McCann1995}, which relies in turn on a refinement of arguments from \citet{aleksandrov1942}. However, the fact that we are only dealing with zero-couplings and that some quantities may be infinite makes the proof substantially more complicated.
Recall that we write $\res m$ for the restriction of a Borel measure $m$ on $\reals^k$ to $\reals^k \sauf \cbr{0}$.

\begin{theorem}[Uniqueness of zero-couplings]
    \label{thm:unique}
    Let $\mu,\nu\in\Mzp(\Rd)$ and assume $\mu$ vanishes on sets of Hausdorff dimension at most $d-1$. Let $\gamma_1,\gamma_2 \in \zcmcoup(\mu,\nu)$  and let $\psi_1,\psi_2 : \Rd \to \reals \cup \cbr{\infty}$ be closed convex functions such that $\spt \gamma_i \subset \partial \psi_i$ for $i \in \cbr{1,2}$. If $(0, 0) \in \partial \psi_1 \cap \partial \psi_2$, then $\nabla \psi_1 = \nabla \psi_2$ $\mu$-almost everywhere and thus $\gamma_1 = \gamma_2$. In particular, if at least one of $\mu$ or $\nu$ has infinite mass, then $\zcmcoup(\mu,\nu)$  is a singleton.
\end{theorem}

The consequence $\gamma_1=\gamma_2$ stated in \cref{thm:unique} follows from the representation in \cref{thm:expr_zcmcoup}. In the proof of \cref{thm:unique}, we actually show the slightly stronger statement that $\nabla \psi_1(x) = \nabla \psi_2(x)$ for any $x \in \spt \mu$ in the domains of both $\nabla \psi_1$ and $\nabla \psi_2$.

In the proof of \cref{thm:unique}, we cannot proceed directly as in the proof of the Main Theorem in \cite{McCann1995} and take any $p \in \spt\mu$ such that $\nabla \psi_1(p) \ne \nabla \psi_2(p)$ to obtain a contradiction, since the pathological set $M$ obtained in that proof via Aleksandrov's lemma must have finite measure, as otherwise an expression of the form infinity minus infinity would arise. This comes at the price of a long build-up in Steps~1--5 in the proof of Theorem~\ref{thm:unique} before arriving at the said contradiction in Step~6.


\begin{remark}
\label{rk:c-e_uniqueness}
In \cref{thm:unique}, the assumption $(0,0) \in \partial \psi_1 \cap \partial \psi_2$ implies that the convex functions $\psi_1$ and $\psi_2$ attain their global minima at the origin and that $0$ is left invariant by their gradients when they are defined at the origin. Because of the nature of zero-couplings, the assumption is essential for the uniqueness of the gradient to hold.
Indeed, as a counter-example in dimension $d = 1$, let $\mu$ and $\nu$ be the Lebesgue measures on $[-1, 0]$ and $[0, 1]$, respectively. Then $\gamma_1 = \mu \otimes \delta_0 + \delta_0 \otimes \nu$ is a zero-coupling of $\mu$ and $\nu$, and belongs to $\zcmcoup(\mu,\nu)$, but so does the measure $\gamma_2 = (\Id \times \nabla \psi)_{\#} \mu$ where $\nabla \psi(x) = x+1$. In particular, the support of $\gamma_1$ is contained in the cyclically monotone set $(\reals_{\le 0} \times \{0\}) \cup (\{0\} \times \reals_{\ge 0})$, and the one of $\gamma_2$ in the cyclically monotone set $\cbr{(x,x+1): x \in \reals}$; incidentally, note that $\gamma_2$ is a true coupling rather than just a zero-coupling.
\end{remark}

\begin{remark}
    When $\mu$ and $\nu$ have equal finite mass, then their extensions $\tilde\mu \in [\mu]$ and $\tilde\nu \in [\nu]$ to $\Rd$ defined by $\tilde{\mu}(\cbr{0}) = \tilde{\nu}(\cbr{0}) = 0$ have equal finite mass too, and Corollary~14 in \citet{McCann1995} yields a unique cyclically monotone coupling $\pi$ between $\tilde\mu$ and $\tilde\nu$. It is clear that its restriction $\res\pi$ to $\reals^{2d}_{\smallsetminus 0}$ lies in $\zcmcoup(\mu,\nu)$ while there might be infinitely many other zero-couplings whose supports are cyclically monotone. We refer to the example in Remark~\ref{rk:c-e_uniqueness}.
\end{remark}

Theorem~\ref{thm:unique} concerns the uniqueness of zero-couplings with cyclically monotone support and requires a condition on the origin $(0, 0)$ in $\RdRd$. This condition is not needed in the next result, which concerns proper representations.

\begin{theorem}[Uniqueness of proper representations]
    \label{thm:uniquerepr}

    Let $\mu,\nu \in \Mz(\Rd)$ with $\mu$ vanishing on sets of Hausdorff dimension at most $d-1$, and let $\psi_1,\psi_2$ be two closed convex functions such that $\nu = \res((\nabla \psi_j)_{\#}\mu)$ and $\mu(\Rdmz \sauf \dom \nabla\psi_j) = 0$ for $j=1,2$. Then $\nabla \psi_1=\nabla\psi_2$ $\mu$-almost everywhere.    
\end{theorem} 

A combination of Theorems~\ref{thm:existence}, \ref{thm:expr_zcmcoup} and \ref{thm:unique} leads to a result similar to the Main Theorem in \citet{McCann1995}, where no moment assumption is required. 
The major differences are the necessity of at least one of the zero-marginals to have infinite mass, and the addition of an improper term in the representation when there is a need for mass to be sent from the origin outside of the punctured space.


\begin{corollary}[Brenier--McCann theorem for zero-couplings]
    \label{cor:BMcC}
    Let $\mu,\nu\in\Mzp(\Rd)$ be such that at least one of them has infinite mass and suppose $\mu$ vanishes on sets of Hausdorff dimension at most $d-1$.
    Then $\zcmcoup(\mu,\nu)$ is a singleton with single element $\gamma$, say, and there exists a closed convex function $\psi$ on $\Rd$ with $\spt \gamma \subset \partial \psi$ such that
    \begin{align*}
        \gamma(A \times B) 
        &= (\Id \times \nabla\psi)_{\#}\mu(A\times B)
        = \mu \rbr{A \cap (\nabla\psi)^{-1}(B)}, &&
        A \in \borel(\Rdmz), \, B \in \borel(\Rd), \\
        \gamma(\cbr{0} \times B)
        &= \nu(B) - (\nabla\psi)_{\#} \mu(B), &&
        B \in \borel(\Rd), \, 0 \not\in \cl(B).
    \end{align*}
    Although $\psi$ may not be unique, the map $\nabla\psi$ is defined and uniquely determined $\mu$-almost everywhere. If, in addition, $\gamma$ is proper, then
    \begin{equation}
    \label{eq:mu2nu}
        \nu(B) 
        = (\nabla\psi)_{\#}\mu(B) 
        = \mu \rbr{\Rdmz \cap (\nabla\psi)^{-1}(B)},
        \qquad B \in \borel(\Rdmz).
    \end{equation}
\end{corollary}

As we are interested in the existence of a transport map from $\mu$ to $\nu$ as in \cref{eq:mu2nu}, knowing whether the unique cyclically monotone coupling in \cref{cor:BMcC} is proper or not is crucial.
Sufficient criteria have been studied in \cref{sec:proper-zcoup} and will be applied in \cref{sec:EVTOT} to zero-couplings between limit measures in the definition of multivariate regular variation.


\begin{remark}
    \label{rk:convexnotcontinuous}
    A complication in the proof of \cref{thm:unique} is that the convex functions $\psi_i$ are not necessarily continuous at $0 \in \Rd$ relative to their effective domains, $\dom \psi_i = \cbr{\psi_i < \infty}$. In dimension $d=2$, define 
    \[
        \psi(x, y) = 
        \begin{cases} 
            y^4/x^2 & \text{if $x > 0$,} \\ 
            0 & \text{if $(x,y) = (0, 0)$,} \\ 
            \infty & \text{otherwise,} 
        \end{cases} 
    \]
    a variation on the counterexample in \citet[pp.~83--84]{Rockafellar1970}. This function is convex (its Hessian matrix has positive trace and determinant), but is not continuous in $(0, 0)$ relative to its effective domain, since $\psi(y^2, y) = 1$ for all $y > 0$. Still, it may serve as the potential function of a cyclically monotone transport between two infinite measures in $\Mz(\Rd)$.
    Let $\mu$ be the measure on the half-plane $\cbr{(x,y) \mid x > 0}$ with Lebesgue density
    \[
        \mu(\diff(x,y)) = \frac{x^3}{(x^2+y^2)^3} \diff x \, \diff y, \qquad x > 0, y \in \reals.
    \]
    In polar coordinates $(x, y) = (r \cos \theta, r \sin \theta)$, its density is
    \[
        (T_{\mathrm{polar}})_{\#} \mu(\diff(r,\theta))
        = (\cos \theta)^3 \diff \theta \, r^{-2} \diff r,
        \qquad r > 0, -\pi/2 < \theta < \pi/2,
    \]
    where $T_{\mathrm{polar}}(x, y) = (r, \theta)$ are the polar coordinates of $(x, y) \in \reals^2_{\smallsetminus 0}$.
    We find that $\mu \in \Mz(\reals^2)$ and that $\mu$ is a possible limit measure in the definition of multivariate regular variation in \cref{sec:EVTOT} below. 
    The push-forward measure $\nu = (\nabla \psi)_{\#} \mu$ is supported on the opposite half-plane $\cbr{(x,y) \mid x < 0}$. Its densities in Cartesian and polar coordinates can be computed to be
    \begin{align*}
        \nu(\diff(x,y))
        &= \frac{64 \abr{x}^3}{(4x^2+y^2)^3} \diff x \, \diff y, 
        && x < 0, y \in \reals; \\
        (T_{\mathrm{polar}})_{\#} \nu(\diff(r, \theta))
        &= \frac{64 \abr{\cos \theta}^3}{(1+2\cos^2 \theta)^3} \diff \theta \, r^{-2} \diff r,
        && r > 0, \pi/2 < \theta < 3\pi/2.
    \end{align*}
    Since its angular density is bounded, we have $\nu \in \Mz(\reals^2)$ too. 
\end{remark}

\begin{remark}
    While Brenier's Theorem characterizes the unique possible solution to the Kantorovich problem with quadratic distance as cost, \cref{cor:BMcC} provides, when at least one of the zero-marginals has infinite mass, an analogous representation for the unique candidate solution to the variant of this problem used in \citet{GuillenMou2019} to define a transportation metric for Lévy measures.
\end{remark}

\subsection*{Proofs}

We first prove a convenient extension of Aleksandrov's lemma as stated in \cite[Lemma~13]{McCann1995}. Here and below, if $B\subset\Rd$, and if $\psi$ is a convex function $\Rd\to\rset\cup\cbr{\infty}$, 
we define 
\[
    (\partial \psi)^{-1}(B) 
    = \cbr{x\in\Rd \mid \exists y \in B: y\in \partial \psi(x) }
    = \cbr{x\in\Rd \mid \partial \psi(x)\cap B \neq \emptyset }.
\]
\begin{lemma}[Aleksandrov's Lemma, extended]
    \label{lem:aleksandrov-extended}
    Let $\phi,\psi:\Rd\to\rset\cup\cbr{\infty}$ be convex functions, differentiable at some $p \in \Rd$, with  $\nabla\phi(p)\neq\nabla\psi(p)$. 
    Define $M = \{\phi - \phi(p)>\psi - \psi(p)\} \subset \dom\psi$, $X = (\nabla\psi)^{-1}(\partial\phi(M))$ and      $\tilde X = (\partial \psi)^{-1}(\partial\phi(M))$. 
    Then $X \subset \tilde X\subset M$, while $p$ lies at a positive distance from $X$, i.e., there is $r > 0$ such that $X \cap \ball(p,r) = \emptyset$. 
\end{lemma}

\begin{proof}
The inclusion $X\subset\tilde X$ follows immediately from the definitions, since $\nabla\psi(x) = y$ is equivalent to $\partial\psi(x) = \cbr{y}$. 

{\bf a.} In case $\phi(p)=\psi(p)=0$ and $\nabla\psi(p) = 0$, the result follows from the proof of Lemma~13 in \cite{McCann1995} that $X\subset M$. That proof also shows  that $\tilde X\subset M$ upon noting that by definition of $(\partial\psi)^{-1}$, for any $x\in\tilde X$, there indeed exists $y\in\partial \phi(M)$ such that $y\in\partial \psi(x)$. 

{\bf b.} In the general case, define 
    \begin{align*}
        \bar\phi(x) & = \phi(x) - \phi(p) - \langle \nabla\psi(p), x-p  \rangle ,\\
         \bar\psi(x) & = \psi(x) - \psi(p) - \langle \nabla\psi(p), x-p  \rangle.
    \end{align*}
Clearly, $\bar \phi(p) = \bar \psi(p) = 0$. For all $x\in\dom\partial\phi=\dom\partial\bar\phi$,
we have $\partial \bar \phi(x)= \partial \phi(x) - \nabla\psi(p) = \cbr{y - \nabla\psi(p) \mid y \in \partial \phi(x)}$, and for all $x \in \dom \partial\psi = \dom  \partial\bar\psi$, we have $\partial \bar \psi(x) = \partial \psi(x) - \nabla \psi(p)$; in particular, $p\in\dom \nabla \bar \psi $ and $\nabla\bar\psi(p) = 0$. Applying part~{\bf a.} to  the functions $\bar\phi,\bar\psi$ at point $p$, with  $M^\prime = \{\bar \phi > \bar \psi\}$, $ \tilde X^\prime =  (\partial\bar\psi)^{-1}(\partial \bar \phi(M^\prime))$ and 
$X^\prime = (\nabla\bar\psi)^{-1}(\partial \bar \phi(M^\prime))$,  
we obtain the inclusions $ X^\prime \subset \tilde X^\prime \subset M^\prime$ and the fact that $p$ is at positive distance from $X^\prime$. From the definitions, we have
\[
    M^\prime 
    = \cbr{ x \in \Rd \mid \phi(x) - \phi(p) > \psi(x) - \psi(p) }
    = M, 
\]
and then
\begin{align*}
    \tilde X^\prime 
    &= \cbr{x \in \dom \partial\bar \psi \mid 
    \partial \bar \psi(x) \cap \partial\bar\phi(M^\prime) \ne \emptyset} \\
    & = \cbr{x \in \dom \partial \psi \mid 
    \rbr{\partial \psi(x) - \nabla \psi(p)} \cap \rbr{\partial\phi(M) - \nabla \psi(p)} \ne \emptyset } \\
    & = \cbr{x \in \dom \partial \psi \mid 
        \partial \psi(x) \cap \partial\phi( M) \ne \emptyset} 
    = \tilde X. 
\end{align*}
A similar reasoning yields $X^\prime = X$.
The desired result then follows. 
\end{proof}

\begin{proof}[Proof of~\cref{thm:unique}]
For convenience, we will change notation and let the two convex functions $\psi_1$ and $\psi_2$ be denoted by $\psi$ and $\phi$ instead, and the two zero-couplings $\gamma_1$ and $\gamma_2$ with supports contained in their subdifferentials by $\gamma_\psi$ and $\gamma_\phi$, respectively.

\medskip{}

\noindent\emph{Proof of $\gamma_\psi=\gamma_\phi$}

\medskip{}

\noindent\emph{Step 1.} ---
By \cref{thm:expr_zcmcoup}, $\psi$ and $\phi$ are differentiable $\mu$-almost everywhere, while $\supp\mu$ is contained in $\cl(\dom\psi)\cap\cl(\dom\phi)$, both closures being convex and therefore having boundaries of Hausdorff dimension at most $d-1$.
As we assumed $(0,0) \in \partial\psi\cap\partial\phi$, the functions $\psi$ and $\phi$ attain their minima at $0$, so that $\psi(0)$ and $\phi(0)$ are certainly finite. Therefore, without losing generality, take $\psi(0)= \phi(0)= 0$.
Furthermore, we can restrict attention to the interiors of the domains of $\psi$ and $\phi$ via
\begin{equation}
\label{eq:calD}
	\mathcal{D} := \intr(\dom\psi) \cap \intr(\dom\phi) = \intr(\dom\psi\cap\dom\phi),
\end{equation}
since the above argument shows that $\tilde \mu(\Rd\setminus \mathcal{D}) = 0$, where, as before, $\tilde{\mu}$ is the extension of $\mu$ to $\Rd$ determined by $\tilde{\mu}(\cbr{0})=0$.
Consider the following subsets of $\mathcal{D}$ [see~\cref{eq:domains}]:
\begin{align}
\label{eq:domnablas}
V  & :=\dom\nabla\psi \cap \dom\nabla\phi, \\
\nonumber
W  & :=\left\{ y\in V \mid \nabla\psi(y)\neq\nabla\phi(y)\right\} ,\\
\label{eq:psiphidiff0}
A  & :=\left\{ y\in\mathcal{D} \mid \psi(y)-\phi(y)\neq 0 \right\} .
\end{align}

By \citet[Theorem~2.2]{AlbertiAmbrosio1999} (see also \citep{AndersonKlee1952}),
$\mathcal{D}\sauf V$ is contained in the union of two sets with Hausdorff dimension at most $d-1$, so 
\begin{equation}
	\label{eq:muDV}
	\tilde\mu(\mathcal{D}\sauf V) = 0 \quad \text{ and thus } \quad \tilde\mu(\Rd\sauf V) = 0
    .
\end{equation}


Our claim  that $\nabla\psi = \nabla\phi$ $\mu$-almost everywhere is thus equivalent to $\tilde \mu(W)=0$.

We first  prove  the claim made below the statement of the theorem, namely we show that $\tilde \mu(W)=0$ if and only if for all $x \in \spt \mu\cap V$ we have $\nabla\psi(x) = \nabla \phi(x)$, or in other words, if and only if 
\begin{equation}
    \label{eq:W_cap_sptmu_is_empty}
    W\cap \spt \mu = \emptyset.
\end{equation}
Of course if \cref{eq:W_cap_sptmu_is_empty} holds true then $\tilde\mu(W)=0$. Conversely, if the intersection in \cref{eq:W_cap_sptmu_is_empty} is non-empty, say there exists $x\in W\cap\spt \mu$, then by continuity of the gradients on their domain and the definition of $\spt\mu$, there exists an open ball $\ball(x,r)$ such that $\ball(x,r)\cap V \subset W$ and $\tilde \mu(\ball(x,r))>0$. But then $\tilde \mu(\ball(x,r)\cap V) = \tilde \mu(\ball(x,r))>0$ and thus $\tilde \mu(W)>0$.

To prove that $\tilde\mu(W)=0$, we assume the opposite,
\begin{equation}
\label{eq:wrongassumption}
	\tilde\mu(W) > 0,
\end{equation}
and we will construct a subset $M$ of $\Rdmz$ such that $\Rd \times M$ receives different masses from $\gamma_\psi$ and $\gamma_\varphi$, in contradiction to the assumption of the theorem that both measures are a zero-coupling of $\mu$ and $\nu$.

\medskip{}

\noindent\emph{Step 2.} ---
By Lemma~\ref{lem:smallset}, 
the set
\[
	W \sauf A = \cbr{ y\in V \mid 
		\psi(y) - \phi(y) = 0 \text{ and } \nabla\psi(y)\ne\nabla\phi(y)
	}
\]
has Hausdorff dimension at most $d-1$.
Therefore, $\tilde\mu(W\sauf A)=0$. By \eqref{eq:wrongassumption}, we obtain 
\begin{equation}
\label{eq:wrongassumption2}
	\tilde\mu(W \cap A) > 0. 
\end{equation}

\medskip{}

\noindent\emph{Step 3.} --- 
Recall $\mathcal{D}$ in \eqref{eq:calD}. For $t\in[-\infty, +\infty]$, define
\begin{align*}
A_{t}^{-} &:= \cbr{ y\in\mathcal{D} \mid \psi(y)-\phi(y)<t } ,\\
A_{t}^{+} &:= \cbr{ y\in\mathcal{D} \mid \psi(y)-\phi(y)>t } ,\\
A_{t}^{0} &:= \cbr{ y\in\mathcal{D} \mid \psi(y)-\phi(y)=t } .
\end{align*}
Then $A_{0}^{-}\cup A_0^{+}$ is a partition of $A$ in \eqref{eq:psiphidiff0} and thus, by \eqref{eq:wrongassumption2},
\[ 
	\tilde\mu(W\cap A_{0}^{-})+\tilde\mu(W\cap A_{0}^{+})=\tilde\mu(W\cap A)>0. 
\]
Therefore, $\tilde\mu(W\cap A_{0}^{-})>0$ or $\tilde\mu(W\cap A_{0}^{+})>0$. Without loss of generality, assume
\begin{equation}
\label{eq:wrongassumption3}
	\tilde\mu(W\cap A_{0}^{-}) > 0.
\end{equation}
In the other case, just switch the roles of $\psi$ and $\phi$ in the next steps.
\medskip{}

\noindent\emph{Step 4.} --- 
The goal of this step is to show that for all  $t< 0$, we have $\mu(A_t^-)<\infty$.

Fix $t<0$.  Since $(0,0) \in \partial\psi\cap\partial\phi$, we were able to choose $\psi(0)=\phi(0)=0$ in Step~1.
Consider the set
\[
    N= \cbr{x \in \dom \nabla \phi \sauf \{0\} \mid \nbr{\nabla \phi(x)} \ge 1}.
\]
From Theorem~\ref{thm:expr_zcmcoup} we have for every Borel set $C$ bounded away from $0$,
\[ 
    \infty 
    > \nu(C) 
    = (\nabla\phi)_\#\mu(C) + \gamma(\{0\}\times C)
    \ge (\nabla\phi)_\#\mu(C), 
\] 
since $\nu\in\Mz(\Rd)$. With $C= \ball(0,1)^c$, we obtain 
\[
    \infty > (\nabla\phi)_\#\mu \rbr{\ball(0,1)^c}
    = \mu\rbr{\{ x \in \Rdmz \cap \dom\nabla\phi \mid \|\nabla\phi(x)\| \ge 1 \}} =  \mu(N). 
\]
Hence $\mu(N) <\infty$. Letting 
\[ 
    S = \dom \nabla \phi  \setminus (N\cup\{0\}), 
\]
we have, since $\Rdmz \sauf S = \rbr{\Rdmz \sauf \dom \nabla \phi} \cup N$ and in view of \cref{eq:muDV}, 
\begin{equation}
    \label{eq:good_origin_neighbourhood_2}
    S\subset \dom \nabla \phi \subset \dom \phi,   ~~ \mu(\Rdmz \sauf S)<\infty,  \text{ and }
    \nbr{\nabla \phi(x)} < 1 \text{ for all } x \in S.
 \end{equation}
For any $x \in S$, it thus holds that
\begin{align*}
     0 = \phi(0) &\ge \phi(x) + \inpr{0-x,\nabla\phi(x)}, 
\end{align*}
and therefore also
\[
    \phi(x) \le \langle x, \nabla\phi(x)\rangle \le \|x\|.
\]

We may now show that $\mu(A_t^-) < \infty$.  
Since $\mu(\Rdmz \sauf S)<\infty$ by~\cref{eq:good_origin_neighbourhood_2}, it suffices to check that $\mu(S\cap A_t^-)<\infty$.  But from the above display and the definition of $A_t^-$ [if $x\in A_t^-$, then $\phi(x) > \psi(x) + |t| \ge |t| > 0$], we have the inclusions
\begin{align*}
    S \cap A_t^-
    &\subset \cbr{ x \in \dom \phi \mid \phi(x) > |t| \text{ and } \phi(x) \le \nbr{x} } \\
    &\subset \{x \in \Rd \mid \nbr{x} \ge |t| \} = \ball(0,|t|)^c. 
\end{align*}
Hence, $0\notin \cl (S \cap A_t^- )$ and thus $\mu( S \cap  A_t^-) <\infty$, as required. 

\medskip{}

\noindent\emph{Step 5.} --- 
Recall $V$ in \eqref{eq:domnablas}. We claim that there exists $z \in V$ with the following three properties:
\begin{itemize}
	\item $t := \psi(z) - \phi(z) < 0$;
	\item $\nabla{\psi}(z) \ne \nabla{\phi}(z)$;
	\item for all $\delta > 0$,
	\begin{equation} 
	\label{eq:mu_pos}
	\mu \rbr{ A_{t}^{-} \cap  \ball(z, \delta) } > 0.
	\end{equation}
\end{itemize}

To prove the claim, define, for every $t < 0$, the sets
\begin{align*}
	\Omega_{t}
	&:=
	A_{t}^{0}\cap W \cap\supp\mu \\
	&\phantom{:}=
	\{ x \in V \mid \psi(x) - \phi(x) = t \text{ and } \nabla\psi(x) \ne \nabla\phi(x) \} \cap \supp \mu, \\
	\Xi_{t}
	&:=
	\{ x \in \Omega_{t} \mid \exists \delta > 0, \mu ( A_{t}^{-} \cap \ball(x, \delta) ) = 0 \}.
\end{align*}
We need to show that there exists $t < 0$ such that $\Omega_{t} \sauf \Xi_{t} \neq \emptyset$. To this end, define
\begin{align*}
	\mathcal{T}_{\Omega} &:= \{ t \in (-\infty, 0) \mid \Omega_{t} \neq \emptyset \}, \\
	\mathcal{T}_{\Xi} &:= \{ t \in (-\infty, 0) \mid \Xi_{t} \neq \emptyset \}.
\end{align*}
We will show that $\mathcal{T}_{\Omega}$ is uncountable while $\mathcal{T}_{\Xi}$ is at most countable, so that there must exist (uncountable many) $t < 0$ such that $\Omega_{t}$ is nonempty and $\Xi_{t}$ is empty.

By Lemma~\ref{lem:smallset}, 
for every $t < 0$, the set $\Omega_{t}$ has Hausdorff measure at most $d-1$ and is therefore a $\mu$-null set. Still, their union over all $t$ less than $0$ is
\[
	\textstyle\bigcup_{t < 0} \Omega_t
	=
	\textstyle\bigcup_{t < 0} A_{t}^{0} \cap W \cap \supp \mu
	=
	A_{0}^{-} \cap W \cap \supp \mu,
\]
a set receiving positive mass from $\mu$ by \eqref{eq:wrongassumption3}. Therefore, $\mathcal{T}_\Omega$ must be uncountable.

We have $\Xi_{t} = \bigcup_{n \in \NN} \Xi_{t,n}$ and thus $\mathcal{T}_{\Xi} = \bigcup_{n \in \NN} \mathcal{T}_{n}$, where, for $n \in \NN$,
\begin{align*}
	\Xi_{t,n}
	&:=
	\{ x \in \Omega_{t} \mid \mu ( A_{t}^{-} \cap \ball(x, 2/n) ) = 0 \}, \\
	\mathcal{T}_{n} 
	&:= \{ t \in (-\infty, 0)  \mid \Xi_{t,n} \neq \emptyset \}.
\end{align*}
To show that $\mathcal{T}_{\Xi}$ is countable, it is sufficient to show that each $\mathcal{T}_{n}$ is countable.

Define
\[
\forall t \in \mathcal{T}_{n}, \qquad
G_t := {\textstyle\bigcup}_{x \in \Xi_{t,n}} \ball(x, 1/n).
\]
Each $G_t$ is non-empty and open, and by the argument just given, $G_s \cap G_t = \emptyset$ whenever $s, t \in \mathcal{T}_{n}$ with $s < t$. Since $\Rd$ is separable, the cardinality of the collection of sets $\cbr{G_t}_{t \in \mathcal{T}_n}$ is at most countable, 
and therefore $\mathcal{T}_n$ itself is at most countable, as required. The claim made at the beginning of this step is proved.

\medskip{}

\emph{Step 6.} --- 
Let $t<0$,  and $z\in V$ from Step 5, such that $t = \psi(z) - \phi(z) < 0$.
Define 
\begin{align}
    M  &= \{ y \in \Rd \mid \psi(y) - \phi(y) < t \}  = \{ y \in \Rd \mid \psi(y) - \psi(z) < \phi(y) - \psi(z) \}, 
    \label{eq:MM_clean}
\end{align}
so that $A_t^- = M\cap \mathcal{D}$.
By Step~4, and since $\tilde\mu(\mathcal{D}^c) = 0$,  we have  $ \mu(M) = \mu(A_t^-)  <\infty$. 

If $x\in A_t^-$, it holds that 
$$
\phi(x) > \psi(x) + |t| \ge |t| >0, 
$$
where the second inequality above follows from $0\in\partial \psi(0)$ and $\psi(0)=0$; hence $x$ is not a minimum of $\phi$ and thus $0\notin \partial \phi(x)$. Hence, with  $B= \partial \phi(A_t^-)$, we obtain 
\begin{equation}
    \label{eq:0notinB}
    0 \notin B. 
\end{equation}

The set $B = \proj_2((A_t^- \times \Rd) \cap \partial \phi)$ is $\sigma$-compact and thus Borel measurable. Indeed, since $\proj_2$ is continuous, the set $B$ is $\sigma$-compact as soon as $(A_t^- \times \Rd) \cap \partial \phi$ is $\sigma$-compact. But the latter property holds because $A_t^-$ is open (since $\phi$ and $\psi$ are continuous on the interiors of their domains) and thus $\sigma$-compact, while $\partial \phi$ is closed (since $\phi$ is lower semicontinuous) and thus also $\sigma$-compact;\footnote{If $U \subset \R^d$ is open, then $U = \bigcup_{n\ge 1} K_n$ with compact $K_n = \cbr{x \in \Rd \mid \nbr{x} \le n, \ball(x,1/n) \subset U}$. If $F \subset \Rd$ is closed, then $F = \bigcup_{n \ge 1} K_n$ with compact $K_n = \cbr{x \in F \mid \nbr{x} \le n}$.} the intersection of two $\sigma$-compact sets is again $\sigma$-compact. 

From the assumption that $\gamma_\psi$ and $\gamma_\phi$ are zero-couplings for $(\mu,\nu)$,  \cref{thm:expr_zcmcoup} and $0\notin B$, we get 
\begin{equation}
    \label{eq:identity_nuB}
    \begin{aligned}
    \nu(B) & = \nabla\psi_\#\mu(B) + \gamma_\psi(\{0 \}\times B ) \\
    & = \nabla\phi_\#\mu(B) + \gamma_\phi(\{0 \}\times B ). 
    \end{aligned}
\end{equation}
We now show that 
\begin{equation}
    \label{eq:gammapsi_0_B_zero}
    \gamma_\psi(\{0 \}\times B ) = 0.
\end{equation}
It is enough to show that $(\{0\}\times B) \cap \partial \psi = \varnothing$ or in other words, that if $(x,y)\in\partial \psi$ is such that $y\in B$, then necessarily $x\neq 0$. Thus, let $(x,y)\in \partial \psi$ be such that $y\in B$. Then we have
$$x\in (\partial\psi)^{-1}(y) \subset (\partial\psi)^{-1}(B) =
(\partial \psi)^{-1} (\partial\phi(A_t^-))  \subset (\partial \psi)^{-1} (\partial\phi(M)). $$ 
Now \cref{lem:aleksandrov-extended} applied at $p=z$ with $M$ defined in \cref{eq:MM_clean}
implies $(\partial \psi)^{-1}(\partial\phi(M)) \subset M$,  
so that $x\in M$.  From the definition of $M$ in \cref{eq:MM_clean} and $t\neq 0$, we have   $0\notin M$,  and  
we obtain $x\neq 0$, 
proving \cref{eq:gammapsi_0_B_zero}.

Recall from Step~1 that we have $\mu(M \sauf A_t^-) \le \mu(\Rdmz \sauf \mcal{D}) = 0$. From~\cref{eq:gammapsi_0_B_zero} we therefore get 
\begin{align}
  \nu(B)
  & = (\nabla\psi)_\#\mu(B) && \nonumber \\
  & = \mu[ (\nabla\psi)^{-1}(\partial \phi(A_t^-))] && \text{(as $B = \partial\phi(A_t^-)$ by definition)} \nonumber \\
  &\le \mu[ M \setminus \ball(z,\rho)] && \text{(by Lemma~\ref{lem:aleksandrov-extended})} \nonumber \\
  &= \mu[ A_t^- \setminus \ball(z,\rho)] && \text{(since $\mu(M \sauf A_t^-) = 0$)} \nonumber \\
  & < \mu(A_t^-) <\infty, && \text{(by Steps~4 and 5).} 
                               \label{eq:nuB_is_small_2}
\end{align}
However,  it also holds that 
\begin{align}
    \nu(B) 
    &\ge (\nabla\phi)_\#\mu(B) \nonumber \\
    &= \mu [(\nabla\phi)^{-1}(\partial\phi(A_t^-))] \nonumber \\
    &= \mu (A_t^{-} \cap \dom \nabla\phi) \nonumber \\
    &\ge \mu(A_t^{-}).
    \label{eq:nu_B_is_large_2}
\end{align}
Together, \cref{eq:nuB_is_small_2,eq:nu_B_is_large_2} provide the desired contradiction. As a consequence, the inequality in \cref{eq:wrongassumption} must be false. The proof of the identities $\nabla\psi=\nabla\phi$ $\mu$-almost everywhere and $\gamma_\psi=\gamma_\phi$ is complete.

\medskip{}

\noindent\emph{Proof of the particular case}

\medskip{}

\noindent By \cref{thm:existence}, $\zcmcoup(\mu,\nu)$ is nonempty. Let $\gamma_1,\gamma_2\in\zcmcoup(\mu,\nu)$. Applying Rockafellar's Theorem~\ref{thm:rockafellar} we get two closed convex functions $\psi_1,\psi_2: \Rd \to \reals \cup \cbr{\infty}$ such that $\spt\gamma_i\subset\partial\psi_i$ for $i$ in $\{1,2\}$.
When at least one of $\mu$ or $\nu$ has infinite mass, Lemma~\ref{lem:Psi-finite} implies that the assumption $(0,0) \in \spt\gamma_1 \cap \spt\gamma_2$ is satisfied.
The conclusion of Step~7 above yields $\gamma_1=\gamma_2$, whence $\zcmcoup(\mu,\nu)$ is a singleton.
\end{proof}

\begin{proof}[Proof of \cref{thm:uniquerepr}]
    Since $\nu = \res((\nabla \psi_j)_{\#}\mu)$, $j=1,2$,  
    we must have $\mu(\Rdmz) \ge \nu(\Rdmz)$.
    If $\mu$ and thus $\nu$ are the zero measure, the result is trivial. 
    
    Suppose $\mu$ has finite, nonzero mass. We consider their extensions $\tilde\mu \in [\mu]$ and $\bar\nu\in[\nu]$ to $\Rd$ with $\tilde{\mu}(0) = 0$ and $\bar\nu(\{0\})=\mu(\Rdmz)-\nu(\Rdmz)$. 
The measures $\bar \nu$ and $\tilde \mu$ verify
\[
    \bar \nu(\Rd) = \bar\nu(\{0\}) + \nu(\Rdmz) = \mu(\Rdmz)= \tilde\mu(\Rd) :=m <\infty.  
\]
Thus $\nu_1 = \bar\nu / m$ and $\mu_1 = \tilde \mu / m$ are probability measures. We check that for $\psi\in\{\psi_1,\psi_2\}$, the gradient $\nabla \psi$ pushes $\mu_1$ forward to $\nu_1$ in the traditional sense. This follows from the following identities: for any Borel set $B\subset\Rdmz$, we have
\[
    \nu_1(B) 
    = m^{-1} (\nabla\psi)_\#\mu(B) 
    = \mu_1\rbr{(\nabla\psi)^{-1}(B)}
    = (\nabla\psi)_\#\mu_1(B)
\]
while the equality $\mu(\Rdmz \sauf \dom \nabla\psi) = 0$ implies $(\nabla\psi)_\#\mu(\Rd) = \mu(\Rdmz \cap \dom \nabla\psi) = \mu(\Rdmz)$ and thus
\begin{align*}
    \nu_1(\{0\}) &= m^{-1} \rbr{\mu(\Rdmz) - \nu(\Rdmz)} \\
    &= m^{-1} \rbr{ (\nabla\psi)_\#\mu(\Rd) -  (\nabla\psi)_\#\mu(\Rdmz) } \\
    &= m^{-1} (\nabla\psi)_\#\mu(\{0\}) 
    = (\nabla\psi)_\#\mu_1(\{0\}). 
\end{align*}
Applying the uniqueness part of the Main Theorem in \citet{McCann1995} yields $\nabla\psi_1 = \nabla \psi_2$, $\mu_1$-almost everywhere, and thus also $\mu$-almost everywhere. 

Assume now that $\mu$ has infinite mass. 
By \cref{thm:unique}, it is sufficient to show that for $\psi\in\{\psi_1,\psi_2\}$, the measure $\gamma := \res\rbr{(\Id \times \nabla \psi)_{\#}\mu}$ is a zero-coupling of $\mu$ and $\nu$ and that $\spt \gamma \subset \partial \psi$, so that $\gamma \in \zcmcoup(\mu,\nu)$. In view of \cref{lem:Psi-finite}, we then necessarily have $(0, 0) \in \spt \gamma$ and thus $(0, 0) \in \partial \psi$, a condition in \cref{thm:unique}.
    
    First, $\gamma$ is a zero-coupling of $\mu$ and $\nu$, since for $A \in \borel(\Rdmz)$, we have
    \begin{align*}
        \gamma(A \times \Rd)
        &= \mu\rbr{A \cap (\nabla\psi)^{-1}(\Rd)}
        = \mu\rbr{A \cap \dom \nabla \psi}
        = \mu(A), \\
        \gamma(\Rd \times A)
        &= \mu\rbr{\Rdmz \cap  (\nabla\psi)^{-1}(A)}
        = \nu(A). 
    \end{align*}

Second, to prove that $\spt \gamma\subset \partial \psi$, we show the equivalent statement that $(x,y)\notin\partial \psi$ implies $(x,y)\notin \spt\gamma$. Let $(x,y)\notin\partial \psi$. Since $\partial \psi$ is closed, there exist open sets $V_x,V_y$ in $\Rd$ such that $x\in V_x$, $y\in V_y$ and $\rbr{ V_x\times V_y }\cap \partial \psi = \varnothing$. By definition, $\gamma(V_x \times V_y) = \tilde{\mu}\rbr{\cbr{\tilde x\in V_x \cap \dom \nabla\psi \mid \nabla\psi(\tilde x)\in V_y}}$. However, from $\partial\psi\cap \rbr{V_x\times V_y} = \emptyset$ we get
\begin{align*}
    \cbr{\tilde x\in V_x \cap \dom \nabla\psi \mid \nabla\psi(\tilde x)\in V_y} &
    \subset \cbr{\tilde x\in V_x \mid \partial\psi(\tilde x)\cap V_y \neq \emptyset} = \emptyset.  
\end{align*}
This shows that $\gamma(V_x\times V_y)=0$, whence $(x,y)\notin\spt \gamma$ as desired. 
\end{proof}

\begin{proof}[Proof of \cref{cor:BMcC}]
    By \cref{thm:existence}, the set $\zcmcoup(\mu,\nu)$ is not empty, and since at least one of the measures $\mu$ and $\nu$ has infinite mass, it is a singleton by \cref{thm:unique}. \cref{thm:expr_zcmcoup} then yields the desired    representations of the unique zero-coupling $\gamma$ and, in the proper case, of $\nu$.
\end{proof}

\section{Tails of transport plans between regularly varying distributions}
\label{sec:EVTOT}




In probability theory, regular variation underpins asymptotic theory for sample maxima and sums \cite{bingham1989regular, de2006extreme, resnick2007heavy, Resnick2008, Resnick2024} 
as it provides a coherent framework to describe the behaviour of functions at infinity and thus for the tails of random variables and vectors. A Borel measurable function $f$ defined in a neighbourhood of infinity and taking positive values is regularly varying with index $\tau \in \reals$ if $f(\lambda r)/f(r) \to \lambda^\tau$ as $r \to \infty$ for all $\lambda \in (0, \infty)$. A random variable $R$ is said to have a regularly varying upper tail with index $\alpha \in (0, \infty)$ if the function $r \mapsto \pr(R > r)$ is regularly varying with index $-\alpha$. The function $t \mapsto b(t) = Q(1 - 1/t)$, with $Q$ the quantile function of $R$, is then regularly varying at infinity with index $1/\alpha$, and we have $t \, \pr[R/b(t) > \lambda] \to \lambda^{-\alpha}$ as $t \to \infty$ for all $\lambda \in (0, \infty)$. The latter statement means that, in the space $\Mz(\reals)$, we have $t \, \pr[R/b(t) \in \point] \vto \nu_\alpha$ as $t \to \infty$ where $\nu_\alpha$ is concentrated on $(0, \infty)$ and is determined by $\nu_\alpha((\lambda,\infty)) = \lambda^{-\alpha}$ for all $\lambda \in (0, \infty)$.

More generally, a probability measure $P \in \Prob(\Rd)$ is regularly varying (see \citet[Definition~3.2, Theorem~3.1]{hult2006regular}, \citet[Definition~3.2, Theorem~3.1]{lindskog2014regular}, and \citet[Definition~2.1]{Resnick2024}) if there exists an increasing Borel measurable function $b$ defined in a neighbourhood of infinity and taking positive values such that
\begin{equation}
\label{eq:MRV_cont}
t \, P(b(t) \point) \vto \mu,
\qquad t \to \infty
\end{equation}
for some non-zero $\mu\in\Mz(\reals^{d})$, where $\lambda A = \cbr{\lambda x \mid x \in A}$ for a subset $A$ of Euclidean speace. The function $b$ must be regularly varying 
with index $1/\alpha \in (0, \infty)$, say, and the limit measure $\nu$ must be homogeneous:
\begin{equation}
    \label{eq:limit_homo}
	\mu(\lambda^{-1/\alpha} \point ) = \lambda \, \mu(\point),
	\qquad \lambda > 0.
\end{equation}
As a consequence, its support is a (closed) multiplicative cone, that is, $x \in \spt \mu$ and $\lambda \in [0, \infty)$ imply $\lambda x \in \spt \mu$. 
We call $\alpha \in (0, \infty)$ the index of regular variation of $P$.


Let $P, Q \in \Prob(\Rd)$ be regularly varying probability measures with auxiliary functions $b_1, b_2$ and limit measures $\mu,\nu\in\Mzp(\Rd)$, respectively.
By Theorem~6 of \citet{McCann1995} and by Rockafellar's Theorem~\ref{thm:rockafellar}, there exists a coupling measure $\pi\in\coup(P,Q)$ with cyclically monotone support contained in the graph of the subdifferential $\partial\psi$ of some closed convex function $\psi$ on $\reals^{d}$. 
By Theorem~\ref{thm:existence} and Rockafellar's Theorem~\ref{thm:rockafellar}, there exists a zero-coupling measure $\gamma\in\zcmcoup(\mu,\nu)$ with cyclically monotone support contained in the subdifferential $\partial\bpsi$ of some closed convex function $\bpsi$ on $\reals^{d}$.
We will connect the two subdifferentials $\partial \psi$ and $\partial \bpsi$ via the asymptotic behaviour of $\partial\psi(x)$ when $\norm{x}$ tends to infinity.
In view of our aim of providing notions for measures in $\Mzp(\Rd)$ that are similar to the center-outward distribution and quantile functions defined in \citet{hallin2021}, 
we believe this connection will be of interest in future work on tail quantile contours.

Write $1_d = (1, \ldots, 1) \in \Rd$ and let $\diag(a 1_d, b 1_d)$ be the $(2d) \times (2d)$ diagonal matrix with diagonal $(a, \ldots, a, b, \ldots, b)$, with $d$ repetitions of $a, b \in \reals$. 
For a diagonal matrix $A = \diag(a_1, \ldots, a_d)$ and a scalar $\lambda > 0$, we write $\lambda^A = \exp(A \log(\lambda)) = \diag(\lambda^{a_1}, \ldots, \lambda^{a_d})$ with $\exp(\point)$ the matrix exponential. The following theorem was summarized in the commutative diagram in \cref{fig:cd} and establishes the link between cyclically monotone couplings and transport plans between regularly varying probability measures on the one hand and those between their exponent measures on the other hand. Recall the Fell topology described in the beginning of \cref{sec:stability} and recall the notation $T \restr V = T \cap (V \times \Rd)$ for $T \subset \RdRd$ and $V \subset \Rd$. For the subdifferential $\partial\psi$ of a convex function $\psi$ and a scalar $b > 0$, we write $\partial\psi(b\point) = \cbr{(x,y) \mid y \in \partial\psi(bx)}$.

\begin{theorem}[Regular variation and monotone transport]
    \label{main_thm_2}
    Let $P$ and $Q$ in $\Prob(\Rd)$ be regularly varying probability measures with auxiliary functions $b_1$ and $b_2$, indices $\alpha_1 > 0$ and $\alpha_2 > 0$, and limit measures $\mu$ and $\nu$ in $\Mzp(\Rd)$, respectively. 
    Let $\pi \in \cmcoup(P,Q)$ with $\spt \pi \subset \partial \psi$ for some closed convex function $\psi : \Rd \to \reals \cup \cbr{\infty}$. 
    If we further assume that
    \begin{enumerate}[label=(\roman*)]
        \item $\mu$ vanishes on sets of Hausdorff dimension not larger than $d-1$,
        \item         $
        \forall y \in \spt\nu\sauf \cbr{0}:
        \exists x \in \spt\mu:
         \inpr{x,y}>0
        $,
    \end{enumerate}
    then we have the following properties:
    \begin{enumerate}[label=(\alph*)]
    \item
    $\zcmcoup(\mu,\nu)$ is a singleton with single element $\gamma= ( \Id \times \nabla\bpsi)_{\#}\mu$ for some closed convex function $\bpsi$ satisfying $\spt \gamma \subset \partial \bpsi$ and $\bpsi(0) = \inf_{x\in\Rd} \bpsi(x) < \infty$.
    Moreover, $\gamma$ is proper, so that $\gamma = \res \rbr{(\Id \times \nabla\bpsi)_{\#}\mu}$ and $\nu = \res \rbr{(\nabla\bpsi)_{\#}\mu}$.

    \item
    Writing $V = \intr \rbr{\spt \mu}$ and $B(t) = \diag(b_1(t) 1_d, b_2(t) 1_d)$ for $t > 0$, we have, as $t \to \infty$,
    \begin{align}
    \label{eq:RV:conv}
        t \, \res\pi(B(t) \point) &\Mto \gamma 
        &&\text{in $\Mzp(\RdRd)$}, \\
    \label{eq:RV:psi}
        (B(t)^{-1} \partial \psi) \llcorner V
        &\Fto \rbr{ \partial\bpsi } \llcorner V  
        && \text{in $\Fell(V\times\Rd)$}.
    \end{align}

    \item
    The measure $\gamma$ is homogeneous in the sense that, writing $E = \diag(\alpha_{1}^{-1} 1_d, \alpha_{2}^{-1} 1_d)$, for all $\lambda > 0$, we have
    \begin{align}
        \label{eq:pibarhomo}
        \gamma(\lambda^{-E}\point) &= \lambda \gamma,
        \\
        \label{eq:spt_homog}
        \lambda^E \spt \gamma &= \spt \gamma.
    \end{align}

    \item
    The gradient $\nabla\bpsi$ is defined and determined uniquely $\mu$-almost everywhere. We have
	\begin{equation}
		\label{eq:homog:1}
    	\nabla\bpsi(\lambda x) = \lambda^{\alpha_1/\alpha_2} \,\nabla\bpsi(x), 
        \qquad \lambda > 0,
    \end{equation}
	for $\mu$-almost every $x \in \Rd$. Moreover, the function $\bpsi$ may be modified such that $\bpsi(0)= 0$; then it is determined uniquely $\mu$--almost everywhere and satisfies
	\begin{equation}
         \label{eq:bpsi:homog}
    	\bpsi(\lambda x) = \lambda^{\alpha_1/\alpha_2+1} \, \bpsi(x),
    	\qquad \lambda \ge 0,
    \end{equation}
	for $\mu$-almost every $x \in \Rd$.
    \end{enumerate}    
\end{theorem}

In \cref{eq:RV:psi}, the set $B(t)^{-1} \partial\psi$ is the subdifferential of the convex function 
\begin{equation}
    \label{eq:psib1tb2t} 
    \psi_{b_1(t),b_2(t)}(x) = (b_1(t) b_2(t))^{-1} \psi(b_1(t)x), 
    \qquad x \in \Rd,
\end{equation}
the subdifferential of which satisfies
\[
    \partial\psi_{b_1(t),b_2(t)}(x) = b_2(t)^{-1} \partial \psi(b_1(t)x), 
    \qquad x \in \Rd,
\]
by \cref{lem:scaling} below. We can therefore view \cref{eq:RV:psi} as the convergence of a rescaled version of the subdifferential of $\psi$, seen as a multivalued map.

If $b_1 = b_2$ and thus $\alpha_1 = \alpha_2$, then the coupling distribution $\pi$ is regularly varying with the same auxiliary function and index, and with limit measure $\gamma$. In that case, the limit measure of the cyclically monotone coupling $\pi$ between $P$ and $Q$ is equal to the cyclically monotone coupling $\gamma$ between their limit measures $\mu$ and $\nu$. This is the gist of the commutative diagram in \cref{fig:cd}. If the two auxiliary functions $b_1$ and $b_2$ of $P$ and $Q$ cannot be taken to be the same, for instance because the indices $\alpha_1$ and $\alpha_2$ are different, then the coupling distribution $\pi$ is not regularly varying in the sense of \cref{eq:MRV_cont}, but we can still think of the convergence relation in \cref{eq:RV:conv} as a kind of extended form of regular variation.

\begin{remark}
    The assumptions guaranteeing the cyclically monotone coupling $\gamma$ to be proper are satisfied as soon as the reference limit measure $\mu$ is smooth enough (absolute continuity suffices) and has large enough support. As a consequence, they come at no cost for practical use as one can easily define a reference distribution satisfying them. For example, the reference distribution may be chosen to have a polar decomposition made up of the uniform distribution on the unit sphere for the angular component and a Fréchet distribution for the radial part.
    If the interest is in maxima only, then we may restrain the angular component to the positive orthant.
\end{remark}

In Theorem~\ref{main_thm_2}, the uniqueness in~(a) follows from the smoothness of $\mu$ in~(i) and Theorem~\ref{thm:unique}. The fact that $\gamma$ is proper and thus $\gamma$ and $\nu$ can be written as push-forwards of $\mu$ follows from~(ii) and Proposition~\ref{cor:cdtn_push_forward} together with homogeneity properties which are the hallmark of regular variation. The same homogeneity properties also explain (c) and (d). Finally, the stability in (b) can be linked back to Theorem~\ref{cor:stability}.

Since the functions $b_1$ and $b_2$ diverge at infinity, the Fell convergence in \cref{eq:RV:psi} implies that the tail of $\partial \psi$ is asymptotically homogeneous and converges to $\partial \bpsi$. Under additional assumptions, this Fell convergence implies locally uniform convergence if we think of the subdifferentials as multi-valued mappings \citep[Theorem~5.44]{rockafellarwets98}.

The limit measure $\nu$ captures the multivariate tail of $Q$ via $Q(A) \approx t^{-1} \nu(b_2(t)^{-1}(A)) = t^{-1} b_2(t)^\alpha \nu(A)$ for Borel sets $A$ far away from the origin. If $\mu$ is chosen to be spherically symmetric, then quantile or depth contours, however defined, are just spheres, and transporting them via $\nabla\bpsi$ gives the corresponding contours for the target measure $\nu$ and thus for the tail quantile contours of $Q$; the contours are all homothetic to a fixed one by the homogeneity of $\nabla\bpsi$. This provides a basis for the study of tail versions of the center-outward quantiles in \citep{Beirlant2020,Hallin2022,hallin2021}.


\subsection*{Proofs}


\begin{lemma}[Scaling]
	\label{lem:scaling}
	Let $P,Q \in \Prob(\Rd)$ and let $\pi \in \coup(P,Q)$ have support contained in the subdifferential $\partial \psi$ of some closed convex function $\psi$. For scalar $b_1 > 0$ and  $b_2 > 0$, consider the $2d \times 2d$ matrix $B=  \diag(b_{1} 1_d,b_{2} 1_d)$ and the measures $P_{b_1} = P(b_1 \point)$, $Q_{b_2} = Q(b_2 \point)$ and $\pi_{b_1,b_2} = \pi(B \point)$. 
	Then $\pi_{b_1,b_2} \in \coup(P_{b_1}, Q_{b_2})$ and its support $\spt \pi_{b_1,b_2} = B^{-1} \spt \pi$ is contained in the subdifferential of the closed convex function $\psi_{b_1,b_2} = (b_1 b_2)^{-1} \psi(b_1 \point)$, the subdifferential of which is $\partial \psi_{b_1,b_2}(x) = b_2^{-1} \partial \psi(b_1 x)$ for all $x \in \Rd$, or $\partial \psi_{b_1, b_2} = B^{-1} \partial \psi$ as subsets of $\RdRd$.
\end{lemma}

\begin{proof}
	The function $\psi_{b_1,b_2}$ is closed and convex since this is true for $\psi$. 
	We have $(x, y) \in \partial \psi_{b_1, b_2}$ if and only if
	\[
		\forall z \in \Rd, \qquad
		(b_1b_2)^{-1} \psi(b_1 z)
		\ge
		(b_1b_2)^{-1} \psi(b_1 x)
		+
		\inpr{z-x, y},
	\]
	which is equivalent to
	\[
		\forall s \in \Rd, \qquad
		\psi(s) \ge \psi(b_1 x) + \inpr{s - b_1x, b_2y}
	\]
	and thus to $(b_1x, b_2y) \in \partial \psi$. We find $\partial \psi_{b_1,b_2} = B^{-1} \partial \psi$ and $\partial \psi_{b_1,b_2}(x) = b_2^{-1} \partial \psi(b_1x)$ for $x \in \Rd$.
	Clearly, $\spt \pi_{b_1,b_2} = B^{-1} \spt \pi \subset B^{-1} \partial \psi = \partial \psi_{b_1,b_2}$.
\end{proof}

\begin{lemma}
    \label{lem:tail_coup}
    Let $P$ and $Q$ in $\Prob(\Rd)$ be regularly varying with auxiliary functions $b_1$ and $b_2$, indices $\alpha_1 > 0$ and $\alpha_2 > 0$, and limit measures $\mu$ and $\nu$ in $\Mzp(\Rd)$, respectively. Let $\psi$ be a closed convex function such $\partial\psi$ contains the cyclically monotone support of some $\pi\in\cmcoup(P,Q)$. Every sequence of positive numbers $t_n \to \infty$ then contains a subsequence such that, as $n \to \infty$ along the subsequence, we have writing $B(t) = \diag(b_1(t) 1_d, b_2(t) 1_d)$,
    \begin{align}
    \label{eq:RV:conv:n}
        t_n \res\pi(B(t_n) \point) 
        &\Mto \gamma 
        &&\text{in $\Mzp(\RdRd)$}, \\
        \notag
        B(t_n)^{-1} \partial \psi
        &\Fto \partial\bpsi 
        &&\text{in $\Fell(\RdRd)$}, 
    \end{align}
	for some $\gamma \in \zcmcoup(\mu, \nu)$ and some closed convex function $\bpsi$ satisfying $\spt \gamma \subset \partial \bpsi$ and $\bpsi(0) = \inf_{x\in\Rd} \bpsi(x) < \infty$.
    If $\zcmcoup(\mu, \nu)$ is a singleton, say $\cbr{\gamma}$, then we actually have $t \res \pi(B(t) \point) \Mto \gamma$ as $t\to\infty$, and, with $E = \diag(\alpha_{1}^{-1} 1_d, \alpha_{2}^{-1} 1_d)$, the measure $\gamma$ satisfies
    \begin{equation}
    \label{eq:gammahomo}
        \forall \lambda > 0, \qquad \gamma(\lambda^{-E}\point) = \lambda \gamma(\point);
    \end{equation}
    in particular, its support satisfies $\lambda^E \spt \gamma = \spt \gamma$ for all $\lambda > 0$.
\end{lemma}

\begin{proof}[Proof of Lemma~\ref{lem:tail_coup}]
    Choose a sequence $1 \le t_n \to \infty$ and consider the Borel measures $\bar\mu_n = t_n \, P(b_1(t_n)\point)$ and $\bar\nu_n = t_n \, Q(b_2(t_n)\point)$, both with mass $t_n$. In view of Lemma~\ref{lem:scaling}, the function $\psi_n = (b_2(t_n) b_1(t_n))^{-1} \psi(b_1(t_n) \point)$ is closed and convex, and its subdifferential $\partial\psi_n = B(t_n)^{-1} \partial \psi$ contains the support of the measure $\pi_n = t_n \, \pi(B(t_n) \point)$ that couples $\bar\mu_n$ and $\bar\nu_n$.

    By regular variation, $\res\bar\mu_n \Mto \mu$ and $\res\bar\nu_n \Mto \nu$ in $\Mz(\Rd)$ as $n \to \infty$, with $\mu$ and $\nu$ not depending on the sequence $t_n$. Moreover, for every $n$, the measure $\gamma_n=\res\pi_n$ is a zero-coupling between $\mu_n=\res\bar\mu_n$ and $\nu_n=\res\bar\nu_n$.
	By \cref{cor:stability}, there exists a subsequence of $t_n$, relabelled as $t_n$ again, that satisfies $\gamma_n \Mto \gamma \in \zcmcoup(\mu, \nu)$ in $\Mz(\Rd \times \Rd)$ and $\partial \psi_n \Fto \partial\bpsi$ as $n \to \infty$ along the subsequence, with $\bpsi$ a closed convex function satisfying $\supp\gamma \subset \gph \partial\bpsi$, the existence of which is a consequence of Rockafellar's Theorem~\ref{thm:rockafellar}.
	
	Since $\mu$ and $\nu$ have infinite mass and since their zero-coupling $\gamma$ is supported by $\partial \bpsi$, \cref{lem:Psi-finite} implies that $\bpsi(0)$ is finite.

    Suppose that $\zcmcoup(\mu, \nu)$ is a singleton, say $\cbr{\gamma}$. The convergence in \cref{eq:RV:conv:n} along some subsequence of an arbitrary sequence $t_n \to \infty$ implies $t \res \pi(B(t) \point) \Mto \gamma$ as $t \to \infty$; otherwise, we could find a sequence $t_n \to \infty$ along which the $\mrm{M}_0$-convergence to $\gamma$ would not take place.
    The matrix $B(t)$ is invertible for every $t > 0$ in an appropriate neighbourhood of infinity and we have $B(\lambda t)^{-1} B(t) \to \lambda^{-E}$ as $t \to \infty$ by regular variation of $b_1$ and $b_2$.
    For $\lambda \in (0, \infty)$, we have, on the one hand,
	\[
		\lambda t \, \res\pi(B(t) \point)
		\vto \lambda \gamma,
		\qquad t \to \infty,
	\]
	and on the other hand,
	\begin{equation}
	\label{eq:gammalambdat}
        \lambda t \, \res\pi(B(t) \point)
		= \lambda t \, \res\pi(B(\lambda t) B(\lambda t)^{-1} B(t) \point) 
		\vto \gamma(\lambda^{-E} \point),
		\qquad t \to \infty,
	\end{equation}
    from which \cref{eq:gammahomo} follows. To see \cref{eq:gammalambdat}, note that $\pi(B(t) \point) = (B(t)^{-1})_{\#} \pi$ and that, for any bounded and continuous function $f : \reals^{2d}_{\smallsetminus 0} \to \reals$ that vanishes on a neighbourhood of the origin, we have
    \begin{align*}
        \lambda t \rbr{\res\pi(B(t) \point)}(f)
        &= \lambda t \int_{\reals^2_{\smallsetminus 0}} f(B(t)^{-1} x) \, \diff \pi(x) \\
        &= \lambda t \int_{\reals^2_{\smallsetminus 0}} f(B(\lambda t)^{-1} B(\lambda t) B(t)^{-1} x) \diff \pi(x) \\ 
        &= \lambda t \rbr{\res\pi(B(\lambda t) \point)}(f_{\lambda, t}),
    \end{align*}
    where $f_{\lambda, t}(x) := f(B(\lambda t) B(t)^{-1}x) \to f_{\lambda,\infty}(x) = f(\lambda^E x)$ as $t \to \infty$ for all $x \in \R^{2d}_{\smallsetminus 0}$. Since the latter convergence takes place uniformly on compact subsets of $\R^{2d}_{\smallsetminus 0}$ and since the functions $f_{\lambda, t}$ for large $t$ and $f_{\lambda}$ all vanish on a common neighourhood of the origin, we find by an extension of the mapping theorem in \citet[Theorem~2.5]{hult2006regular} that
    \[
        \lambda t \rbr{\res\pi(B(t) \point)}(f)
        \to \gamma(f_{\lambda,\infty}) = \gamma(f(\lambda^E \point)) = (\gamma(\lambda^{-E} \point))(f),
        \qquad t \to \infty,
    \]
    as required.
\end{proof}

\begin{proof}[Proof of Theorem~\ref{main_thm_2}]
    (a) Since $\mu$ and $\nu$ are limit measures in the definition of regular variation, they are nonzero, and the homogeneity in \cref{eq:limit_homo} then implies $\mu(\Rdmz)=\nu(\Rdmz)=\infty$. By assumption~(i) and \cref{cor:BMcC}, it follows that $\zcmcoup(\mu,\nu)$ is a singleton, $\cbr{\gamma}$, say. \cref{lem:tail_coup} then implies that $\spt \gamma$ satisfies the homogeneity assumption in \cref{cor:cdtn_push_forward}, which, in combination with assumption~(ii), implies that $\gamma$ is proper. The existence and properties of $\bpsi$ follow from the same \cref{lem:tail_coup} too. Finally, the expressions for $\gamma$ and $\nu$ in terms of $\nabla\bpsi$ and $\mu$ follow from \cref{cor:BMcC}.

    \smallskip

    (b) The $\Mzp$-convergence in \cref{eq:RV:conv} is clear from (a) and \cref{lem:tail_coup}.
    By \cref{lem:scaling}, $\spt\rbr{\res\pi(B(t) \point)}$ is a subset of $B(t)^{-1} \partial\psi$, which is the subdifferential of the function $\psi_{b_1(t),b_2(t)}$ in \cref{eq:psib1tb2t}. Since $\zcmcoup(\mu,\nu)$ is a singleton from (a), and as $\res\pi(B(t) \point) \Mto \gamma$ as $t\to\infty$, we can apply \cref{cv:cm_coupling} to any sequence $0 < t_n \to \infty$.
    Again from the proof of Lemma~\ref{lem:tail_coup}, we have $\spt\gamma\subset\partial\bpsi$ for $\bpsi$ from (a). It follows that in $\Fell(V \times \Rd)$, we have
    \[
        B(t_n)^{-1} \partial \psi \llcorner V \Fto \partial \bpsi \llcorner V, \qquad n \to \infty.
    \]
    Since this is true for every sequence $t_n \to \infty$, the convergence in \cref{eq:RV:psi} follows.

    \smallskip

    (c) The two properties of $\gamma$ are already proven in Lemma~\ref{lem:tail_coup}.

    \smallskip

    (d) 
    Define $T = \spt \gamma$ and $T(x) = \cbr{y \in \Rd \mid (x,y) \in T}$ for $x \in \Rd$. Further, put $\dom T = \{x \in \Rd \mid T(x) \ne \emptyset\} = \proj_1(T)$. By \cref{eq:spt_homog}, we have $y \in T(x)$ if and only if  $\lambda^{1/\alpha_2} y \in T(\lambda^{1/\alpha_1} x)$, so that
	\begin{equation} 
	\label{eq:T_homog}
        \forall \lambda > 0, x \in \Rd: \qquad
		T(\lambda x)= \lambda^{\alpha_1/\alpha_2} T(x),
	\end{equation}
	an identity in which one and hence both sets can be empty.
	
	By \cref{eq:T_homog}, the set $\dom T$ is a multiplicative cone, which furthermore contains the origin. Further, note that $\dom T= \proj_1(\supp \gamma)$, so $\mu(\Rd \sauf \dom T)= 0$ and, by Lemma~\ref{lem:supp-zero-proj} and Remark~\ref{rk:zero_margin_spt_incl}, $\spt\mu = \cl(\dom T)$. As $\supp \gamma \subset \partial\bpsi$, also $\dom T \subset \dom \partial \bpsi$. Replace $\bpsi$ by $\bpsi - \bpsi(0)$ to ensure $\bpsi(0) = 0$ without changing the subdifferential or gradient of $\bpsi$.

	Let $x \in \dom T$ and define $f(t) = \bpsi(t x)$ for $t \geq 0$. The convex function $f$ is  finite on $[0, \infty)$, as $ \dom T \subset \dom \partial\bpsi$ and $\dom T$ is a multiplicative cone.
	Its subdifferential  $\partial f(t)$ contains the set $\inpr{\partial \bpsi(tx), x} = \{ \inpr{y, x} : y \in \partial \bpsi(tx) \}$ (just check the inequality in the definition of a subdifferential); therefore, it also contains the set $\inpr{T(tx), x}$.
	As a consequence, $f$ has a left derivative $f'_{-} $ and a right derivative $f'_{+}$ satisfying 
	$f'_{-}(t) \leq \inf \inpr{T(tx), x} \leq \sup \inpr{T(tx), x} \leq f'_{+}(t)$ 
	at every $t>0$. 	
	As furthermore
	\begin{equation}
        \label{eq:f}
        \forall t \in (0, \infty): \qquad
        f(t)
        = \int_{(0,t)} f'_{-}(s) \, \diff s
        = \int_{(0,t)} f'_{+}(s) \, \diff s, 
	\end{equation}	
	(e.g. \cite{Rockafellar1970}, Corollary 24.2.1), \cref{eq:T_homog} implies 
	\begin{equation}
        \label{eq:bpsi}
        \bpsi(tx)
        = \int_{(0,t)} \sup \inpr{T(sx), x} \, \diff s
        = \sup \inpr{T(x), x} \int_{(0,t)}  s^{\alpha_1/\alpha_2} \, \diff s.
	\end{equation}
    We find that $\bpsi(x)$ is uniquely determined by $T$ for $x \in \dom T$ and thus $\mu$-almost everywhere, and also that \cref{eq:bpsi:homog} holds for such $x$. 
    
    
    As in Lemma~\ref{lem:clconvhull}, let the function $\phi : \Rd \to \R \cup \cbr{\infty}$ be the closure of the convex hull of the restriction of $\bpsi$ to $D = \dom T$. 
    Its epigraph is 
    \[ 
        \epi \phi = \cl\rbr{\conv(\epi \bpsi \cap (D \times \reals))}.
    \] 
    By Lemma~\ref{lem:clconvhull}(c), we have $\partial \bpsi(x) \subset \partial \phi(x)$ for all $x \in D$, and therefore 
    \[ 
    	\spt \gamma 
    	\subset \partial \bar\psi \cap (D \times \Rd) 
    	\subset \partial \phi.
    \]
    \cref{thm:expr_zcmcoup} yields $\mu(\Rdmz \sauf \dom \nabla \phi) = 0$, that is, $\nabla \phi$ is defined $\mu$-almost everywhere.
    
    From Lemma~\ref{lem:clconvhull}(b), we get $\phi(x)= \bpsi(x)$ for $x \in D$. Since $D$ is a multiplicative cone and since $\bpsi(\lambda x) = \lambda^{\alpha_1/\alpha_2+1} \bpsi(x)$ for all $x \in D$ and $\lambda \ge 0$, we have
    \begin{equation}
    \label{eq:phi_homog}
    	\phi(\lambda x) = \lambda^{\alpha_1/\alpha_2+1} \phi(x), 
    	\qquad \lambda > 0, \, x \in D.
    \end{equation}

    Take $x \in D$ and $\lambda > 0$. By Lemma~\ref{lem:clconvhull}(b)--(c), we have 
    \[ 
        v \in \partial\phi(x) \iff \forall z \in D: \phi(z) \ge \phi(x) + \inpr{v, z-x}. 
    \]
    Therefore, as $D$ is a multiplicative cone, 
    \begin{align*}
        w \in \partial\phi(\lambda x)
        &\iff
        \forall z \in D: \phi(\lambda z) \ge \phi(\lambda x) + \inpr{w, \lambda z - \lambda x} \\
        &\iff 
        \forall z \in D: \phi(z) \ge \phi(x) + \inpr{\lambda^{-\alpha_1/\alpha_2} w, z-x} \\
        &\iff
        \lambda^{-\alpha_1/\alpha_2} w \in \partial \phi(x),
    \end{align*}
    where we used \cref{eq:phi_homog} in the second equivalence and Lemma~\ref{lem:clconvhull}(b)--(c) in the third one. We conclude
    \[
    	\partial\phi(\lambda x) = \lambda^{\alpha_1/\alpha_2} \partial\phi(x),
    	\qquad \lambda>0, \, x \in D.
    \]
    
    As a consequence, $D \cap \dom\nabla\phi$ is a multiplicative cone and $\nabla\phi(\lambda x) = \lambda^{\alpha_1/\alpha_2} \nabla\phi(x)$ for all $\lambda>0$ and all $x \in D \cap \dom\nabla\phi$. In view of Lemma~\ref{lem:clconvhull}(e), Eq.~\eqref{eq:homog:1} must hold for all $x \in D \cap \dom\nabla\phi$ and therefore $\mu$-almost everywhere. 
\end{proof}

\section{Discussion}
\label{sec:discussion}

We have studied measure transportation between infinite Borel measures on Euclidean space by cyclically monotone mappings. The measures under consideration are finite on sets bounded away from the origin, and so are their zero-couplings, which have the prescribed margins apart from a possible contribution at the origin. We have shown that a cyclically monotone zero-coupling always exists, and we have identified conditions under which it is unique and can be represented as the restriction to the punctured space of the push-forward of one measure by a cyclically monotone mapping, that is, the gradient of a closed convex function. In addition, we have established stability properties of cyclically monotone zero-couplings with respect to $\Mz$-convergence of measures and Fell convergence of subgradients.

The cost function was implicit in our work through the notion of cyclical monotonicity. In contrast to classical optimal transport theory, we did not assume finiteness of transport costs or moments, and no minimization of an expected cost was required. Allowing the measures to have potentially infinite mass leads to new phenomena compared to the case of probability measures. In particular, the possible presence of a sink or infinite reservoir of mass at the origin may preclude representations in terms of push-forward measures. This observation motivated the introduction of the notion of proper zero-couplings and the study of conditions ensuring their existence.

We have further connected these results to the theory of regular variation. Specifically, we have shown how cyclically monotone couplings between regularly varying probability measures converge at high levels to cyclically monotone zero-couplings between their limit measures. In this way, the present work provides a rigorous link between monotone transport for regularly varying distributions and monotone transport for their tail measures, thereby laying a theoretical foundation for the use of transport-based methods in multivariate extreme value analysis.

Several directions for further research remain open. One natural extension is the study of transport problems under cyclical monotonicity with respect to more general cost functions beyond the quadratic cost considered here. Another direction concerns the development of statistical methodology based on cyclically monotone transport for infinite measures, including inference and simulation methods for multivariate extremes and related stochastic models such as Lévy processes.

\bibliographystyle{chicago}
\bibliography{biblio}

\begin{thebibliography}{}

\bibitem[\protect\citeauthoryear{Albergo and Vanden-Eijnden}{Albergo and Vanden-Eijnden}{2022}]{albergo2022building}
Albergo, M.~S. and E.~Vanden-Eijnden (2022).
\newblock Building normalizing flows with stochastic interpolants.
\newblock {\em arXiv preprint arXiv:2209.15571\/}.

\bibitem[\protect\citeauthoryear{Alberti and Ambrosio}{Alberti and Ambrosio}{1999}]{AlbertiAmbrosio1999}
Alberti, G. and L.~Ambrosio (1999).
\newblock A geometrical approach to monotone functions in $\mathbb{R}^{n}$.
\newblock {\em Mathematische Zeitschrift\/}~{\em 230}, 259--316.

\bibitem[\protect\citeauthoryear{Aleksandrov}{Aleksandrov}{1942}]{aleksandrov1942}
Aleksandrov, A. (1942).
\newblock Existence and uniqueness of a convex surface with a given integral curvature.
\newblock {\em C. R. (Doklady) Acad. Sci. URSS(N.S.)\/}~{\em 35}, 131--134.

\bibitem[\protect\citeauthoryear{Ambrosio, Bru{\'e}, Semola, et~al.}{Ambrosio et~al.}{2021}]{ambrosio2021lectures}
Ambrosio, L., E.~Bru{\'e}, D.~Semola, et~al. (2021).
\newblock {\em Lectures on optimal transport}, Volume 130.
\newblock Springer.

\bibitem[\protect\citeauthoryear{Anderson and Klee}{Anderson and Klee}{1952}]{AndersonKlee1952}
Anderson, R.~D. and V.~L. Klee (1952).
\newblock Convex functions and upper semi-continuous collections.
\newblock {\em Duke Math. J.\/}~{\em 19}, 349--357.

\bibitem[\protect\citeauthoryear{Beiglböck}{Beiglböck}{2015}]{Beiglböck2015}
Beiglböck, M. (2015).
\newblock Cyclical monotonicity and the ergodic theorem.
\newblock {\em Ergodic Theory Dyn. Syst.\/}~{\em 35(3)}.

\bibitem[\protect\citeauthoryear{Beirlant, Buitendag, Del~Barrio, Hallin, and Kamper}{Beirlant et~al.}{2020}]{Beirlant2020}
Beirlant, J., S.~Buitendag, E.~Del~Barrio, M.~Hallin, and F.~Kamper (2020).
\newblock Center-outward quantiles and the measurement of multivariate risk.
\newblock {\em Insurance: Mathematics and Economics\/}~{\em 95}, 79--100.

\bibitem[\protect\citeauthoryear{Bingham, Goldie, Teugels, and Teugels}{Bingham et~al.}{1989}]{bingham1989regular}
Bingham, N.~H., C.~M. Goldie, J.~L. Teugels, and J.~Teugels (1989).
\newblock {\em Regular Variation}.
\newblock Cambridge University Press.

\bibitem[\protect\citeauthoryear{Brenier}{Brenier}{1987}]{brenier1987decomposition}
Brenier, Y. (1987).
\newblock D{\'e}composition polaire et r{\'e}arrangement monotone des champs de vecteurs.
\newblock {\em CR Acad. Sci. Paris S{\'e}r. I Math.\/}~{\em 305}, 805--808.

\bibitem[\protect\citeauthoryear{Brenier}{Brenier}{1991}]{brenier1991polar}
Brenier, Y. (1991).
\newblock Polar factorization and monotone rearrangement of vector-valued functions.
\newblock {\em Communications on pure and applied mathematics\/}~{\em 44\/}(4), 375--417.

\bibitem[\protect\citeauthoryear{Catalano, Lavenant, Lijoi, and Prünster}{Catalano et~al.}{2021}]{CatalanoLavenant2021}
Catalano, M., H.~Lavenant, A.~Lijoi, and I.~Prünster (2021).
\newblock A {Wasserstein} index of dependence for random measures.
\newblock {\em arXiv preprint arXiv:2109.06646\/}.

\bibitem[\protect\citeauthoryear{Chernozhukov, Galichon, Hallin, and Henry}{Chernozhukov et~al.}{2017}]{Chernozhukov2017}
Chernozhukov, V., A.~Galichon, M.~Hallin, and M.~Henry (2017).
\newblock Monge–kantorovich depth, quantiles, ranks and signs.
\newblock {\em The Annals of Statistics\/}~{\em 45}, 223--256.

\bibitem[\protect\citeauthoryear{de~Haan and Ferreira}{de~Haan and Ferreira}{2006}]{de2006extreme}
de~Haan, L. and A.~Ferreira (2006).
\newblock {\em Extreme Value Theory: An Introduction}.
\newblock Springer New York.

\bibitem[\protect\citeauthoryear{De~Pascale, Kausamo, and Wyczesany}{De~Pascale et~al.}{2023}]{de202360}
De~Pascale, L., A.~Kausamo, and K.~Wyczesany (2023).
\newblock 60 years of cyclic monotonicity: a survey.
\newblock {\em arXiv preprint arXiv:2308.07682\/}.

\bibitem[\protect\citeauthoryear{{de Valk} and Segers}{{de Valk} and Segers}{2019}]{devalk2019tailsoptimaltransportplans}
{de Valk}, C. and J.~Segers (2019).
\newblock Tails of optimal transport plans for regularly varying probability measures.
\newblock {\em arXiv preprint arXiv:1811.12061v2\/}.

\bibitem[\protect\citeauthoryear{Figalli and Gigli}{Figalli and Gigli}{2010}]{figalli2010new}
Figalli, A. and N.~Gigli (2010).
\newblock A new transportation distance between non-negative measures, with applications to gradients flows with {Dirichlet} boundary conditions.
\newblock {\em Journal de math{\'e}matiques pures et appliqu{\'e}es\/}~{\em 94\/}(2), 107--130.

\bibitem[\protect\citeauthoryear{Guillen, Mou, and Swiech}{Guillen et~al.}{2019}]{GuillenMou2019}
Guillen, N., C.~Mou, and S.~Swiech (2019).
\newblock Coupling {Lévy} measures and comparison principles for viscosity solutions.
\newblock {\em Transactions of the American Mathematical Society\/}~{\em 372\/}(10), 7327--7370.

\bibitem[\protect\citeauthoryear{Hallin}{Hallin}{2022}]{Hallin2022}
Hallin, M. (2022).
\newblock Measure transportation and statistical decision theory.
\newblock {\em Annual Review of Statistics and Its Application\/}~{\em 9\/}(1), 401--424.

\bibitem[\protect\citeauthoryear{Hallin, Del~Barrio, Cuesta-Albertos, and Matrán}{Hallin et~al.}{2021}]{hallin2021}
Hallin, M., E.~Del~Barrio, J.~Cuesta-Albertos, and C.~Matrán (2021).
\newblock Distribution and quantile functions, ranks and signs in dimension d: a measure transportation approach.
\newblock {\em Annals of Statistics\/}~{\em 49}, 1139--1165.

\bibitem[\protect\citeauthoryear{Huesmann}{Huesmann}{2016}]{huesmann2016transportrandom}
Huesmann, M. (2016).
\newblock Optimal transport between random measures.
\newblock {\em Ann. Inst. Henri Poincar{\'e} Probab. Stat.\/}~{\em 52(1)}, 196–232.

\bibitem[\protect\citeauthoryear{Huesmann and Sturm}{Huesmann and Sturm}{2013}]{huesmann2013optimal}
Huesmann, M. and K.-T. Sturm (2013).
\newblock Optimal transport from {Lebesgue} to {Poisson}.
\newblock {\em The Annals of Probability\/}~{\em 41\/}(4), 2426--2478.

\bibitem[\protect\citeauthoryear{Hult and Lindskog}{Hult and Lindskog}{2006}]{hult2006regular}
Hult, H. and F.~Lindskog (2006).
\newblock Regular {Variation} for {Measures} on {Metric} {Spaces}.
\newblock {\em Publ. Inst. Math.\/}~{\em 80\/}(94), 121--140.

\bibitem[\protect\citeauthoryear{Lindskog, Resnick, and Roy}{Lindskog et~al.}{2014}]{lindskog2014regular}
Lindskog, F., S.~I. Resnick, and J.~Roy (2014).
\newblock Regularly varying measures on metric spaces: Hidden regular variation and hidden jumps.
\newblock {\em Probability Surveys\/}~{\em 11}, 270--314.

\bibitem[\protect\citeauthoryear{Lipman, Chen, Ben-Hamu, Nickel, and Le}{Lipman et~al.}{2022}]{lipman2022flow}
Lipman, Y., R.~T. Chen, H.~Ben-Hamu, M.~Nickel, and M.~Le (2022).
\newblock Flow matching for generative modeling.
\newblock {\em arXiv preprint arXiv:2210.02747\/}.

\bibitem[\protect\citeauthoryear{McCann}{McCann}{1995}]{McCann1995}
McCann, R.~J. (1995).
\newblock Existence and uniqueness of monotone measure-preserving maps.
\newblock {\em Duke Mathematical Journal\/}~{\em 80\/}(2), 309--323.

\bibitem[\protect\citeauthoryear{Molchanov}{Molchanov}{2017}]{Molchanov2017}
Molchanov, I. (2017).
\newblock {\em Theory of Random Sets}.
\newblock Springer Cham.

\bibitem[\protect\citeauthoryear{Pratelli}{Pratelli}{2008}]{pratelli2008sufficiency}
Pratelli, A. (2008).
\newblock On the sufficiency of c-cyclical monotonicity for optimality of transport plans.
\newblock {\em Mathematische Zeitschrift\/}~{\em 258\/}(3), 677--690.

\bibitem[\protect\citeauthoryear{Resnick}{Resnick}{2007}]{resnick2007heavy}
Resnick, S.~I. (2007).
\newblock {\em Heavy-Tail Phenomena: Probabilistic and Statistical Modeling}.
\newblock Springer Science \& Business Media.

\bibitem[\protect\citeauthoryear{Resnick}{Resnick}{2008}]{Resnick2008}
Resnick, S.~I. (2008).
\newblock {\em Extreme Values, Regular Variation and Point Processes}.
\newblock Springer Science \& Business Media.

\bibitem[\protect\citeauthoryear{Resnick}{Resnick}{2024}]{Resnick2024}
Resnick, S.~I. (2024).
\newblock {\em The art of finding hidden risks: Hidden regular variation in the 21st century}.
\newblock Springer Nature.

\bibitem[\protect\citeauthoryear{Rockafellar}{Rockafellar}{1970}]{Rockafellar1970}
Rockafellar, R.~T. (1970).
\newblock {\em Convex Analysis}.
\newblock Princeton: Princeton University Press.

\bibitem[\protect\citeauthoryear{Rockafellar and Wets}{Rockafellar and Wets}{1998}]{rockafellarwets98}
Rockafellar, R.~T. and R.~J.-B. Wets (1998).
\newblock {\em Variational Analysis}.
\newblock New York : Springer.

\bibitem[\protect\citeauthoryear{R{\"u}schendorf and Rachev}{R{\"u}schendorf and Rachev}{1990}]{ruschendorf1990characterization}
R{\"u}schendorf, L. and S.~T. Rachev (1990).
\newblock A characterization of random variables with minimum ${L}^2$-distance.
\newblock {\em Journal of Multivariate Analysis\/}~{\em 32\/}(1), 48--54.

\bibitem[\protect\citeauthoryear{Santambrogio}{Santambrogio}{2015}]{santambrogio2015optimal}
Santambrogio, F. (2015).
\newblock {\em {Optimal Transport for Applied Mathematicians: Calculus of Variations, PDEs, and Modeling}}.
\newblock Cham: Birkh{\"a}user/Springer.

\bibitem[\protect\citeauthoryear{Schachermayer and Teichmann}{Schachermayer and Teichmann}{2009}]{schachermayer2009characterization}
Schachermayer, W. and J.~Teichmann (2009).
\newblock Characterization of optimal transport plans for the {Monge--Kantorovich} problem.
\newblock {\em Proceedings of the American Mathematical Society\/}~{\em 137\/}(2), 519--529.

\bibitem[\protect\citeauthoryear{Segers}{Segers}{2022}]{segers2022}
Segers, J. (2022).
\newblock Graphical and uniform consistency of estimated optimal transport plans.
\newblock {\em arXiv preprint arXiv:2208.02508\/}.

\bibitem[\protect\citeauthoryear{Villani}{Villani}{2009}]{villani2009optimal}
Villani, C. (2009).
\newblock {\em {Optimal Transport: Old and New}}, Volume 338.
\newblock Springer.

\end{thebibliography}

\appendix

\section{Background}
\label{sec:background}


\subsection{Portmanteau and Prohorov theorem for infinite measures}
\label{section:Mzp}

The space $\Mzp(\Rd)$ enjoys useful properties similar to the classical Portmanteau theorem or Prohorov theorem for probability measures which will be of constant use. We refer to  \cite{hult2006regular, lindskog2014regular} for details and quote here the results we need.

\begin{theorem}[Portmanteau Theorem 2.4 in \cite{hult2006regular}]
\label{thm:portmanteau}
Let $\mu,\mu_n\in\Mzp(\Rd),n\ge1$. The following statements are equivalent.
\begin{enumerate}[label=(\roman*)]
\item $\mu_n\Mto\mu$ as $n\to\infty$,

\item $\mu_n(A)\to\mu(A)$ as $n\to\infty$ for every $A\in\Szbor$ satisfying $\mu(\bnd A)=0$ and $\cl(A)\cap\cbr{0}=\emptyset$ (closure in $\Rd$ as a whole),

\item $ \mu(\intr(A)) \le \liminf  \mu_n(\intr(A)\le \limsup \mu_n(\cl(A)) \le \mu(\cl(A))$ for every $A\in\Rdmz$ such that $\cl(A)\cap\cbr{0}=\emptyset$ (closure in $\Rd$ as a whole).
\end{enumerate}
\end{theorem}

\begin{theorem}[Prohorov Theorem 2.7 in \cite{hult2006regular}]
\label{thm:prohorov_Mzp}
The set $M\subset\Mzp(\Rd)$ is relatively compact if and only if there exists a sequence $r_i>0,i\ge1$ satisfying $r_i\downarrow0$ as $i\to\infty$ such that for each $i\ge1$ we have 
$$\sup_{\mu\in M}\mu(\Rd\setminus\ball_{0,r_i}) < \infty$$
and for each $\eta >0$ there exists a compact set $K_i\subset\Rd\setminus\ball_{0,r_i}$ such that
\[ 
    \sup_{\mu\in M}\mu\brbr{\Rd\setminus(\ball_{0,r_i}\cup K_i)} < \eta . 
\]
\end{theorem}

\cref{thm:prohorov_Mzp} gives a path to prove the $\Mzp$-convergence of $\mu_n$ to $\mu$ by showing that (i) the set $\cbr{\mu_n}_{n\ge1}$ is relatively compact and (ii) all its accumulation points are equal to $\mu$.



\subsection{Convex functions and cyclical monotonicity}
\label{section:CVX}


In this subsection we recall convex functions and cyclical monotonicity, which are connected to optimal transport. 

A function $f: \Rd \to \rset\cup\{+\infty\}$ is said to be convex if the inequality
\[
    f \rbr{ (1-\tau)x + \tau y } \le \rbr{1-\tau} f(x) + \tau f(y)
\] 
holds for every $x,y\in\Rd$ and $\tau \in (0,1)$; equivalently, $f$ is a convex function if its epigraph $\epi f = \cbr{(x, t) \in \Rd \times \rset \mid f(x) \le t}$ is a convex set. The convex function $f$ is proper if $f(x)<+\infty$ for some $x$ in $\Rd$ and it is closed if its epigraph is closed, or equivalently, if $f$ is lower semi-continuous. The effective domain of $f$ is the set  $\dom f = \{ x \in \Rd \mid f(x) \in \rset \}$. 



For any proper, convex function $f: \Rd \to \rset\cup\{+\infty\}$ and any point $\bar x$ lying in $\dom f$, we call subgradient of $f$ at $\bar x$ any point $v\in\Rd$ satisfying
\[
f(x) \ge f(\bar x) + \inpr{v,x-\bar x}
\]
for every $x$ in $\Rd$. We let $\partial f(\bar x)$ denote the set containing all such $v$ and call it the subdifferential of $f$ at $\bar x$.
Further, $\partial f = \cbr{(\bar x, v) \in \dom f \times \Rd \mid v \in \partial f(\bar{x})}$.
The domain of $\partial f$, notation $\dom \partial f$, is the set of all $x$ at which $\partial f(x)$ is not empty. The function $f$ is differentiable at $x$ if and only $\partial f(x)$ is a singleton.
Combining \citet[Theorem~6.3, Theorem~23.4, and Corollary~25.1.1]{Rockafellar1970}, we have in fact
\begin{equation}
\label{eq:domains}
	\left.
	\begin{array}{rcrcr}
		\dom \nabla f &\subset& \intr(\dom \partial f) &=& \intr(\dom f), \\
		&& \dom \partial f\phantom{)} &\subset& \dom f\phantom{)}, \\
		&& \cl(\dom \partial f) &=& \cl(\dom f).
	\end{array}	
	\right\}
\end{equation}

The concept of (cyclical) monotonicity is linked to the subdifferentials of convex functions.

\begin{definition}[Cyclically monotone maps and sets]
\label{def:monotonicity}
A set $T \subset \RdRd$ is cyclically monotone if,  for every integer $n \ge 1$ and all $(x_1,y_1),\ldots,(x_n,y_n) \in T$ with $y_{n+1} := y_1$, we have 
\begin{equation}
    \label{rel:cyclical_mononicity}
    \sum_{i=1}^n \inpr{x_i,y_i} \ge \sum_{i=1}^n \inpr{x_i,y_{i+1}}.
\end{equation}
If \cref{rel:cyclical_mononicity} holds for $n=2$, then $T$ is monotone.
The set $T$ is maximal (cyclically) monotone if it is not contained in a larger subset of $\RdRd$ that is also (cyclically) monotone.
\end{definition}



For a family $\mcal{G}$ of subsets of $\RdRd$, we write $\mcal{G}_{(c)m}$ for the subfamily consisting of (cyclically) monotone sets in $\mcal{G}$. 
Likewise we add m(c)m for maximal (cyclically) monotonicity. The following theorem due to Rockafellar establishes a connection between convex functions and cyclically monotone sets. It is taken from Theorems~24.8 and~24.9 in \cite{Rockafellar1970} and Theorem~12.25 in \cite{rockafellarwets98}.

\begin{theorem}[Rockafellar's Theorem]
\label{thm:rockafellar}
\begin{enumerate}[label=(\alph*)]
\item A set $T \subset \RdRd$ is cyclically monotone if and only if it is contained in the subdifferential of a closed convex function.

\item A set $T \subset \RdRd$ has the form $T = \partial \psi$ for some proper, closed convex function $\psi : \Rd\to\rset\cup\{+\infty\}$ if and only if $T$ is maximal cyclically monotone. Then $\psi$ is determined by $T$ uniquely up to an additive constant.
\end{enumerate}
\end{theorem}

The following lemma is used in the proof of \cref{thm:unique}.

\begin{lemma}
	\label{lem:smallset}
	Let $\psi$ and $\phi$ be convex functions on $\Rd$. For every $t \in \reals$, the set
	\begin{equation}
	\label{eq:smallset}
		\cbr{ 
			x \in \dom \nabla\psi \cap \dom \nabla\phi \mid 
			\psi(x) - \phi(x) = t \text{ and } \nabla\psi(x) \ne \nabla\phi(x)
		}
	\end{equation}
	has Hausdorff dimension at most $d-1$.
\end{lemma}

\begin{proof}
	Let $X$ be the set in \cref{eq:smallset}. By McCann's Implicit Function Theorem \citep[Theorem~17 and Corollary~19]{McCann1995}, every point $x$ in $X$ has an open neighbourhood $N_{x}$ such that the $(d-1)$-dimensional Hausdorff measure of the set
	\[
		\cbr{ y \in N_{x} \mid \psi(y)-\phi(y) = t }
	\] 
	is finite. The collection of sets $\{N_{x} \mid x \in X\}$ forms an open cover of $X$. By the Lindel\"{o}f property of $\Rd$, there exists a countable subcover, i.e., a countable set $Q \subset X$ such that $X$ is a subset of $\bigcup_{x \in Q}N_{x}$ and then also of 
	\[ 
		\textstyle\bigcup_{x \in Q} \cbr{y \in N_{x} \mid \psi(y)-\phi(y) = t}. 
	\]
	The latter set is a countable union of sets of Hausdorff dimension at most $d-1$ and has therefore Hausdorff dimension at most $d-1$ as well. 
\end{proof}

The convex hull of a function $g : \Rd \to \reals \cup \{\infty\}$ is the function $\conv g$ with epigraph $\epi(\conv g) = \conv(\epi g)$, the convex hull of the epigraph of $g$. It is the greatest convex function majorized by $g$ \citep[p.~36]{Rockafellar1970}. The closure of a function $f$ is the function $\cl f$ whose epigraph is $\epi(\cl f) = \cl(\epi f)$; it is the same as the lower semicontinuous hull of $f$. The following lemma is used in the proof of \cref{main_thm_2}.

\begin{lemma}[Restrictions of convex functions]
	\label{lem:clconvhull}
	Let $\psi : \reals^d \to \reals \cup \{\infty\}$ be a closed convex function and let $D \subset \Rd$ be such that $D \cap \dom \psi \ne \varnothing$. Define the function $\phi$ by $\phi = \cl(\conv g)$, where $g(x) = \psi(x)$ if $x \in D$ and $g(x) = \infty$ if $x \in \Rd \sauf D$. Then $\phi$ is a closed convex function with the following properties:
	\begin{enumerate}[label=(\alph*)]
		\item $\phi \ge \psi$;
		\item $\phi(x) = \psi(x)$ for $x \in D$;
		\item $\partial \phi(x) = \{ v \in \Rd : \forall z \in D, \, \psi(z) \ge \psi(x) + \inpr{z-x,v} \} \supset \partial \psi(x)$ for $x \in  D$;
		\item $\dom \nabla \phi \subset \dom \partial \psi$;
		\item $D \cap \dom \nabla \phi \subset \dom \nabla \psi$ and $\nabla \phi(x) = \nabla \psi(x)$ for all $x \in D \cap \dom \nabla \phi$.
	\end{enumerate}
\end{lemma}

\begin{proof}
	Since $\epi \phi = \cl(\epi (\conv g)) = \cl(\conv (\epi g))$, the epigraph of $\phi$ is closed and convex. Hence $\phi$ is a closed and convex function.  
		
	(a) As $\epi g = \epi \psi \cap (D \times \reals)$, we have $\epi \psi \supset \epi g$ and thus, since $\epi \psi$ is convex and closed, $\epi \psi \supset \conv (\epi g)$ and then $\epi \psi \supset \cl(\conv(\epi g)) = \epi \phi$. Hence $\psi \le \phi$.
	
	(b) Since $\epi \phi \supset \epi g$, we must have $\phi(x) \le g(x) = \psi(x)$ for $x \in D$, and thus $\phi(x) = \psi(x)$ for such $x$.
	
	(c) Let $x \in D$ and write $V = \{ v \in \Rd : \forall z \in D, \, \psi(z) \ge \psi(x) + \inpr{z-x,v} \}$.
	By definition of the subdifferential, $V \supset \partial \psi(x)$. Likewise, in view of (b), also $V \supset \partial \phi(x)$. We need to show that $V \subset \partial \phi(x)$. Let $v \in V$. By definition, 
	\[
		\epi g 
		= \epi \psi \cap (D \times \reals) 
		\subset \{ (z, \lambda) \in \Rd \times \reals : \lambda \ge \psi(x) + \inpr{z-x,v} \}.
	\]
	The set on the right-hand side is a closed half-space in $\reals^{d+1}$. Therefore, it also contains 
	$\cl(\conv(\epi g)) = \epi \phi$.
	But that means that $\phi(z) \ge \psi(x) + \inpr{z-x,v}$ for all $z \in \Rd$. As $\psi(x) = \phi(x)$, we find that $v \in \partial \phi(x)$. 
	Therefore, $V \subset \partial \phi(x)$ as required.
	
	(d) From $\psi \le \phi$ we infer $\dom \phi \subset \dom \psi$. \mdf~{Therefore, by \eqref{eq:domains},}
	\[
		\dom \nabla \phi
		\subset \intr (\dom \phi)
		\subset \intr (\dom \psi)
		= \dom \partial \psi.
	\]
	
	(e) Let $x \in D \cap \dom \nabla \phi$ and write $\nabla \phi(x) = y$. On the one hand, by (d), $x \in \dom \partial \psi$ and thus $\partial \psi(x) \ne \varnothing$. On the other hand, by (c), $\partial \psi(x) \subset \partial \phi(x) = \{y\}$. We conclude that $\partial \psi(x) = \{y\}$, so that $x \in \dom \nabla \psi$ and $\nabla \psi(x) = y = \nabla \phi(x)$.
\end{proof}

\section{Zero-couplings and their support}
\label{section:seq_zcoup}


This section relies heavily on the work in \citet{segers2022} on the convergence of sequences of couplings $\pi_n\in\Pi(P_n,Q_n)$ where $P_n,Q_n$ are sequences of probability measures on $\Rd$, and the convergence of closed cyclically monotone sets containing their supports, i.e., $T_n\in\Fell(\RdRd)$ satisfying $\spt\pi_n \subset T_n$.
Although being formulated for probability measures, most of the results in \cite{segers2022} are in fact given with a demonstration that extends straightforwardly to $\Mzp$.
Indeed, thanks to Theorem~2.7 (Prohorov) and Theorem~2.4 (Portmanteau) for $\Mzp$ (which are quite different from the ones in the set of probability measures, see Section~\ref{section:Mzp}) in \cite{hult2006regular}, most of the arguments can be reused almost directly. In what follows, we adapt the results given in \cite{segers2022} to our context.

\begin{lemma}
\label{lm:varphi}
The maps $\varphi_1,\varphi_2$ in \cref{def:zero-marginals} are continuous as functions $\Mzp(\RdRd) \to \Mzp(\Rd)$.
\end{lemma}

\begin{proof}[Proof of Lemma~\ref{lm:varphi}]
    We only treat $\varphi_1$, as $\varphi_2$ can be treated in exactly the same way.
    It suffices to check that $\gamma_n \Mto \gamma$ in $\Mzp(\RdRd)$ implies $\varphi_1 \gamma_n \Mto \varphi_1 \gamma$ in $\Mzp(\RdRd)$ as $n\to\infty$.
    To do so, we consider the Portmanteau Theorem~2.1(iii) in \citet{lindskog2014regular} (see \cref{thm:portmanteau}) and apply it twice.
    Let $A\subset\Rdmz$ be bounded away from the origin. Then $\intr(A)$, $\intr(A)\times\Rd$ and $\cl(A)$, $\cl(A)\times\Rd$ are open and closeds set, respectively, and bounded away from the origin in $\Rd$ or $\RdRd$. 
    Since $\gamma_n \Mto \gamma$ as $n\to\infty$ in $\Mzp(\RdRd)$, we have
    \[ 
        \gamma \brbr{\intr(A)\times\Rd}
        \le \liminf_{n\to\infty} \gamma_n\brbr{\intr(A)\times\Rd}
        \le \limsup_{n\to\infty} \gamma_n\brbr{\cl(A)\times\Rd}
        \le \gamma\brbr{\cl(A)\times\Rd},
    \]
    which we can rewrite as
    \[ 
        (\varphi_1\gamma)(\intr(A)) 
        \le \liminf_{n\to\infty} (\varphi_1\gamma_n)(\intr(A))
        \le \limsup_{n\to\infty} (\varphi_1\gamma_n)(\cl(A)) 
        \le (\varphi_1\gamma)(\cl(A)). 
    \]
    It follows that $\varphi_1\gamma_n \Mto \varphi_1\gamma$ as $n\to\infty$, as required.
\end{proof}

For our purpose, the main consequence of \cref{lem:sptlsc} below is that if a sequence $(\gamma_n,T_n)\in\zcoup(\mu_n,\nu_n)\times\Fell(\RdRd)$ satisfying $\spt\gamma_n\subset T_n$ converges in the product topology to some $(\gamma,T)$, then the limit pair satisfies $\spt\gamma\subset T$.
It is adapted from Lemma~D.11 (Lower semi-continuity of the support map) in \cite{segers2022}. The $\liminf$ of a sequence $F_n$ in $\Fell(\E)$ is defined as the set of points $x \in \E$ such that there exist $x_n \in F_n$ such that $x_n \to x$ in $\E$; the set $\liminf_n F_n$ is necessarily closed in $\E$ too.
We recall that a map $f$ from a topological space $\mathbb D$ into $\Fell(\E)$ is lower semi-continuous if $f^{-1}(\Fell_G)$ is open for any open set $G$ in $\E$, where $\Fell_G = \cbr{F \in \Fell(\E) : F \cap G \neq \emptyset}$; which, in the case where the space $\mathbb D$ is metrizable with metric $\rho$ and $\E$ is a locally compact Hausdorff topological space (LCHS), is equivalent to $f(d) \subset \liminf_n f(d_n)$ whenever $d_n\to d$ in $\mathbb D$ \citep[][lemma D.8]{segers2022}.

\begin{lemma}[Lower semi-continuity of the support map]
		\label{lem:sptlsc}
		Let $\E$ be a non-empty open subset of $\rset^k$ for some $k\ge1$.
		If $\mu_n \Mto \mu$ in $\Mzp(\E)$---seen as a subset of $\Mzp(\Rd)$---as $n \to \infty$, then $\spt \mu \subset \liminf_{n \to \infty} \spt \mu_n$, that is, the map 
		\[ 
		\spt : \Mzp(\E) \to \Fell(\E) : \mu \mapsto \spt \mu 
		\]
		is lower semi-continuous. In particular, $\left\{ (\mu, F) \in \Mzp(\E) \times \Fell(\E) : \spt \mu \subset F \right\}$ is closed in the product topology.
	\end{lemma}

\begin{proof}[Proof of Lemma~\ref{lem:sptlsc}]
    The proof is adapted from the one of Lemma~D.11 in \cite{segers2022}, with minor changes to deal with the $\Mzp$ space.
	Let $G \subset \E$ be open and suppose that $\spt\mu \cap G \neq \emptyset$. By Lemma~D.8(ii) in \cite{segers2022} we need to show that $\spt(\mu_n) \cap G \neq \emptyset$ for all large $n$. For $\eps>0$ small enough, the set $G_\eps = G \sauf \cl \ball_{0,\eps}$ is a non-empty open set bounded away from $0$.
    Since $\mu\in\Mzp(\E)\subset\Mzp(\Rd)$ we have $\spt\mu\neq\{0\}$ as $\spt\mu$ is by definition equal to $\spt\tilde\mu$ and $\tilde\mu(\{0\})=0$, so, decreasing $\eps$ if needed, we can also suppose $G_\eps \cap \spt\mu\neq\emptyset$.
    It follows from the Portmanteau lemma for $\Mzp$-convergence (\cref{thm:portmanteau}, see Theorem~2.1 in \cite{lindskog2014regular}) that 
    \[ 
        \liminf_{n \to \infty} \mu_n(G) 
        \ge \liminf_{n \to \infty} \mu_n(G_\eps) 
        \ge \mu(G_\eps)>0. 
    \] 
    Since $\spt(\mu_{(n)}) \cap G \neq \emptyset$ if and only if $\mu_{(n)}(G) > 0$ (since $G$ is open), the first part of the Lemma is proved.
	The set $\{(\mu, F) : \spt \mu \subset F\}$ is closed in view of Lemma D.8(iv) in \cite{segers2022} since $\E$ is a LCHS space.
\end{proof}

The following result is adapted from Lemma~4.1 in \cite{segers2022} and will be useful when dealing with converging sequences. Let $\Mzpcm(\RdRd)$ denote the subset of $\Mzp(\RdRd)$ of measures whose support is cyclically monotone in the sense of \cref{def:monotonicity}. Recall that $\Fellcm$ denotes the collection of closed subsets of $\RdRd$ that are cyclically monotone.

\begin{lemma}[Closure and compactness properties of sets of zero-coupling measures]
    \label{segers2022:lemma4.1} \mbox{}

\begin{enumerate}[label=(\alph*)]
\item For $\mu,\nu\in\Mzp(\Rd)$, the sets $\Mzpcm(\RdRd)$ and $\zcmcoup(\mu,\nu)$ are closed in $\Mzp(\RdRd)$.

\item The sets 
\begin{align*}
	&\left\{ (\mu,\nu,\gamma ) \mid \mu,\nu\in\Mzp(\Rd) ,\gamma \in \zcoup(\mu,\nu) \right\} \qquad \text{and} \\
	&\left\{ (\mu,\nu,\gamma ) \mid \mu,\nu\in\Mzp(\Rd),\gamma \in \zcmcoup(\mu,\nu) \right\}
\end{align*}
are closed in $\Mzp(\Rd)\times \Mzp(\Rd)\times \Mzp(\RdRd)$.

\item If $K,L\subset\Mzp(\Rd)$ are compact, then 
\begin{equation}
\label{eq:KLzcoup}
	\left\{ 
		(\mu,\nu,\gamma ) \mid
		\mu\in K, \nu\in L, \gamma \in \zcoup(\mu,\nu) 
	\right\}
\end{equation}
is compact too.
\end{enumerate}
\end{lemma}

\begin{proof}[Proof of Lemma~\ref{segers2022:lemma4.1}]
We deal with the different assertions separately.
\smallskip

\noindent\emph{Proof of (a).}
As $\zcmcoup(\mu,\nu)=\zcoup(\mu,\nu)\cap\Mzpcm(\RdRd)$, it suffices to show that $\zcoup(\mu,\nu)$ and $\Mzpcm(\RdRd)$ are closed.

According to \cref{lm:varphi}, the maps $\varphi_1,\varphi_2 : \Mz(\RdRd) \to \Mz(\Rd)$ sending a measure in $\Mz(\RdRd)$ to its zero-marginals are continuous. As a consequence, $\zcoup(\mu,\nu) = \varphi_1^{-1}(\cbr{\mu}) \cap \varphi_2^{-1}(\cbr{\nu})$ is closed.

Next, we have $\Mzpcm = \spt^{-1}(\Fellcm)$ where the map
$\spt : \Mzp(\RdRd) \to \Fell(\RdRd)$ sends a measure to its
support.  The set $\spt^{-1}(\Fellcm)$ is closed by lower semi-continuity of $\spt$
(Lemma~\ref{lem:sptlsc}) and the fact that the complement of $\Fellcm$
can be written as a union over finite intersections of sets of the
form $\Fell_{G}=\{ F\in\Fell(\RdRd) \mid F\cap G \neq \emptyset\}$ for open $G \subset \Rd \times \Rd$ \cite[Lemma 3.2]{segers2022}.
\smallskip

\noindent\emph{Proof of (b).}
If $\gamma_n \in \zcoup(\mu_n, \nu_n)$ for all $n$ and if $\mu_n \Mto \mu$, $\nu_n \Mto \nu$ and $\gamma_n \Mto \gamma$, then also $\gamma \in \zcoup(\mu, \nu)$ using arguments similar to those used for (a) and relying on the continuity of $\varphi_i$. It follows that $\cbr{ (\mu,\nu,\gamma )\mid \mu,\nu\in\Mzp ,\gamma \in \zcoup(\mu,\nu) }$ is closed in $\Mzp(\Rd)\times \Mzp(\Rd)\times \Mzp(\RdRd)$. The intersection of the former set with the closed set $\Mzp(\Rd)\times \Mzp(\Rd)\times \Mzpcm(\RdRd)$ is the set
\[ 
    \cbr{ 
        (\mu,\nu,\gamma ) \mid 
        \mu,\nu\in\Mzp ,\gamma \in \zcmcoup(\mu,\nu) 
    },
\] 
which is thus closed as well.
\smallskip

\noindent\emph{Proof of (c).}
Let $M=\bigcup_{\mu\in K, \nu\in L} \zcoup(\mu,\nu)$.
The set 
\[
	\cbr{ (\mu,\nu,\gamma )\mid \mu\in K, \nu\in L, \gamma \in \zcoup(\mu,\nu) } 
	= \cbr{ (\mu,\nu,\gamma )\mid \gamma \in \zcoup(\mu,\nu) } \cap \brbr{ K\times L\times \Mzp(\RdRd) }
\]
is closed by (b) and contained in $K\times L\times M$. Since $K,L$ are compact subsets of $\Mzp(\Rd)$, it suffices to prove that $M$ is relatively compact in $\Mzp(\RdRd)$ to conclude that $K\times L\times M$ is relatively compact, yielding the desired compactness of the set of interest.

Let $\epsilon >0$. Since $K$ and $L$ are compact in $\Mzp(\Rd)$, the Prohorov-like characterization of relative compactness in $\Mzp(\Rd)$ stated in \cref{thm:prohorov_Mzp} \citep[Theorem~2.7]{hult2006regular} yields the existence of two positive sequences $(r_{i}^{K})_{,i\ge1}$ and $(r_{i}^{L})_{i\ge1}$ decreasing to zero, so that for each $i\ge1$ there exist compact sets $C_i^K \subset \Rd\sauf \ball(0,r_{i}^{K})$ and $C_i^L \subset \Rd\sauf \ball(0,r_{i}^{L})$ such that 
\begin{align*}
    \sup_{\mu\in K} \mu \brbr{\Rd \sauf \ball(0,r_{i}^K)}
    &< \infty, &
    \sup_{\mu \in K} \mu\brbr{\Rd \sauf (\ball(0,r_{i}^K) \cup C_i^K )}
    &\le \epsilon / 2, \\
    \sup_{\nu\in L} \nu \brbr{\Rd \sauf \ball(0,r_{i}^L)}
    &< \infty, &
    \sup_{\nu \in L} \nu\brbr{\Rd \sauf (\ball(0,r_{i}^L) \cup C_i^L )} 
    &\le \epsilon / 2.
\end{align*}
It is easy to see that we can choose common sequences of numbers and sets by setting
\begin{align*}
    r_i &= \max(r_i^K, r_i^L), &
    C_i &= (C_i^K \cup C_i^L) \sauf \ball(0, r_i),
\end{align*}
to satisfy the above inequalities for $\mu$ and $\nu$ simultaneously.


For the Euclidean norm, $\nbr{(x,y)} \ge \sqrt{2} r$ implies $\max(\nbr{x}, \nbr{y}) \ge r$, where $(x,y) \in \RdRd$ and $r > 0$. As a consequence, we have, indicating the dimensions of the balls by a subscript,
\[
    \reals^{2d} \sauf \ball_{2d}(0, \sqrt{2}r_i)
    \subset 
    \brbr{(\Rd \sauf \ball_d(0,r_i)) \times \Rd} \cup \brbr{\Rd \times (\Rd \sauf \ball_d(0,r_i))},
\]
and thus, for every $i \ge 1$,
\[
    \sup_{\gamma \in M} \gamma \brbr{\reals^{2d} \sauf \ball_{2d}(0, \sqrt{2}r_i)}
    \le 
    \sup_{\mu \in K} \mu\brbr{\Rd \sauf \ball_d(0,r_i)} + 
    \sup_{\nu \in L} \nu\brbr{\Rd \sauf \ball_d(0,r_i)}
    < \infty.
\]
Finally, define the compact set
\[
    K_i = \brbr{C_i \cup \cl \ball_d(0,r_i)}^2 \sauf \ball_{2d}(0, \sqrt{2}r_i),
\]
which is obviously contained in $\reals^{2d} \sauf \ball_{2d}(0, \sqrt{2}r_i)$.
For all $(x, y) \in \reals^{2d}$, we have
\begin{align*}
    \lefteqn{
    (x,y) \notin K_i \cup \ball_{2d}(0, \sqrt{2}r_i)
    } \\
    &\implies (x,y) \notin \brbr{C_i \cup \cl \ball_d(0,r_i)}^2 \\
    &\implies x \notin C_i \cup \cl \ball_d(0,r_i) \text{ or } y \notin C_i \cup \cl \ball_d(0,r_i) \\
    &\implies x \notin C_i \cup \ball_d(0,r_i) \text{ or } y \notin C_i \cup \ball_d(0,r_i),
\end{align*}
whence, in $\reals^{2d}$,
\[
    [K_i \cup \ball_{2d}(0, \sqrt{2}r_i)]^c
    \subset 
    [(C_i \cup \ball_d(0,r_i))^c \times \Rd] \cup
    [\Rd \times (C_i \cup \ball_d(0,r_i))^c].
\]
It follows that
\begin{align*}
    \sup_{\gamma \in M} \gamma\brbr{[K_i \cup \ball_{2d}(0, \sqrt{2}r_i)]^c}
    &\le \sup_{\mu \in K} \mu\brbr{[C_i \cup \ball_d(0,r_i)]^c}
    + \sup_{\nu \in L} \nu\brbr{[C_i \cup \ball_d(0,r_i)]^c} \\
    &\le \eps/2 + \eps/2 = \eps.
\end{align*}
Since $\eps > 0$ was arbitrary, we have shown that the set $M$ satisfies the two criteria in \cref{thm:prohorov_Mzp}, confirming that $M$ is relatively compact, as required.
\end{proof}

The following lemma is similar to Lemma~C.1 in \cite{segers2022} which concerns probability measures. Recall that $\varphi_1\gamma$ and $\varphi_2\gamma$ are the zero-marginals of $\gamma\in\Mzp(\Rd)$ while $\phi_1\gamma$ and $\phi_2\gamma$ are its ordinary marginals  (\cref{def:zero-coupling} and \cref{rkdef:couplings}). Let $\proj_1 : \RdRd \to \Rd$ be the projection mapping $(x,y) \mapsto x$. Seeing a subset $T$ of $\RdRd$ as (the graph of) a multivalued mapping from $\Rd$ into $\Rd$, we also write $\dom T = \proj_1(T) = \cbr{x \in \Rd \mid \exists y \in \Rd : (x,y) \in T}$.


\begin{lemma}[Support of a zero-margin]
	\label{lem:supp-zero-proj}
	Let $\gamma \in \Mzp(\RdRd)$ and let $\mu = \varphi_1\gamma\in\Mzp(\Rd)$ denote its left zero-marginal. Then $\spt\mu \subset\spt\phi_1\gamma= \cl (\proj_1 (\supp \gamma))$. 
	As a consequence, if $\supp \gamma \subset T$ for some $T \subset \RdRd$, then $\spt \mu \subset \cl(\dom T)$.
	In particular, if $T$ is maximal monotone, then $\intr(\supp \mu) \subset \dom T$.
\end{lemma}

\begin{proof}[Proof of Lemma~\ref{lem:supp-zero-proj}]
    We consider the extension $\tilde\gamma$ of $\gamma$ to a positive Borel measure on $\RdRd$ and write $\bar\mu=\phi_1\tilde\gamma=\phi_1\gamma$ for its left marginal. Since $\mu$ is the restriction of $\bar\mu$ to $\Rdmz$, we have $\spt \mu \subset \spt \bar\mu$.


    We show that $\spt\bar\mu$ is equal to $F = \cl( \proj_1 (\supp \tilde\gamma) )$. First, to show that $\spt \bar\mu \supset F$, it suffices to show that $\spt \bar\mu \supset \proj_1(\spt \tilde\gamma)$, as $\spt \bar\mu$ is closed. But the latter inclusion follows easily: if $x \in \proj_1(\spt \tilde\gamma)$, then there exists $y$ such that $(x, y) \in \spt \tilde\gamma$; for any open neighbourhood $U$ of $x$, the set $U \times \Rd$ is an open neighbourhood of $(x, y)$, and thus $\bar\mu(U) = \tilde\gamma(U \times \Rd) > 0$, which, as $U$ was arbitrary, implies $x \in \spt \bar\mu$.	
    	
    Second, to show that $\spt \bar\mu \subset F$, we need to show that $\bar\mu(F^c) = 0$, as $F$ is closed. We have $\bar\mu(F^c) = \tilde\gamma(F^c \times \Rd)$ and
    \[
    F^c \times \Rd
    \subset (\proj_1 (\supp \tilde\gamma))^c \times \Rd
    \subset (\supp \tilde\gamma)^c,
    \]
    which is a $\tilde\gamma$-null set.




If $\spt \gamma \subset T$ for some $T \subset \RdRd$, then $\proj_1(\spt \gamma) \subset \dom T$ and thus $\spt \bar\mu \subset \cl(\dom T)$.


Finally, assume $T$ contains $\spt \gamma$ and is maximal monotone. By Theorem~12.41 (near convexity of domains and ranges) in \cite{rockafellarwets98}, there exists a convex subset $C$ of $\Rd$ such that $C \subset \dom T \subset \cl C$. 
As $\spt \mu \subset \cl(\dom T)$, we have $\spt \mu \subset \cl C$, hence $\intr(\spt\mu) \subset \intr\cl C$. Moreover, the convexity of $C$ yields $\ri C = \ri \cl C$, where $\ri$ is the relative interior of $C$ with respect to the smallest affine subspace that contains $C$. As a consequence, we can write
\begin{align*}
	\intr(\spt\mu)
	\subset \intr\cl C 
	&\subset \ri \cl C \\
	&= \ri C 
	\subset C 
	\subset \dom T,
\end{align*} 
and the statement is proved.
\end{proof}

\begin{remark}
\label{rk:zero_margin_spt_incl}
Let $\mu,\nu\in\Mzp(\Rd)$ and $\gamma\in\zcoup(\mu,\nu)$, and write $\bar\mu = \phi_1\gamma$ for its left margin (\cref{rkdef:couplings}). Since $\mu = \res \bar\mu$ we have the inclusions
$\spt \mu \subset \spt \bar\mu \subset \spt \mu \cup \{0\}$.
It follows by \cref{lem:supp-zero-proj} that $\spt \mu \sauf \cbr{0} = \cl (\proj_1 (\spt \gamma)) \sauf \cbr{0}$.
%
\end{remark}

Recall that we write $T\llcorner V = T \cap (V \times \Rd)$ for $T \subset \RdRd$ and $V \subset \Rd$. The following lemma will be useful to prove the graphical convergence, relative to the interior $V$ of the support of $\mu$, of a sequence of maximal cyclically monotone sets $T_n\in\Fell(\RdRd)$, for $n\in\nat$, containing the support of $\gamma_n$ by allowing one to prove that the limit points of $T_n \restr V$ are all the same. It is adapted from Lemma~4.2 in \cite{segers2022}. For $T \subset \RdRd$, we write $T^{-1} = \cbr{(y,x) \in \RdRd \mid (x,y) \in T}$, $\dom T = \proj_1(T)$ and $\rge = \proj_2(T)$.

\begin{lemma}[Uniqueness of maximal monotone extension on the interior of the support]
	\label{lem:STV}	
	Let $\gamma \in \zcoup(\mu,\nu)$ for $\mu, \nu \in \Mzp(\Rd)$ and let $S, T \in \Fellmm(\Rd \times \Rd)$ satisfy $\spt \gamma \subset S \cap T$. Put  $V = \intr(\spt \mu)$ and $W = \intr(\spt \nu)$. Then $V \subset \dom S \cap \dom T$ and $S \restr V = T \restr V$, as well as $W \subset \rge S \cap \rge T$ and $S^{-1} \restr W = T^{-1} \restr W$.
\end{lemma}

\begin{proof}[Proof of Lemma~\ref{lem:STV}]
    Let $\tilde{\gamma}$ be the extension of $\gamma$ to $\RdRd$ defined by $\tilde{\gamma}(B) = \gamma(B \sauf \cbr{(0,0)})$ for $B \in \borel(\RdRd)$.
Let $\hat{\mu}=\phi_1\tilde{\gamma}$ and $\hat{\nu} = \phi_2\tilde{\gamma}$ be its left- and right-marginals; by definition, $\tilde\gamma$ lies in $\coup(\hat\mu,\hat\nu)$.
According to \cref{lem:supp-zero-proj} and Lemma~C.1 in \citet{segers2022},
it is immediate that both $\hat{V} = \intr\spt\hat{\mu}$ and $V$ are contained in both $\dom S$ and $\dom T$, since $\spt\gamma=\spt\tilde\gamma$. The rest of the proof of Lemma~4.2 in \cite{segers2022} can then be reused directly since it exclusively relies on properties of maximal monotone maps and not on the finiteness of the measures considered there. It follows that $S(x) = T(x)$ for all $x \in \hat{V}$.
Since $\spt\mu\subset\spt\hat{\mu}$, we have $S(x) = T(x)$ for all $x \in V$ too.
	
The statements about $W$ follow by switching the roles of $\mu$ and $\nu$, upon noting that $\dom S^{-1} = \rge S$.
\end{proof}

\begin{lemma}
	\label{lem:Psi-finite}
	Let $\mu, \nu \in \Mz(\Rd)$ with $\mu(\Rdmz) = \infty$ or $\nu(\Rdmz) = \infty$ (or both), and let $\gamma \in \zcoup(\mu, \nu)$ be supported by $\partial\psi$ for some closed convex function $\psi$. Then $(0,0) \in \spt\gamma \subset \partial\psi$ and $\psi(0) = \inf_{x \in \Rd} \psi(x)$.
\end{lemma}

\begin{proof}
    The zero-coupling measure $\gamma$ must have infinite mass: if, for instance, $\mu(\Rdmz) = \infty$, then also $\gamma(\Rdmz \times \Rd) = \mu(\Rdmz) = \infty$. By \cref{lm:zcoup_in_Mzp}, $\gamma$ belongs to $\Mz(\RdRd)$. Therefore, necessarily $\gamma(\ball_d(0,\eps) \times \ball_d(0,\eps)) = \infty > 0$ for every $\eps > 0$, which means that $\spt \gamma$ meets $\ball_d(0,\eps) \times \ball_d(0,\eps)$ for every $\eps > 0$. As a consequence, $(0, 0) \in \spt \gamma \subset \partial \psi$, the latter inclusion holding by assumption. But then also $0 \in \dom \partial \psi \subset \dom \psi$ in view of \cref{eq:domains}. By definition of the subdifferential we find, as required,
	\[
		\forall x \in \Rd, \qquad
		\psi(x) \ge \psi(0) + \inpr{x - 0, 0} = \psi(0).
        \qedhere
	\]
\end{proof}

\end{document}